\documentclass[a4paper,12pt]{article}

\usepackage{amsmath,amssymb,amsthm,amscd}

\newcommand{\Fg}{\mathfrak{g}}
\newcommand{\Fh}{\mathfrak{h}}
\newcommand{\Fsl}{\mathfrak{sl}}
\newcommand{\Fsp}{\mathfrak{sp}}

\newcommand{\CB}{\mathcal{B}}
\newcommand{\CG}{\mathcal{G}}
\newcommand{\CR}{\mathcal{R}}

\newcommand{\edge}{\mathcal{E}}
\newcommand{\ggms}{\mathcal{V}}
\newcommand{\pwp}{\mathcal{P}}
\newcommand{\MV}{\text{\rm MV}}
\newcommand{\mv}{\mathcal{MV}}

\newcommand{\BC}{\mathbb{C}}
\newcommand{\BR}{\mathbb{R}}
\newcommand{\BZ}{\mathbb{Z}}

\newcommand{\GL}{\mathop{\rm GL}\nolimits}
\newcommand{\Hom}{\mathop{\rm Hom}\nolimits}
\newcommand{\wt}{\mathop{\rm wt}\nolimits}
\newcommand{\Aut}{\mathop{\rm Aut}\nolimits}

\newcommand{\pair}[2]{\langle #1,\,#2 \rangle}

\newcommand{\ve}{\varepsilon}
\newcommand{\vp}{\varphi}

\newcommand{\bzero}{{\bf 0}}
\newcommand{\bi}{{\bf i}}
\newcommand{\bj}{{\bf j}}

\newcommand{\ti}[1]{\widetilde{#1}}
\newcommand{\ha}[1]{\widehat{#1}}

\makeatletter
\@addtoreset{equation}{subsection}

\renewcommand\section{\@startsection{section}{1}{0pt}
{-3.5ex plus -1ex minus -.2ex}{1.0ex plus .2ex}{\large\bf}}
\renewcommand\subsection{\@startsection{subsection}{1}{0pt}
{2.5ex plus 1ex minus .2ex}{-1em}{\bf}}
\makeatother

\newcommand{\svsp}{\vspace{1.5mm}}
\newcommand{\vsp}{\vspace{3mm}}

\theoremstyle{plain}
\newtheorem{thm}{Theorem}[subsection]
\newtheorem{lem}[thm]{Lemma}
\newtheorem{prop}[thm]{Proposition}
\newtheorem{cor}[thm]{Corollary}

\newtheorem{claim}{Claim}[thm]

\theoremstyle{definition}
\newtheorem{dfn}[thm]{Definition}

\theoremstyle{remark}
\newtheorem{rem}[thm]{Remark}

\begin{document}

\setlength{\baselineskip}{20pt}

\title{\Large\bf
A modification of the Anderson-Mirkovi\'c conjecture \\
for Mirkovi\'c-Vilonen polytopes in types $B$ and $C$}
\author{
 Satoshi Naito \\ 
 \small Institute of Mathematics, University of Tsukuba, \\
 \small Tsukuba, Ibaraki 305-8571, Japan \ 
 (e-mail: {\tt naito@math.tsukuba.ac.jp})
 \\[2mm] and \\[2mm]
 Daisuke Sagaki \\ 
 \small Institute of Mathematics, University of Tsukuba, \\
 \small Tsukuba, Ibaraki 305-8571, Japan \ 
 (e-mail: {\tt sagaki@math.tsukuba.ac.jp})
}
\date{}
\maketitle

%=======================%
%     START ABSTRACT    %
%=======================%
%
\begin{abstract} \setlength{\baselineskip}{16pt}
We give an explicit description of the (lowering) 
Kashiwara operators on Mirkovi\'c-Vilonen polytopes in 
types $B$ and $C$, which provides a simple method 
for generating Mirkovi\'c-Vilonen polytopes inductively. 
This description can be thought of as a modification of 
the original Anderson-Mirkovi\'c conjecture, 
which Kamnitzer proved in the case of type $A$, and 
presented a counterexample in the case of type $C_{3}$.
\end{abstract}
%
%=========================%
%     START SECTION 01    %
%=========================%
%
\section{Introduction.}
\label{sec:intro}
%
%
%%%%%
Let $G$ be a connected, simply-connected, 
semisimple algebraic group over $\BC$, and 
$G^{\vee}$ its Langlands dual group.
Mirkovi\'c and Vilonen (\cite{MV1}, \cite{MV2}) 
discovered a family of closed, irreducible, algebraic subvarieties, 
called MV cycles, of the affine Grassmannian 
$\CG$ associated to $G$, 
which provide a basis for each finite-dimensional 
irreducible highest weight representation of $G^{\vee}$ 
(or equivalently, of its Lie algebra $\Fg^{\vee}$).

In order to obtain an explicit combinatorial description of MV cycles, 
Anderson (\cite{A}) defined MV polytopes 
for the Lie algebra $\Fg$ of $G$ to 
be moment map images of these cycles, 
which are drawn in the real form 
$\Fh_{\BR} := \sum_{j \in I} \BR h_{j}$ 
of the Cartan subalgebra $\Fh$ of $\Fg$, 
where the $h_{j}$, $j \in I$, are the simple coroots of $\Fg$; 
in \cite{Kam1}, Kamnitzer characterized 
these MV polytopes as pseudo-Weyl polytopes that satisfy 
``tropical'' Pl\"ucker relations. 
Furthermore, inspired by the crystal structure 
on the set of MV cycles due to Braverman, 
Finkelberg, and Gaitsgory (\cite{BG}, \cite{BFG}), 
Anderson and Mirkovi\'c proposed a conjecture 
(the AM conjecture) describing a crystal structure 
for $\Fg^{\vee}$ on the set of MV polytopes; 
this conjecture gives a method for generating 
MV polytopes inductively without making use of 
the tropical Pl\"ucker relations.
The AM conjecture above was proved in the case 
$\Fg = \Fsl_{n}$ by Kamnitzer (\cite{Kam2}), 
who also presented a counterexample 
in the case $\Fg = \Fsp_{6}$.

The purpose of this paper is 
to prove a kind of modification of 
the original AM conjecture for simple Lie algebras of types $B$ and $C$.
Let us explain our results more precisely. 
In this paper, we assume that $\Fg$ is 
a simple Lie algebra of type $A$ over $\BC$. 
Let $\omega : I \rightarrow I$ be 
a (Dynkin) diagram automorphism of order $2$ of 
the index set $I = \{1,\,2,\,\dots,\,\ell\}$.
Then it induces a Lie algebra automorphism 
(also denoted by) $\omega : \Fg \rightarrow \Fg$, 
which stabilizes the Cartan subalgebra $\Fh$, 
and hence induces $\omega \in \GL(\Fh^{\ast})$ by: 
$\pair{\omega(\lambda)}{h} = \pair{\lambda}{\omega(h)}$ 
for $\lambda \in \Fh^{\ast}$ and $h \in \Fh$.
We set $\Fg^{\omega} := \{x \in \Fg \mid \omega(x) = x\}$ 
and $\Fh^{\omega} := \{h \in \Fh \mid \omega(h) = h\}$.
It is known that if $\Fg$ is of type $A_{\ell}$ with 
$\ell=2n-1$, $n \in \BZ_{\ge 2}$, 
(resp., of type $A_{\ell}$ with $\ell=2n$, $n \in \BZ_{\ge 2}$,)
then $\Fg^{\omega}$ is a simple Lie algebra of type $C_{n}$ 
(resp., type $B_{n}$) with Cartan subalgebra $\Fh^{\omega}$.
Moreover, the Weyl group $\ha{W}$ of $\Fg^{\omega}$ 
can be identified with the subgroup $W^{\omega}$ of 
the Weyl group $W = \langle s_{i} \mid i \in I \rangle$ 
(through a group isomorphism $\Theta : \ha{W} \to W^{\omega}$) 
consisting of the elements of $W$ fixed under the action of 
the diagram automorphism $\omega : I \rightarrow I$ given by: 
$\omega(s_{i}) = s_{\omega(i)}$ for $i \in I$.

Following Kamnitzer, 
let $\mv$ denote the set of MV polytopes 
$P = P(\mu_{\bullet}) \subset \Fh_{\BR}$, with GGMS datum 
$\mu_{\bullet} = (\mu_{w})_{w \in W}$, such that 
$\mu_{w_{0}} = 0 \in \Fh_{\BR}$, 
where $w_{0} \in W$ is the longest element. 
Here the GGMS datum $\mu_{\bullet} = (\mu_{w})_{w \in W}$ 
of an MV polytope $P$ is a collection 
(which may have repetition) of elements of 
$\Fh_{\BZ}:=\sum_{j \in I} \BZ h_{j}$, and gives 
the set of vertices of the convex polytope $P$. 
Let $P = P(\mu_{\bullet}) \in \mv$ be an MV polytope 
with GGMS datum $\mu_{\bullet} = (\mu_{w})_{w \in W}$.
Then, the image $\omega(P)$ of $P$ (as a set) 
under $\omega \in \GL(\Fh)$ is 
identical to the element $P(\mu_{\bullet}') \in \mv$ 
with GGMS datum $\mu_{\bullet}' = (\mu_{w}')_{w \in W}$, 
where $\mu_{w}' := \omega(\mu_{\omega(w)})$ for $w \in W$.
We set $\mv^{\omega} := \bigl\{P \in \mv \mid \omega(P) = P\bigr\}$, 
and define the set $\ha{\mv}$ of MV polytopes 
for $\Fg^{\omega}$ in the same manner as we defined $\mv$ for $\Fg$.
Now, to each element $P = P(\mu_{\bullet})$ of 
$\mv^{\omega}$ with GGMS datum 
$\mu_{\bullet} = (\mu_{w})_{w \in W}$, 
we assign a convex polytope $\Phi(P) = P \cap \Fh^{\omega}$ 
in $\Fh^{\omega} \cap \Fh_{\BR}$, which turns out to be 
the element $\ha{P}(\ha{\mu}_{\bullet})$ of $\ha{\mv}$ 
with GGMS datum 
$\ha{\mu}_{\bullet} = (\ha{\mu}_{\ha{w}})_{\ha{w} \in \ha{W}}$, 
where $\ha{\mu}_{\ha{w}} = \mu_{\Theta(\ha{w})} 
\in \Fh^{\omega} \cap \Fh_{\BZ}$ for $\ha{w} \in \ha{W}$.

One of our main results (Theorem~\ref{thm:fixed-mv})
of this paper asserts that 
the map $\Phi : \mv^{\omega} \to \ha{\mv}$ 
defined above is a bijection. 
This result can be thought of as 
an application of the general idea of realizing crystals 
for a non-simply-laced Kac-Moody algebra 
as the fixed point subsets under a diagram automorphism 
of those for a simply-laced Kac-Moody algebra.
Such an idea has often been used since Lusztig's 
pioneering work (\cite[Chapter~14]{L1}); cf., to 
name a few, \cite{Xu}, \cite{NS1}, \cite{NS2}, \cite{S}, 
and also \cite{KLP}.

Using the result above, 
we prove that for each $1 \le j \le n$, 
the (lowering) Kashiwara operator $\ha{f}_{j}$ on $\ha{\mv}$ 
for the ``LBZ'' crystal structure due to Lusztig and 
Berenstein-Zelevinsky (\cite{BZ2}, \cite{Kam2}) is realized 
(through the bijection $\Phi : \mv^{\omega} \rightarrow \ha{\mv}$) 
as the restriction to $\mv^{\omega} \subset \mv$ 
of a certain composition $f_{j}^{\omega}$ of 
the Kashiwara operators $f_{j}$ and $f_{\omega(j)}$ on $\mv$ 
for the LBZ crystal structure.
Moreover, from the original AM conjecture 
(proved by Kamnitzer) applied to 
MV polytopes in $\mv^{\omega}$, 
we obtain a description 
(Theorems~\ref{thm:pd01}, \ref{thm:pd02}, and \ref{thm:pd03}), 
in terms of GGMS data, of the (lowering) Kashiwara operators 
$\ha{f}_{j}$, $1 \le j \le n$, on MV polytopes in $\ha{\mv}$.
Here we should mention that our description of 
the (lowering) Kashiwara operators on MV polytopes 
for $\Fg^{\omega}$ in types $B$ and $C$ is 
rather analogous to the one in the original AM conjecture, 
and can be thought of as a kind of modification of it.  

This paper is organized as follows. 
In subsections~\ref{subsec:MV} and \ref{subsec:trans}, 
following Kamnitzer, we recall the definition and some basic properties 
of MV polytopes, and also the LBZ crystal structure on them.
Next, in subsections~\ref{subsec:da} and \ref{subsec:damv}, 
we introduce a natural action of the diagram automorphism 
$\omega$ on MV polytopes in type $A$, and then study 
the set of MV polytopes fixed by this action. 
In subsection~\ref{subsec:phi}, we state one of our main results 
(Theorem~\ref{thm:fixed-mv}), the proof of which occupies 
subsections~\ref{subsec:phi-bz} and \ref{subsec:fixed-mv}.
By making use of this result, in section~\ref{sec:hk}, 
we present an explicit description 
(Theorems~\ref{thm:pd01}, \ref{thm:pd02}, and \ref{thm:pd03}) 
of the (lowering) Kashiwara operators on MV polytopes 
in types $B$ and $C$. 

\paragraph{Acknowledgments.}
When we gave a talk on the results of this paper at a conference 
held in June of 2007, we were informed by Toshiyuki Tanisaki that 
Jiuzu Hong (\cite{H1}, \cite{H2}) also 
obtained closely related results by 
an approach different from ours. 
We would like to express our sincere thanks to Toshiyuki Tanisaki 
for his kindness, and to Jiuzu Hong for sending us 
a rough draft (\cite{H1}) of his paper (\cite{H2}). 

%=========================%
%     START SECTION 02    %
%=========================%
%
\section{Mirkovi\'c-Vilonen polytopes and diagram automorphisms.}
\label{sec:MVda}

%==============================%
%     START SUBSECTION 0201    %
%==============================%
%
\subsection{Mirkovi\'c-Vilonen polytopes.}
\label{subsec:MV}

Let $\Fg$ be a finite-dimensional semisimple Lie algebra 
(not necessarily of type $A$) over the field $\BC$ 
of complex numbers associated to the root datum 
$\bigl(A=(a_{ij})_{i,j \in I}, \, 
 \Pi=\bigl\{\alpha_{j}\bigr\}_{j \in I}, \, 
 \Pi^{\vee}=\bigl\{h_{j}\bigr\}_{j \in I}, \, 
 \Fh^{\ast},\,\Fh
 \bigr)$, where 
$A=(a_{ij})_{i,j \in I}$ is the Cartan matrix, 
$\Fh$ is the Cartan subalgebra, 
$\Pi=\bigl\{\alpha_{j}\bigr\}_{j \in I} \subset 
 \Fh^{\ast}:=\Hom_{\BC}(\Fh,\,\BC)$ 
is the set of simple roots, and 
$\Pi^{\vee}=\bigl\{h_{j}\bigr\}_{j \in I} \subset \Fh$ 
is the set of simple coroots; note that 
$\pair{\alpha_{j}}{h_{i}}=a_{ij}$ for $i,\,j \in I$, 
where $\pair{\cdot}{\cdot}$ denotes the canonical pairing between 
$\Fh^{\ast}$ and $\Fh$. 
We denote by $x_{j},\,y_{j}$, $j \in I$, 
the Chevalley generators of $\Fg$.
Let $W=\langle s_{i} \mid i \in I \rangle$ 
be the Weyl group of $\Fg$, where $s_{i}$ is the simple reflection 
for $i \in I$, and let $e,\,w_{0} \in W$ denote the unit element and 
the longest element of $W$, respectively. 
Denote by $\Lambda_{i} \in \Fh^{\ast}$, $i \in I$, 
the fundamental weights, and set 
\begin{equation*}
\Gamma:=
\bigl\{w \cdot \Lambda_{i} \mid 
w \in W,\,i \in I\bigr\} \subset \Fh^{\ast}.
\end{equation*}
Let $\Fg^{\vee}$ be 
the (Langlands) dual Lie algebra of $\Fg$, that is, 
the finite-dimensional semisimple Lie algebra over $\BC$ 
associated to the root datum 
$\bigl({}^{t}A=(a_{ji})_{i,j \in I}, \, 
 \Pi^{\vee}=\bigl\{h_{j}\bigr\}_{j \in I}, \, 
 \Pi=\bigl\{\alpha_{j}\bigr\}_{j \in I}, \, 
 \Fh,\,\Fh^{\ast}
 \bigr)$; 
note that the Cartan subalgebra of 
$\Fg^{\vee}$ is not $\Fh$, but $\Fh^{\ast}$. 

We recall from \cite{Kam1} the definitions and 
some basic properties of pseudo-Weyl polytopes 
and Mirkovi\'c-Vilonen (MV for short) polytopes. 
Set $\Fh_{\BZ}:=\bigoplus_{j \in I} \BZ h_{j}$, and 
$\Fh_{\BR}:=\bigoplus_{j \in I} \BR h_{j}$. 
For each $w \in W$, 
we define a partial ordering $\ge_{w}$ on $\Fh_{\BR}$ by: 
$h \ge_{w} h'$ if $w^{-1} \cdot h-w^{-1} \cdot h' 
\in \sum_{j \in I} \BR_{\ge 0} h_{j}$. 
Denote by $\ggms$ the set of 
collections $\mu_{\bullet}=(\mu_{w})_{w \in W}$ of 
elements in $\Fh_{\BR}$ such that 
$\mu_{w'} \ge_{w} \mu_{w}$ for all $w,\,w' \in W$ and 
$\mu_{w_{0}}=0$. 
Note that 
if $\mu_{\bullet}=(\mu_{w})_{w \in W} \in \ggms$, then 
$\mu_{w} \in \sum_{j \in I} \BR_{\le 0} h_{j}$ 
for all $w \in W$, since $\mu_{w} \ge_{w_{0}} \mu_{w_{0}}=0$ 
implies $w_{0}^{-1} \cdot \mu_{w} \in \sum_{j \in I} \BR_{\ge 0} h_{j}$ 
and hence $\mu_{w} \in \sum_{j \in I} \BR_{\ge 0}\,w_{0} \cdot h_{j}$. 
To each $\mu_{\bullet}=(\mu_{w})_{w \in W} \in \ggms$, 
we associate a (convex) polytope $P(\mu_{\bullet}) \subset \Fh_{\BR}$ by:
%
%%%%%%%%%%%%%%%
%%% eq:po02 %%%
%%%%%%%%%%%%%%%
%
\begin{equation} \label{eq:po02}
P(\mu_{\bullet})=
 \bigl\{ 
 h \in \Fh_{\BR} \mid 
 h \ge_{w} \mu_{w} \ 
 \text{for all $w \in W$}
 \bigr\}, 
\end{equation}
and call it the pseudo-Weyl polytope with 
Gelfand-Goresky-MacPherson-Serganova 
(GGMS for short) datum 
$\mu_{\bullet}=(\mu_{w})_{w \in W}$. 
It is easy to see that 
if $\mu_{\bullet}=(\mu_{w})_{w \in W} \in \ggms,\,
\mu_{\bullet}'=(\mu_{w}')_{w \in W} \in \ggms$ and 
$P(\mu_{\bullet})=P(\mu_{\bullet}')$, 
then $\mu_{\bullet}=\mu_{\bullet}'$, i.e., 
$\mu_{w}=\mu_{w}'$ for all $w \in W$. 

Let $\mu_{\bullet}=(\mu_{w})_{w \in W} \in \ggms$. 
For each $\gamma \in \Gamma$, we define $M_{\gamma} \in \BR$ by: 
%
%%%%%%%%%%%%
%%% eq:M %%%
%%%%%%%%%%%%
%
\begin{equation} \label{eq:M}
M_{\gamma}=\pair{w \cdot \Lambda_{i}}{\mu_{w}} \in \BR
\quad
\text{if $\gamma=w \cdot \Lambda_{i}$ for some $w \in W$ and 
$i \in I$};
\end{equation}
note that $\pair{w \cdot \Lambda_{i}}{\mu_{w}} \in \BR$ does not 
depend on the expression $\gamma=w \cdot \Lambda_{i}$, 
$w \in W$, $i \in I$, of $\gamma \in \Gamma$. 
Then, we have 
%
%%%%%%%%%%%%%%%
%%% eq:mu-M %%%
%%%%%%%%%%%%%%%
%
\begin{equation} \label{eq:mu-M}
\mu_{w}=\sum_{i \in I} 
 M_{w \cdot \Lambda_{i}} \, w \cdot h_{i}
\quad \text{for $w \in W$}.
\end{equation}
It follows immediately from \eqref{eq:mu-M} 
that for $w \in W$ and $i \in I$,
%
%%%%%%%%%%%%%%%%%
%%% eq:length %%%
%%%%%%%%%%%%%%%%%
%
\begin{equation} \label{eq:length}
\begin{array}{l}
\mu_{ws_{i}}-\mu_{w}=L w \cdot h_{i}, \quad \text{where} \\[3mm]
L=-M_{w \cdot \Lambda_{i}} 
    -M_{ws_{i} \cdot \Lambda_{i}} - 
    {\displaystyle\sum_{j \in I,\, j \ne i}}
    a_{ji} M_{w \cdot \Lambda_{j}},
\end{array}
\end{equation}
which we call the length formula (see \cite[Eq.(8)]{Kam1}). 
By using the length formula, we see easily 
that for each $w \in W$ and $i \in I$, 
the condition $\mu_{ws_{i}} \ge_{w} \mu_{w}$ is equivalent to the edge 
inequality (see \cite[Eq.(6)]{Kam1}): 
%
%%%%%%%%%%%%%%%
%%% eq:edge %%%
%%%%%%%%%%%%%%%
%
\begin{equation} \label{eq:edge}
 M_{ws_{i} \cdot \Lambda_{i}} + 
 M_{w \cdot \Lambda_{i}} + 
  \sum_{j \in I,\, j \ne i} a_{ji} M_{w \cdot \Lambda_{j}} \le 0. 
\end{equation}

\begin{rem}
Let $w \in W$. It follows by induction on $W$ 
with respect to the (weak) Bruhat ordering that 
$\mu_{ws_{i}} \ge_{w} \mu_{w}$ for all $i \in I$ 
implies $\mu_{w'} \ge_{w} \mu_{w}$ for all $w' \in W$.
\end{rem}

We denote by $\edge$ the set of 
collections $M_{\bullet}=(M_{\gamma})_{\gamma \in \Gamma}$ of 
real numbers, with $M_{w_{0} \cdot \Lambda_{i}} = 0$ for all $i \in I$, 
satisfying the edge inequality \eqref{eq:edge} 
for all $w \in W$ and $i \in I$.
Now, it is clear that by \eqref{eq:M} and \eqref{eq:mu-M}, 
the elements of $\ggms$ and those of $\edge$ are in bijective 
correspondence, which we denote by $D:\ggms \rightarrow \edge$, 
so that if $M_{\bullet}=D(\mu_{\bullet})$, then 
the pseudo-Weyl polytope $P(\mu_{\bullet})$ is 
identical to 
%
%%%%%%%%%%%%%%%
%%% eq:po01 %%%
%%%%%%%%%%%%%%%
%
\begin{equation} \label{eq:po01}
P(M_{\bullet}):=
 \bigl\{ 
 h \in \Fh_{\BR} \mid 
 \pair{\gamma}{h} \ge M_{\gamma} \ 
 \text{for all $\gamma \in \Gamma$}
 \bigr\}; 
\end{equation}
we call $M_{\bullet} \in \edge$ 
the edge datum of the pseudo-Weyl polytope 
$P(\mu_{\bullet})=P(M_{\bullet})$. We set 
$\pwp:=
  \bigl\{P(\mu_{\bullet}) \mid \mu_{\bullet} \in \ggms \bigr\}=
  \bigl\{P(M_{\bullet}) \mid M_{\bullet} \in \edge \bigr\}$.
%
%%%%%%%%%%%%%%%%%%
%%% rem:vertex %%%
%%%%%%%%%%%%%%%%%%
%
\begin{rem} \label{rem:vertex}
We know from \cite[Proposition~2.2]{Kam1} that 
the set of vertices of the pseudo-Weyl polytope 
$P(\mu_{\bullet})$ is the collection 
$\mu_{\bullet}=(\mu_{w})_{w \in W}$ (possibly, with repetitions).
In particular, $P(\mu_{\bullet})$ is identical to the convex hull 
in $\Fh_{\BR}$ of the collection $\mu_{\bullet}=(\mu_{w})_{w \in W}$. 
\end{rem}

Let $w \in W$ and $i,\,j \in I$ be such that 
$ws_{i} > w$, $ws_{j} > w$, and $i \ne j$, 
where $>$ denotes the (weak) Bruhat ordering on $W$. 
We say that an element 
$M_{\bullet}=(M_{\gamma})_{\gamma \in \Gamma} \in \edge$
satisfies the tropical Pl\"ucker relation at $(w,\,i,\,j)$ 
if $a_{ij}=a_{ji}=0$, or one of the following holds: 

\noindent
(1) $a_{ij}=a_{ji}=-1$, and 
%
%%%%%%%%%%%%%%%%
%%% eq:tp1-1 %%%
%%%%%%%%%%%%%%%%
%
\begin{equation} \label{eq:tp1-1}
M_{ws_{i} \cdot \Lambda_{i}} + 
M_{ws_{j} \cdot \Lambda_{j}}
=
\min\Bigl(
M_{w \cdot \Lambda_{i}}+
M_{ws_{i}s_{j} \cdot \Lambda_{j}}, \ 
M_{ws_{j}s_{i} \cdot \Lambda_{i}}+
M_{w \cdot \Lambda_{j}}
\Bigr);
\end{equation}

\noindent
(2) $a_{ij}=-1$, $a_{ji}=-2$, and 
%
%%%%%%%%%%%%%%%%
%%% eq:tp2-1 %%%
%%%%%%%%%%%%%%%%
%
\begin{equation} \label{eq:tp2-1}
M_{ws_{j} \cdot \Lambda_{j}}+ 
M_{ws_{i}s_{j} \cdot \Lambda_{j}}+
M_{ws_{i} \cdot \Lambda_{i}} = 
\min \left(
 \begin{array}{l}
 2M_{ws_{i}s_{j} \cdot \Lambda_{j}}+
 M_{w \cdot \Lambda_{i}}, \\[3mm]
 2M_{w \cdot \Lambda_{j}}+
 M_{ws_{i}s_{j}s_{i} \cdot \Lambda_{i}}, \\[3mm]
 M_{w \cdot \Lambda_{j}} + 
 M_{ws_{j}s_{i}s_{j} \cdot \Lambda_{j}}+ 
 M_{ws_{i} \cdot \Lambda_{i}}
 \end{array}
 \right),
\end{equation}
%
%%%%%%%%%%%%%%%%
%%% eq:tp2-2 %%%
%%%%%%%%%%%%%%%%
%
\begin{equation} \label{eq:tp2-2}
M_{ws_{j}s_{i} \cdot \Lambda_{i}}+ 
2M_{ws_{i}s_{j} \cdot \Lambda_{j}}+
M_{ws_{i} \cdot \Lambda_{i}} = 
\min \left(
 \begin{array}{l}
 2M_{w \cdot \Lambda_{j}}+
 2M_{ws_{i}s_{j}s_{i} \cdot \Lambda_{i}}, \\[3mm]
 2M_{ws_{j}s_{i}s_{j} \cdot \Lambda_{j}}+
 2M_{ws_{i} \cdot \Lambda_{i}}, \\[3mm]
 M_{ws_{i}s_{j}s_{i} \cdot \Lambda_{i}} + 
 2M_{ws_{i}s_{j} \cdot \Lambda_{j}}+ 
 M_{w \cdot \Lambda_{i}}
 \end{array}
 \right);
\end{equation}

\noindent
(3) $a_{ij}=-2$, $a_{ji}=-1$, and 
%
%%%%%%%%%%%%%%%%
%%% eq:tp3-1 %%%
%%%%%%%%%%%%%%%%
%
\begin{equation} \label{eq:tp3-1}
M_{ws_{j}s_{i} \cdot \Lambda_{i}}+ 
M_{ws_{i} \cdot \Lambda_{i}}+
M_{ws_{i}s_{j} \cdot \Lambda_{j}}=
\min \left(
 \begin{array}{l}
 2M_{ws_{i} \cdot \Lambda_{i}}+
 M_{ws_{j}s_{i}s_{j} \cdot \Lambda_{j}}, \\[3mm]
 2M_{ws_{i}s_{j}s_{i} \cdot \Lambda_{i}}+
 M_{w \cdot \Lambda_{j}}, \\[3mm]
 M_{ws_{i}s_{j}s_{i} \cdot \Lambda_{i}} + 
 M_{w \cdot \Lambda_{i}}+ 
 M_{ws_{i}s_{j} \cdot \Lambda_{j}}
 \end{array}
 \right),
\end{equation}
%
%%%%%%%%%%%%%%%%
%%% eq:tp3-2 %%%
%%%%%%%%%%%%%%%%
%
\begin{equation} \label{eq:tp3-2}
M_{ws_{j} \cdot \Lambda_{j}}+ 
2M_{ws_{i} \cdot \Lambda_{i}}+
M_{ws_{i}s_{j} \cdot \Lambda_{j}} = 
\min \left(
 \begin{array}{l}
 2M_{ws_{i}s_{j}s_{i} \cdot \Lambda_{i}}+
 2M_{w \cdot \Lambda_{j}}, \\[3mm]
 2M_{w \cdot \Lambda_{i}}+
 2M_{ws_{i}s_{j} \cdot \Lambda_{j}}, \\[3mm]
 M_{w \cdot \Lambda_{j}} + 
 2M_{ws_{i} \cdot \Lambda_{i}}+ 
 M_{ws_{j}s_{i}s_{j} \cdot \Lambda_{j}}
 \end{array}
 \right).
\end{equation}
We omit the tropical Pl\"ucker relations for the case 
$a_{ij}a_{ji}=3$, since we do not use them in this paper. 

We say that an element $M_{\bullet} \in \edge$ satisfies 
the tropical Pl\"ucker relations if it satisfies 
the tropical Pl\"ucker relation at $(w,\,i,\,j)$ for 
each $w \in W$ and $i,\,j \in I$ such that 
$ws_{i} > w$, $ws_{j} > w$, and $i \ne j$.

%%%%%%%%%%%%%%%%
%%% dfn:bzmv %%%
%%%%%%%%%%%%%%%%
%
\begin{dfn} \label{dfn:bzmv}
An element $M_{\bullet}=(M_{\gamma})_{\gamma \in \Gamma} \in 
\edge$ is called a Berenstein-Zelevinsky (BZ for short) datum 
if $M_{\gamma} \in \BZ$ for all $\gamma \in \Gamma$, and 
if it satisfies the tropical Pl\"ucker relations.
In this case, the pseudo-Weyl polytope $P(M_{\bullet})$ 
with edge datum $M_{\bullet}$ is called a Mirkovi\'c-Vilonen 
(MV for short) polytope for $\Fg$.
\end{dfn}

Let $\edge_{\MV}$ denote the subset of $\edge$ 
consisting of all BZ data, and $\ggms_{\MV}$ 
the corresponding subset of $\ggms$ 
under the bijection $D:\ggms \rightarrow \edge$. 
We set 
\begin{equation*}
\mv:=
 \bigl\{P(M_{\bullet}) \mid 
        M_{\bullet} \in \edge_{\MV} \bigr\}=
 \bigl\{P(\mu_{\bullet}) \mid 
        \mu_{\bullet} \in \ggms_{\MV}\bigr\} \subset \pwp.
\end{equation*}

\begin{rem}
If $\mu_{\bullet}=(\mu_{w})_{w \in W} \in \ggms$ corresponds to 
$M_{\bullet}=(M_{\gamma})_{\gamma \in \Gamma} \in \edge$ under 
the bijection $D:\ggms \rightarrow \edge$, then, by \eqref{eq:M} 
\begin{equation*}
\text{$\mu_{w} \in \Fh_{\BZ}$ for all $w \in W$}
\quad \Longleftrightarrow \quad
\text{$M_{\gamma} \in \BZ$ for all $\gamma \in \Gamma$}.
\end{equation*}
Hence, if $\mu_{\bullet}=(\mu_{w})_{w \in W} \in \ggms_{\MV}$, 
then $\mu_{w} \in \Fh_{\BZ}$ for all $w \in W$.
\end{rem}

Now, let $\CB$ denote the canonical basis of 
the negative part $U_{q}^{-}(\Fg^{\vee})$ of 
the quantized universal enveloping algebra 
$U_{q}(\Fg^{\vee})$ associated to 
the (Langlands) dual Lie algebra $\Fg^{\vee}$ 
(see \cite[Part~4]{L}).
For each reduced word $\bi=(i_{1},\,i_{2},\,\dots,\,i_{m})$ for 
the longest element $w_{0} \in W$, where $m$ denotes
the length of $w_{0}$, there exists a bijection 
$b_{\bi}:\BZ_{\ge 0}^{m} \rightarrow \CB$, 
which is called a Lusztig parametrization of $\CB$ 
(see \cite[Proposition~8.2]{L2}).
Also, by \cite[Theorem~7.1]{Kam1}, there exists a bijection 
$\psi_{\bi}:\mv \rightarrow \BZ_{\ge 0}^{m}$ given by: 
$\psi_{\bi}(P(\mu_{\bullet}))=(L_{1},\,L_{2},\,\dots,\,L_{m})$, 
where the $L_{k} \in \BZ_{\ge 0}$, $1 \le k \le m$, are determined 
via the length formula (see \eqref{eq:length}): 
$\mu_{w_{k}^{\bi}}-\mu_{w_{k-1}^{\bi}}=
 L_{k}w_{k-1}^{\bi} \cdot h_{i_{k}}$, with 
$w_{k}^{\bi}:=s_{i_{1}}s_{i_{2}} \cdots s_{i_{k}}$, 
for $1 \le k \le m$. 
Furthermore, we know from \cite[Theorem~7.2]{Kam1} 
that there exists a bijection $\Psi':\mv \rightarrow \CB$ 
such that $\Psi'=b_{\bi} \circ \psi_{\bi}$ holds for 
all reduced words $\bi$ for $w_{0}$. 
Thus, we define a bijection $\Psi:\mv \rightarrow \CB(\infty)$ to be 
the composition of the bijection $\Psi':\mv \rightarrow \CB$ with 
the canonical bijection from the canonical basis $\CB$ onto 
the crystal basis $\CB(\infty)$ for the negative part 
$U_{q}^{-}(\Fg^{\vee})$. 

We endow $\mv$ with a crystal structure 
(due to Lusztig and Berenstein-Zelevinsky) for $\Fg^{\vee}$ 
through the bijection $\Psi:\mv \rightarrow \CB(\infty)$ above 
so that $\Psi:\mv \rightarrow \CB(\infty)$ is 
an isomorphism of crystals for $\Fg^{\vee}$.
Let us recall from \cite[\S\S3.5 and 3.6]{Kam2} 
a description of this crystal structure on $\mv$. 
Let $P=P(\mu_{\bullet}) \in \mv$ be 
an MV polytope with GGMS datum 
$\mu_{\bullet}=(\mu_{w})_{w \in W} \in \ggms_{\MV}$.
The weight $\wt(P)$ of $P$ is, by definition, equal to 
the vertex $\mu_{e} \in \sum_{j \in I} \BZ_{\le 0} h_{j}
=\bigl(\sum_{j \in I} \BR_{\le 0} h_{j}\bigr) \cap \Fh_{\BZ}$.
For each $j \in I$, let $f_{j}$ (resp., $e_{j}$) 
denote the lowering (resp., raising) 
Kashiwara operator on $\mv$.
Then, $e_{j}P$ and $f_{j}P$ for each $j \in I$ 
are given as follows (see \cite[Theorem~3.5]{Kam2}). 
If $\mu_{e}=\mu_{s_{j}}$, then 
$e_{j}P=\bzero$, where $\bzero$ is an additional element, 
which is not contained in $\mv$. 
Otherwise, $e_{j}P$ is a unique 
MV polytope $P(\mu_{\bullet}') \in \mv$ 
with GGMS datum $\mu_{\bullet}'=(\mu_{w}')_{w \in W}$ 
such that $\mu_{e}'=\mu_{e}+h_{j}$, and 
$\mu_{w}'=\mu_{w}$ for all $w \in W$ with $s_{j}w < w$. 
Similarly, $f_{j}P$ is a unique MV polytope 
$P(\mu_{\bullet}') \in \mv$ 
with GGMS datum $\mu_{\bullet}'=(\mu_{w}')_{w \in W}$ 
such that $\mu_{e}'=\mu_{e}-h_{j}$, and 
$\mu_{w}'=\mu_{w}$ for all $w \in W$ with $s_{j}w < w$.
Note that since $s_{j}w_{0} < w_{0}$ for all $j \in I$, 
$\mu_{w_{0}}=0$ implies $\mu_{w_{0}}'=0$. 
It is understood that $e_{j}\bzero=f_{j}\bzero=\bzero$. 
In addition, we set
$\ve_{j}(P):=
 \max \bigl\{e_{j}^{k}P \mid e_{j}^{k}P= \bzero \bigr\}$ 
and 
$\vp_{j}(P):=
 \pair{\alpha_{j}}{\wt(P)}+\ve_{j}(P)$. 
%
%%%%%%%%%%%%%%%%
%%% rem:binf %%%
%%%%%%%%%%%%%%%%
%
\begin{rem} \label{rem:binf}
Define an element $\mu_{\bullet}^{0}=(\mu_{w})_{w \in W}$ of 
$\ggms$ by: $\mu_{w}=0 \in \Fh$ for all $w \in W$. 
It is obvious that 
$\mu_{\bullet}^{0} \in \ggms$ 
is contained in $\ggms_{\MV}$, 
and the weight of the MV polytope 
$P^{0}:=P(\mu_{\bullet}^{0}) \in \mv$ is equal to 
$0 \in \Fh_{\BZ}$. 
Therefore, under the isomorphism 
$\Psi:\mv \stackrel{\sim}{\rightarrow} \CB(\infty)$ 
of crystals for $\Fg^{\vee}$,
the MV polytope $P^{0} \in \mv$ is sent to 
the element $u_{\infty} \in \CB(\infty)$ 
corresponding to the identity element 
$1 \in U_{q}^{-}(\Fg^{\vee})$. 
\end{rem}

%==============================%
%     START SUBSECTION 0202    %
%==============================%
%
\subsection{Transition map between Lusztig parametrizations.}
\label{subsec:trans}

In this subsection, we keep the notation and 
assumptions of \S\ref{subsec:MV}. 
For two reduced words $\bi$ and $\bi'$ for the longest element 
$w_{0} \in W$ of length $m$, we define the transition map 
$R_{\bi}^{\bi'}:\BZ_{\ge 0}^{m} \rightarrow \BZ_{\ge 0}^{m}$ 
between Lusztig parametrizations by: 
$R_{\bi}^{\bi'}=b_{\bi'}^{-1} \circ b_{\bi}$. Note that 
the transition map 
$R_{\bi}^{\bi'}:\BZ_{\ge 0} \rightarrow \BZ_{\ge 0}$ is 
identical to the bijection $\psi_{\bi'} \circ \psi_{\bi}^{-1}:
 \BZ_{\ge 0}^{m} \rightarrow \BZ_{\ge 0}^{m}$ 
since $b_{\bi} \circ \psi_{\bi}=
b_{\bi'} \circ \psi_{\bi'} \, (=\Psi')$. 

In this subsection, we briefly review 
the theory of ``geometric lifting'' of 
the transition map between Lusztig parametrizations of 
the canonical basis, which plays a key role in our proof 
of Proposition~\ref{prop:phi-bz} below. 
Let $G=G(\BC)$ be a connected, simply-connected, 
semisimple algebraic group (or rather, Lie group)
over $\BC$ with Lie algebra $\Fg$. For $j \in I$, 
we denote by $x_{j}(t)$ (resp., $y(t)$), $t \in \BC$, 
the one-parameter subgroup of $G$ given by: 
$x_{j}(t)=\exp(tx_{j})$ (resp., $y_{j}(t)=\exp(ty_{j})$) for 
$t \in \BC$, where $\exp:\Fg \rightarrow G$ 
denotes the exponential map. 
Now, let $N_{\ge 0}$ denote the multiplicative 
semigroup generated by all $x_{j}(t)$ for $j \in I$ and 
$t \ge 0$, and set $N_{> 0}:=N_{\ge 0} \cap (B_{-}w_{0}B_{-})$, 
where $B_{-}$ is the Borel subgroup of $G$ generated by all 
$y_{j}(t)=\exp(ty_{j})$ for $j \in I$ and $t \in \BC$, together with 
the maximal torus $T$ of $G$ with Lie algebra $\Fh$. 
Each reduced word $\bi=(i_{1},\,i_{2},\,\dots,\,i_{m})$ for 
$w_{0}$ gives rise to a bijection 
$x_{\bi}:\BR_{> 0}^{m} \rightarrow N_{> 0}$ by: 
\begin{equation*}
x_{\bi}(t_{1},\,t_{2},\,\dots,\,t_{m})=
x_{i_{1}}(t_{1})x_{i_{2}}(t_{2}) \cdots x_{i_{m}}(t_{m})
\end{equation*}
for $(t_{1},\,t_{2},\,\dots,\,t_{m}) \in \BR_{> 0}^{m}$ 
(see \cite{L1}). The following is one of the main results 
of \cite{BZ2} (for the ``tropicalization'' procedure, we refer 
the reader to \cite[\S2.1]{BFZ}, \cite[\S2.4]{BK}, and 
also \cite[\S1.3]{NY}). 
%
%%%%%%%%%%%%%%%%%
%%% thm:trans %%%
%%%%%%%%%%%%%%%%%
%
\begin{thm}[{\cite[Theorem~5.2]{BZ2}}] \label{thm:trans}
Let $\bi$, $\bi'$ be two reduced words for $w_{0} \in W$. 

\noindent
{\rm (1)} Each component of the transition map
$\CR_{\bi}^{\bi'}(t_{1},\,t_{2},\,\dots,\,t_{m}):=
 x_{\bi'}^{-1} \circ x_{\bi}: \BR_{> 0}^{m} \rightarrow \BR_{> 0}^{m}$
is a subtraction-free rational expression in $t_{1},\,t_{2},\,\dots,\,t_{m}$. 

\noindent
{\rm (2)} Each component of the transition map
$R_{\bi}^{\bi'}=b_{\bi'}^{-1} \circ b_{\bi}: 
 \BZ_{\ge 0}^{m} \rightarrow \BZ_{\ge 0}^{m}$ is 
the tropicalization of the corresponding component of 
$\CR_{\bi}^{\bi'}(t_{1},\,t_{2},\,\dots,\,t_{m})$. 
\end{thm}
%
%%%%%%%%%%%%%%%%%
%%% rem:trans %%%
%%%%%%%%%%%%%%%%%
%
\begin{rem} \label{rem:trans}
For later use, 
we record explicit formulas for the transition map 
$\CR_{\bi}^{\bi'}: \BR_{> 0}^{d} \rightarrow \BR_{> 0}^{d}$ 
from \cite[Theorem~3.1]{BZ1}, where $\bi$ and $\bi'$ have the form 
$\bi=(i,\,j,\,i,\,\dots)$, 
$\bi'=(j,\,i,\,j,\,\dots)$ of length $d$. We use the notation 
$\CR_{\bi}^{\bi'}
 (t_{1},\,t_{2},\,\dots,\,t_{d})=
 (t_{1}',\,t_{2}',\,\dots,\,t_{d}')$. 
Note that explicit formulas for the transition map 
$R_{\bi}^{\bi'}: \BZ_{\ge 0}^{d} \rightarrow \BZ_{\ge 0}^{d}$ 
are also obtained from these formulas through 
the tropicalization procedure by Theorem~\ref{thm:trans}\,(2). 

\noindent
(0) If $a_{ij}=a_{ji}=0$, then $d=2$ and 
$t_{1}'=t_{2}$, \, $t_{2}'=t_{1}$. 

\noindent
(1) If $a_{ij}=a_{ji}=-1$, then $d=3$ and 
\begin{equation*}
t_{1}'=\frac{t_{2}t_{3}}{\pi}, \qquad
t_{2}'=t_{1}+t_{3}, \qquad
t_{3}'=\frac{t_{1}t_{2}}{\pi},
\end{equation*}
where $\pi=t_{1}+t_{3}$. 

\noindent
(2) If $a_{ij}=-1$, $a_{ji}=-2$, then $d=4$ and 
\begin{equation*}
t_{1}'=\frac{t_{2}t_{3}t_{4}}{\pi_{1}}, \qquad
t_{2}'=\frac{\pi_{1}^{2}}{\pi_{2}}, \qquad
t_{3}'=\frac{\pi_{2}}{\pi_{1}}, \qquad
t_{4}'=\frac{t_{1}t_{2}^{2}t_{3}}{\pi_{2}},
\end{equation*}
where $\pi_{1}=t_{1}t_{2}+(t_{1}+t_{3})t_{4}$, \, 
$\pi_{2}=t_{1}(t_{2}+t_{4})^{2}+t_{3}t_{4}^{2}$. 

\noindent
(3) If $a_{ij}=-2$, $a_{ji}=-1$, then $d=4$ and 
\begin{equation*}
t_{1}'=\frac{t_{2}t_{3}^{2}t_{4}}{\pi_{2}}, \qquad
t_{2}'=\frac{\pi_{2}}{\pi_{1}}, \qquad
t_{3}'=\frac{\pi_{1}^{2}}{\pi_{2}}, \qquad
t_{4}'=\frac{t_{1}t_{2}t_{3}}{\pi_{1}}, 
\end{equation*}
where $\pi_{1}=t_{1}t_{2}+(t_{1}+t_{3})t_{4}$, \, 
$\pi_{2}=t_{1}^{2}t_{2}+(t_{1}+t_{3})^{2}t_{4}$. 
\end{rem}

%==============================%
%     START SUBSECTION 0203    %
%==============================%
%
\subsection{Diagram automorphism for $A_{\ell}$.}
\label{subsec:da}

For the remainder of this paper, 
we assume that $\Fg$ is of type $A_{\ell}$, 
$\ell \ge 3$, and $I:=\bigl\{1,\,2,\,\dots,\,\ell\bigr\}$.
Let $\omega:I \rightarrow I$ be the Dynkin diagram automorphism 
of order $2$ given by: $\omega(j)=\ell-j+1$ 
for $j \in I$. Then, the $\omega:I \rightarrow I$ induces a Lie algebra 
automorphism $\omega \in \Aut(\Fg)$ of order $2$ such that 
$\omega(x_{j})=x_{\omega(j)}$, 
$\omega(y_{j})=y_{\omega(j)}$, 
$\omega(h_{j})=h_{\omega(j)}$ for $j \in I$.
Note that the Cartan subalgebra $\Fh$ is stable 
under $\omega \in \Aut(\Fg)$, and hence induces 
$\omega \in \GL(\Fh^{\ast})$ by: 
$\pair{\omega(\lambda)}{h}=\pair{\lambda}{\omega(h)}$ for 
$\lambda \in \Fh^{\ast}$ and $h \in \Fh$. We set 
\begin{equation*}
\Fg^{\omega}:=\bigl\{x \in \Fg \mid \omega(x)=x\bigr\}
\quad \text{and} \quad 
\Fh^{\omega}:=\bigl\{h \in \Fh \mid \omega(h)=h\bigr\}.
\end{equation*}
Furthermore, the $\omega:I \rightarrow I$ 
induces a group automorphism 
$\omega \in \Aut(W)$ of order $2$ such that 
$\omega(s_{i})=s_{\omega(i)}$ for $i \in I$. 
We set $W^{\omega}:=\bigl\{ w \in W \mid \omega(w)=w \bigr\}$.
%
%%%%%%%%%%%%%%%%%
%%% rem:omega %%%
%%%%%%%%%%%%%%%%%
%
\begin{rem} \label{rem:omega}
(1) We see easily from 
the definition of $\omega \in \Aut(W)$ that 
if $w > w'$, then $\omega(w) > \omega(w')$ for $w,\,w' \in W$, 
and $\ell(\omega(w))=\ell(w)$ for $w \in W$, 
where $\ell:W \rightarrow \BZ_{\ge 0}$ denotes 
the length function on $W$. 
In particular, we have $\omega(e)=e$ and $\omega(w_{0})=w_{0}$.

\noindent
(2) It follows immediately 
from the definition of $\omega \in \GL(\Fh^{\ast})$ that 
$\omega(\Lambda_{j})=\Lambda_{\omega(j)}$ and 
$\omega(\alpha_{j})=\alpha_{\omega(j)}$ for $j \in I$. 

\noindent
(3) It is easy to show that 
%
%%%%%%%%%%%%%%%%
%%% eq:omg0* %%%
%%%%%%%%%%%%%%%%
%
\begin{align}
& \omega(w \lambda)=\omega(w)\,\omega(\lambda) 
  \quad \text{for $w \in W$ and $\lambda \in \Fh^{\ast}$}, 
  \label{eq:omg01} \\
& \omega(w h)=\omega(w)\,\omega(h) 
  \quad \text{for $w \in W$ and $h \in \Fh$}. 
  \label{eq:omg02}
\end{align}
In particular, it follows from \eqref{eq:omg02} that 
$\Fh^{\omega} \subset \Fh$ is stable under the action of 
$W^{\omega} \subset W$. 

\noindent
(4) It follows from part (2) and \eqref{eq:omg01} that 
$\omega(w\Lambda_{i})=\omega(w) \Lambda_{\omega(i)}$ for 
$w \in W$ and $i \in I$. Therefore, the set 
$\Gamma$ is stable under the action of 
$\omega \in \GL(\Fh^{\ast})$.

\noindent
(5) We see easily that $h \ge_{w} h'$ 
if and only if 
$\omega(h) \ge_{\omega(w)} \omega(h')$ 
for $w \in W$ and $h,\,h' \in \Fh_{\BR}$.
\end{rem}

In the following, we assume that $\Fg$ is either of type $A_{\ell}$ 
with $\ell=2n-1$, $n \in \BZ_{\ge 2}$, 
or of type $A_{\ell}$ with $\ell=2n$, $n \in \BZ_{\ge 2}$. 
If $\ell=2n-1$, $n \in \BZ_{\ge 2}$, then we know 
(see, for example, \cite[\S8.3]{Kac}) that
the fixed point subalgebra $\Fg^{\omega}$ 
is the finite-dimensional 
simple Lie algebra of type $C_{n}$ (see the figure below); 
the Cartan subalgebra of $\Fg^{\omega}$ is $\Fh^{\omega}$, and 
the Chevalley generators 
$\bigl\{ x_{j}^{\omega},\, y_{j}^{\omega},\, h_{j}^{\omega} \mid 
1 \le j \le n \bigr\}$ of $\Fg^{\omega}$ are as follows:
%
%%%%%%%%%%%%%%%%
%%% eq:h-odd %%%
%%%%%%%%%%%%%%%%
%
\begin{align}
& x_{j}^{\omega}=x_{j}+x_{\omega(j)} \quad
  \text{for $1 \le j \le n-1$}, \qquad 
  x_{n}^{\omega}=x_{n}, \nonumber \\
& y_{j}^{\omega}=y_{j}+y_{\omega(j)} \quad
  \text{for $1 \le j \le n-1$}, \qquad 
  y_{n}^{\omega}=y_{n}, \nonumber \\
& h_{j}^{\omega}=h_{j}+h_{\omega(j)} \quad
  \text{for $1 \le j \le n-1$}, \qquad 
  h_{n}^{\omega}=h_{n}. \label{eq:h-odd}
\end{align}

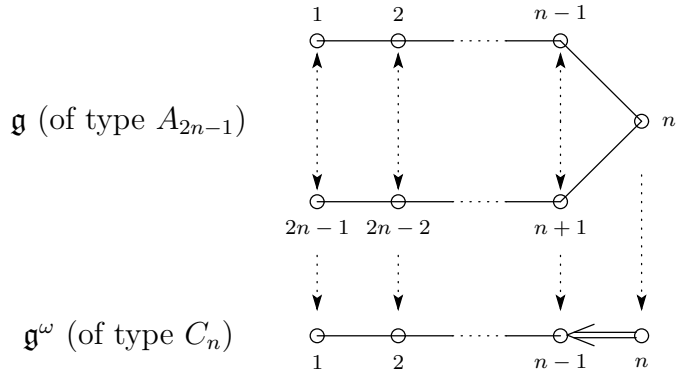
\svsp

\begin{center}
{\scriptsize
\hspace*{-50mm}
%
% \input{fixedmv-fig01.tex}
%
%WinTpicVersion3.08
\unitlength 0.1in
\begin{picture}( 47.6200, 18.2000)(  8.1900,-21.0400)
% CIRCLE 2 0 3 0
% 4 3866 509 3866 474 3866 474 3866 474
% 
\special{pn 8}%
\special{ar 3866 510 36 36  0.0000000 6.2831853}%
% CIRCLE 2 0 3 0
% 4 3866 1349 3866 1314 3866 1314 3866 1314
% 
\special{pn 8}%
\special{ar 3866 1350 36 36  0.0000000 6.2831853}%
% CIRCLE 2 0 3 0
% 4 4286 509 4286 474 4286 474 4286 474
% 
\special{pn 8}%
\special{ar 4286 510 36 36  0.0000000 6.2831853}%
% CIRCLE 2 0 3 0
% 4 4286 1349 4286 1314 4286 1314 4286 1314
% 
\special{pn 8}%
\special{ar 4286 1350 36 36  0.0000000 6.2831853}%
% CIRCLE 2 0 3 0
% 4 5126 509 5126 474 5126 474 5126 474
% 
\special{pn 8}%
\special{ar 5126 510 36 36  0.0000000 6.2831853}%
% CIRCLE 2 0 3 0
% 4 5126 1349 5126 1314 5126 1314 5126 1314
% 
\special{pn 8}%
\special{ar 5126 1350 36 36  0.0000000 6.2831853}%
% CIRCLE 2 0 3 0
% 4 5546 929 5546 894 5546 894 5546 894
% 
\special{pn 8}%
\special{ar 5546 930 36 36  0.0000000 6.2831853}%
% LINE 2 2 3 0
% 2 4566 509 4846 509
% 
\special{pn 8}%
\special{pa 4566 510}%
\special{pa 4846 510}%
\special{dt 0.045}%
% LINE 2 2 3 0
% 2 4566 1349 4846 1349
% 
\special{pn 8}%
\special{pa 4566 1350}%
\special{pa 4846 1350}%
\special{dt 0.045}%
% CIRCLE 2 0 3 0
% 4 3866 2049 3866 2014 3866 2014 3866 2014
% 
\special{pn 8}%
\special{ar 3866 2050 36 36  0.0000000 6.2831853}%
% CIRCLE 2 0 3 0
% 4 4286 2049 4286 2014 4286 2014 4286 2014
% 
\special{pn 8}%
\special{ar 4286 2050 36 36  0.0000000 6.2831853}%
% CIRCLE 2 0 3 0
% 4 5126 2049 5126 2014 5126 2014 5126 2014
% 
\special{pn 8}%
\special{ar 5126 2050 36 36  0.0000000 6.2831853}%
% CIRCLE 2 0 3 0
% 4 5546 2049 5546 2014 5546 2014 5546 2014
% 
\special{pn 8}%
\special{ar 5546 2050 36 36  0.0000000 6.2831853}%
% LINE 2 0 3 0
% 4 5126 509 5546 929 5546 929 5126 1349
% 
\special{pn 8}%
\special{pa 5126 510}%
\special{pa 5546 930}%
\special{fp}%
\special{pa 5546 930}%
\special{pa 5126 1350}%
\special{fp}%
% LINE 2 0 3 0
% 8 3866 509 4566 509 4846 509 5126 509 5126 1349 4846 1349 4566 1349 3866 1349
% 
\special{pn 8}%
\special{pa 3866 510}%
\special{pa 4566 510}%
\special{fp}%
\special{pa 4846 510}%
\special{pa 5126 510}%
\special{fp}%
\special{pa 5126 1350}%
\special{pa 4846 1350}%
\special{fp}%
\special{pa 4566 1350}%
\special{pa 3866 1350}%
\special{fp}%
% LINE 2 2 3 0
% 2 4566 2049 4846 2049
% 
\special{pn 8}%
\special{pa 4566 2050}%
\special{pa 4846 2050}%
\special{dt 0.045}%
% LINE 2 0 3 0
% 4 3866 2049 4566 2049 4846 2049 5126 2049
% 
\special{pn 8}%
\special{pa 3866 2050}%
\special{pa 4566 2050}%
\special{fp}%
\special{pa 4846 2050}%
\special{pa 5126 2050}%
\special{fp}%
% VECTOR 2 2 3 0
% 8 3866 1629 3866 1909 4286 1629 4286 1909 5126 1629 5126 1909 5546 1209 5546 1909
% 
\special{pn 8}%
\special{pa 3866 1630}%
\special{pa 3866 1910}%
\special{dt 0.045}%
\special{sh 1}%
\special{pa 3866 1910}%
\special{pa 3886 1842}%
\special{pa 3866 1856}%
\special{pa 3846 1842}%
\special{pa 3866 1910}%
\special{fp}%
\special{pa 4286 1630}%
\special{pa 4286 1910}%
\special{dt 0.045}%
\special{sh 1}%
\special{pa 4286 1910}%
\special{pa 4306 1842}%
\special{pa 4286 1856}%
\special{pa 4266 1842}%
\special{pa 4286 1910}%
\special{fp}%
\special{pa 5126 1630}%
\special{pa 5126 1910}%
\special{dt 0.045}%
\special{sh 1}%
\special{pa 5126 1910}%
\special{pa 5146 1842}%
\special{pa 5126 1856}%
\special{pa 5106 1842}%
\special{pa 5126 1910}%
\special{fp}%
\special{pa 5546 1210}%
\special{pa 5546 1910}%
\special{dt 0.045}%
\special{sh 1}%
\special{pa 5546 1910}%
\special{pa 5566 1842}%
\special{pa 5546 1856}%
\special{pa 5526 1842}%
\special{pa 5546 1910}%
\special{fp}%
% STR 2 0 3 0
% 3 3866 300 3866 369 5 0
% $1$
\put(38.6600,-3.6900){\makebox(0,0){$1$}}%
% STR 2 0 3 0
% 3 4286 300 4286 369 5 0
% $2$
\put(42.8600,-3.6900){\makebox(0,0){$2$}}%
% STR 2 0 3 0
% 3 5126 300 5126 369 5 0
% $n-1$
\put(51.2600,-3.6900){\makebox(0,0){$n-1$}}%
% STR 2 0 3 0
% 3 5686 859 5686 929 5 0
% $n$
\put(56.8600,-9.2900){\makebox(0,0){$n$}}%
% STR 2 0 3 0
% 3 5126 1419 5126 1489 5 0
% $n+1$
\put(51.2600,-14.8900){\makebox(0,0){$n+1$}}%
% STR 2 0 3 0
% 3 4286 1419 4286 1489 5 0
% $2n-2$
\put(42.8600,-14.8900){\makebox(0,0){$2n-2$}}%
% STR 2 0 3 0
% 3 3866 1419 3866 1489 5 0
% $2n-1$
\put(38.6600,-14.8900){\makebox(0,0){$2n-1$}}%
% STR 2 0 3 0
% 3 3866 2119 3866 2189 5 0
% $1$
\put(38.6600,-21.8900){\makebox(0,0){$1$}}%
% STR 2 0 3 0
% 3 4286 2119 4286 2189 5 0
% $2$
\put(42.8600,-21.8900){\makebox(0,0){$2$}}%
% STR 2 0 3 0
% 3 5126 2119 5126 2189 5 0
% $n-1$
\put(51.2600,-21.8900){\makebox(0,0){$n-1$}}%
% STR 2 0 3 0
% 3 5546 2119 5546 2189 5 0
% $n$
\put(55.4600,-21.8900){\makebox(0,0){$n$}}%
% VECTOR 2 2 3 0
% 2 3866 929 3866 579
% 
\special{pn 8}%
\special{pa 3866 930}%
\special{pa 3866 580}%
\special{dt 0.045}%
\special{sh 1}%
\special{pa 3866 580}%
\special{pa 3846 646}%
\special{pa 3866 632}%
\special{pa 3886 646}%
\special{pa 3866 580}%
\special{fp}%
% VECTOR 2 2 3 0
% 2 4286 929 4286 579
% 
\special{pn 8}%
\special{pa 4286 930}%
\special{pa 4286 580}%
\special{dt 0.045}%
\special{sh 1}%
\special{pa 4286 580}%
\special{pa 4266 646}%
\special{pa 4286 632}%
\special{pa 4306 646}%
\special{pa 4286 580}%
\special{fp}%
% VECTOR 2 2 3 0
% 2 5126 929 5126 579
% 
\special{pn 8}%
\special{pa 5126 930}%
\special{pa 5126 580}%
\special{dt 0.045}%
\special{sh 1}%
\special{pa 5126 580}%
\special{pa 5106 646}%
\special{pa 5126 632}%
\special{pa 5146 646}%
\special{pa 5126 580}%
\special{fp}%
% VECTOR 2 2 3 0
% 2 5126 929 5126 1279
% 
\special{pn 8}%
\special{pa 5126 930}%
\special{pa 5126 1280}%
\special{dt 0.045}%
\special{sh 1}%
\special{pa 5126 1280}%
\special{pa 5146 1212}%
\special{pa 5126 1226}%
\special{pa 5106 1212}%
\special{pa 5126 1280}%
\special{fp}%
% VECTOR 2 2 3 0
% 2 4286 929 4286 1279
% 
\special{pn 8}%
\special{pa 4286 930}%
\special{pa 4286 1280}%
\special{dt 0.045}%
\special{sh 1}%
\special{pa 4286 1280}%
\special{pa 4306 1212}%
\special{pa 4286 1226}%
\special{pa 4266 1212}%
\special{pa 4286 1280}%
\special{fp}%
% VECTOR 2 2 3 0
% 2 3866 929 3866 1279
% 
\special{pn 8}%
\special{pa 3866 930}%
\special{pa 3866 1280}%
\special{dt 0.045}%
\special{sh 1}%
\special{pa 3866 1280}%
\special{pa 3886 1212}%
\special{pa 3866 1226}%
\special{pa 3846 1212}%
\special{pa 3866 1280}%
\special{fp}%
% STR 2 0 3 0
% 3 2890 860 2890 930 5 0
% {\normalsize $\Fg$ (of type $A_{2n-1}$)}
\put(28.9000,-9.3000){\makebox(0,0){{\normalsize $\Fg$ (of type $A_{2n-1}$)}}}%
% STR 2 0 3 0
% 3 2889 1983 2889 2053 5 0
% {\normalsize $\Fg^{\omega}$ (of type $C_{n}$)}
\put(28.8900,-20.5300){\makebox(0,0){{\normalsize $\Fg^{\omega}$ (of type $C_{n}$)}}}%
% LINE 2 0 3 0
% 2 5517 2032 5195 2032
% 
\special{pn 8}%
\special{pa 5518 2032}%
\special{pa 5196 2032}%
\special{fp}%
% LINE 2 0 3 0
% 2 5517 2060 5195 2060
% 
\special{pn 8}%
\special{pa 5518 2060}%
\special{pa 5196 2060}%
\special{fp}%
% LINE 2 0 3 0
% 2 5160 2046 5300 1990
% 
\special{pn 8}%
\special{pa 5160 2046}%
\special{pa 5300 1990}%
\special{fp}%
% LINE 2 0 3 0
% 2 5160 2046 5300 2102
% 
\special{pn 8}%
\special{pa 5160 2046}%
\special{pa 5300 2102}%
\special{fp}%
\end{picture}%
}
\end{center}

If $\ell=2n$, $n \in \BZ_{\ge 2}$, then 
we know (see, for example, \cite[\S8.3]{Kac}) that
the fixed point subalgebra $\Fg^{\omega}$ is 
the finite-dimensional simple Lie algebra of type $B_{n}$ 
(see the figure below); 
the Cartan subalgebra of $\Fg^{\omega}$ is $\Fh^{\omega}$, 
and the Chevalley generators 
$\bigl\{ x_{j}^{\omega},\, y_{j}^{\omega},\, h_{j}^{\omega} \mid 
1 \le j \le n \bigr\}$ of $\Fg^{\omega}$ are as follows: 
%
%%%%%%%%%%%%%%%%%
%%% eq:h-even %%%
%%%%%%%%%%%%%%%%%
%
\begin{align}
& x_{j}^{\omega}=x_{j}+x_{\omega(j)} \quad
  \text{for $1 \le j \le n-1$}, \qquad 
  x_{n}^{\omega}=\sqrt{2}(x_{n}+x_{\omega(n)}), \nonumber \\
& y_{j}^{\omega}=y_{j}+y_{\omega(j)} \quad
  \text{for $1 \le j \le n-1$}, \qquad 
  y_{n}^{\omega}=\sqrt{2}(y_{n}+y_{\omega(n)}), \nonumber \\
& h_{j}^{\omega}=h_{j}+h_{\omega(j)} \quad
  \text{for $1 \le j \le n-1$}, \qquad 
  h_{n}^{\omega}=2(h_{n}+h_{\omega(n)}). \label{eq:h-even}
\end{align}

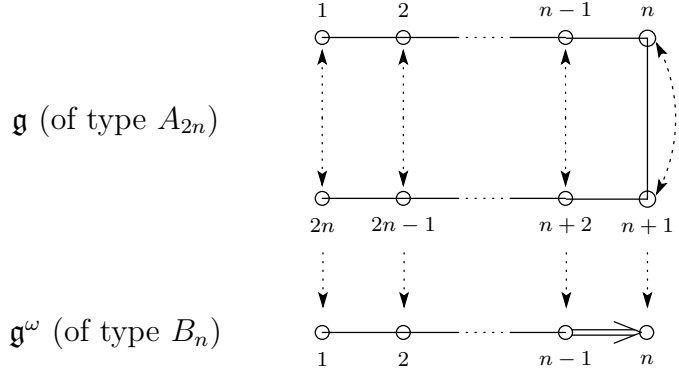
\svsp

\begin{center}
{\scriptsize
\hspace*{-50mm}
%
%\input{fixedmv-fig02.tex}
%
%WinTpicVersion3.08
\unitlength 0.1in
\begin{picture}( 49.7000, 18.2200)(  3.3500,-23.6200)
% CIRCLE 2 0 3 0
% 4 3476 765 3476 730 3476 730 3476 730
% 
\special{pn 8}%
\special{ar 3476 766 36 36  0.0000000 6.2831853}%
% CIRCLE 2 0 3 0
% 4 3476 1605 3476 1570 3476 1570 3476 1570
% 
\special{pn 8}%
\special{ar 3476 1606 36 36  0.0000000 6.2831853}%
% CIRCLE 2 0 3 0
% 4 3896 765 3896 730 3896 730 3896 730
% 
\special{pn 8}%
\special{ar 3896 766 36 36  0.0000000 6.2831853}%
% CIRCLE 2 0 3 0
% 4 3896 1605 3896 1570 3896 1570 3896 1570
% 
\special{pn 8}%
\special{ar 3896 1606 36 36  0.0000000 6.2831853}%
% CIRCLE 2 0 3 0
% 4 4736 765 4736 730 4736 730 4736 730
% 
\special{pn 8}%
\special{ar 4736 766 36 36  0.0000000 6.2831853}%
% CIRCLE 2 0 3 0
% 4 4736 1605 4736 1570 4736 1570 4736 1570
% 
\special{pn 8}%
\special{ar 4736 1606 36 36  0.0000000 6.2831853}%
% LINE 2 2 3 0
% 2 4176 765 4456 765
% 
\special{pn 8}%
\special{pa 4176 766}%
\special{pa 4456 766}%
\special{dt 0.045}%
% LINE 2 2 3 0
% 2 4176 1605 4456 1605
% 
\special{pn 8}%
\special{pa 4176 1606}%
\special{pa 4456 1606}%
\special{dt 0.045}%
% CIRCLE 2 0 3 0
% 4 3476 2305 3476 2270 3476 2270 3476 2270
% 
\special{pn 8}%
\special{ar 3476 2306 36 36  0.0000000 6.2831853}%
% CIRCLE 2 0 3 0
% 4 3896 2305 3896 2270 3896 2270 3896 2270
% 
\special{pn 8}%
\special{ar 3896 2306 36 36  0.0000000 6.2831853}%
% CIRCLE 2 0 3 0
% 4 4736 2305 4736 2270 4736 2270 4736 2270
% 
\special{pn 8}%
\special{ar 4736 2306 36 36  0.0000000 6.2831853}%
% CIRCLE 2 0 3 0
% 4 5156 2305 5156 2270 5156 2270 5156 2270
% 
\special{pn 8}%
\special{ar 5156 2306 36 36  0.0000000 6.2831853}%
% LINE 2 0 3 0
% 8 3476 765 4176 765 4456 765 4736 765 4736 1605 4456 1605 4176 1605 3476 1605
% 
\special{pn 8}%
\special{pa 3476 766}%
\special{pa 4176 766}%
\special{fp}%
\special{pa 4456 766}%
\special{pa 4736 766}%
\special{fp}%
\special{pa 4736 1606}%
\special{pa 4456 1606}%
\special{fp}%
\special{pa 4176 1606}%
\special{pa 3476 1606}%
\special{fp}%
% LINE 2 2 3 0
% 2 4176 2305 4456 2305
% 
\special{pn 8}%
\special{pa 4176 2306}%
\special{pa 4456 2306}%
\special{dt 0.045}%
% LINE 2 0 3 0
% 4 3476 2305 4176 2305 4456 2305 4736 2305
% 
\special{pn 8}%
\special{pa 3476 2306}%
\special{pa 4176 2306}%
\special{fp}%
\special{pa 4456 2306}%
\special{pa 4736 2306}%
\special{fp}%
% STR 2 0 3 0
% 3 3476 555 3476 625 5 0
% $1$
\put(34.7600,-6.2500){\makebox(0,0){$1$}}%
% STR 2 0 3 0
% 3 3896 555 3896 625 5 0
% $2$
\put(38.9600,-6.2500){\makebox(0,0){$2$}}%
% STR 2 0 3 0
% 3 4736 555 4736 625 5 0
% $n-1$
\put(47.3600,-6.2500){\makebox(0,0){$n-1$}}%
% STR 2 0 3 0
% 3 4736 1675 4736 1745 5 0
% $n+2$
\put(47.3600,-17.4500){\makebox(0,0){$n+2$}}%
% STR 2 0 3 0
% 3 3896 1675 3896 1745 5 0
% $2n-1$
\put(38.9600,-17.4500){\makebox(0,0){$2n-1$}}%
% STR 2 0 3 0
% 3 3476 1675 3476 1745 5 0
% $2n$
\put(34.7600,-17.4500){\makebox(0,0){$2n$}}%
% STR 2 0 3 0
% 3 3476 2375 3476 2445 5 0
% $1$
\put(34.7600,-24.4500){\makebox(0,0){$1$}}%
% STR 2 0 3 0
% 3 3896 2375 3896 2445 5 0
% $2$
\put(38.9600,-24.4500){\makebox(0,0){$2$}}%
% STR 2 0 3 0
% 3 4736 2375 4736 2445 5 0
% $n-1$
\put(47.3600,-24.4500){\makebox(0,0){$n-1$}}%
% STR 2 0 3 0
% 3 5156 2375 5156 2445 5 0
% $n$
\put(51.5600,-24.4500){\makebox(0,0){$n$}}%
% VECTOR 2 2 3 0
% 2 3476 1185 3476 835
% 
\special{pn 8}%
\special{pa 3476 1186}%
\special{pa 3476 836}%
\special{dt 0.045}%
\special{sh 1}%
\special{pa 3476 836}%
\special{pa 3456 902}%
\special{pa 3476 888}%
\special{pa 3496 902}%
\special{pa 3476 836}%
\special{fp}%
% VECTOR 2 2 3 0
% 2 3896 1185 3896 835
% 
\special{pn 8}%
\special{pa 3896 1186}%
\special{pa 3896 836}%
\special{dt 0.045}%
\special{sh 1}%
\special{pa 3896 836}%
\special{pa 3876 902}%
\special{pa 3896 888}%
\special{pa 3916 902}%
\special{pa 3896 836}%
\special{fp}%
% VECTOR 2 2 3 0
% 2 4736 1185 4736 835
% 
\special{pn 8}%
\special{pa 4736 1186}%
\special{pa 4736 836}%
\special{dt 0.045}%
\special{sh 1}%
\special{pa 4736 836}%
\special{pa 4716 902}%
\special{pa 4736 888}%
\special{pa 4756 902}%
\special{pa 4736 836}%
\special{fp}%
% VECTOR 2 2 3 0
% 2 4736 1185 4736 1535
% 
\special{pn 8}%
\special{pa 4736 1186}%
\special{pa 4736 1536}%
\special{dt 0.045}%
\special{sh 1}%
\special{pa 4736 1536}%
\special{pa 4756 1468}%
\special{pa 4736 1482}%
\special{pa 4716 1468}%
\special{pa 4736 1536}%
\special{fp}%
% VECTOR 2 2 3 0
% 2 3896 1185 3896 1535
% 
\special{pn 8}%
\special{pa 3896 1186}%
\special{pa 3896 1536}%
\special{dt 0.045}%
\special{sh 1}%
\special{pa 3896 1536}%
\special{pa 3916 1468}%
\special{pa 3896 1482}%
\special{pa 3876 1468}%
\special{pa 3896 1536}%
\special{fp}%
% VECTOR 2 2 3 0
% 2 3476 1185 3476 1535
% 
\special{pn 8}%
\special{pa 3476 1186}%
\special{pa 3476 1536}%
\special{dt 0.045}%
\special{sh 1}%
\special{pa 3476 1536}%
\special{pa 3496 1468}%
\special{pa 3476 1482}%
\special{pa 3456 1468}%
\special{pa 3476 1536}%
\special{fp}%
% CIRCLE 2 0 3 0
% 4 5159 769 5159 727 5159 727 5159 727
% 
\special{pn 8}%
\special{ar 5160 770 42 42  0.0000000 6.2831853}%
% CIRCLE 2 0 3 0
% 4 5159 1609 5159 1567 5159 1567 5159 1567
% 
\special{pn 8}%
\special{ar 5160 1610 42 42  0.0000000 6.2831853}%
% LINE 2 0 3 0
% 6 4739 769 5159 769 5159 769 5159 1609 5159 1609 4739 1609
% 
\special{pn 8}%
\special{pa 4740 770}%
\special{pa 5160 770}%
\special{fp}%
\special{pa 5160 770}%
\special{pa 5160 1610}%
\special{fp}%
\special{pa 5160 1610}%
\special{pa 4740 1610}%
\special{fp}%
% CIRCLE 2 2 3 0
% 4 4459 1189 5229 1539 5229 1539 5229 839
% 
\special{pn 8}%
\special{ar 4460 1190 846 846  5.8565578 5.8707422}%
\special{ar 4460 1190 846 846  5.9132954 5.9274798}%
\special{ar 4460 1190 846 846  5.9700330 5.9842174}%
\special{ar 4460 1190 846 846  6.0267706 6.0409550}%
\special{ar 4460 1190 846 846  6.0835082 6.0976926}%
\special{ar 4460 1190 846 846  6.1402458 6.1544302}%
\special{ar 4460 1190 846 846  6.1969833 6.2111677}%
\special{ar 4460 1190 846 846  6.2537209 6.2679053}%
\special{ar 4460 1190 846 846  6.3104585 6.3246429}%
\special{ar 4460 1190 846 846  6.3671961 6.3813805}%
\special{ar 4460 1190 846 846  6.4239337 6.4381181}%
\special{ar 4460 1190 846 846  6.4806713 6.4948557}%
\special{ar 4460 1190 846 846  6.5374089 6.5515933}%
\special{ar 4460 1190 846 846  6.5941465 6.6083309}%
\special{ar 4460 1190 846 846  6.6508841 6.6650685}%
\special{ar 4460 1190 846 846  6.7076216 6.7098128}%
% VECTOR 2 0 3 0
% 2 5229 839 5208 797
% 
\special{pn 8}%
\special{pa 5230 840}%
\special{pa 5208 798}%
\special{fp}%
\special{sh 1}%
\special{pa 5208 798}%
\special{pa 5220 866}%
\special{pa 5232 846}%
\special{pa 5256 848}%
\special{pa 5208 798}%
\special{fp}%
% VECTOR 2 0 3 0
% 2 5229 1539 5208 1581
% 
\special{pn 8}%
\special{pa 5230 1540}%
\special{pa 5208 1582}%
\special{fp}%
\special{sh 1}%
\special{pa 5208 1582}%
\special{pa 5256 1530}%
\special{pa 5232 1534}%
\special{pa 5220 1512}%
\special{pa 5208 1582}%
\special{fp}%
% LINE 2 0 3 0
% 2 4770 2320 5092 2320
% 
\special{pn 8}%
\special{pa 4770 2320}%
\special{pa 5092 2320}%
\special{fp}%
% LINE 2 0 3 0
% 2 4770 2292 5092 2292
% 
\special{pn 8}%
\special{pa 4770 2292}%
\special{pa 5092 2292}%
\special{fp}%
% LINE 2 0 3 0
% 2 5127 2306 4987 2362
% 
\special{pn 8}%
\special{pa 5128 2306}%
\special{pa 4988 2362}%
\special{fp}%
% LINE 2 0 3 0
% 2 5127 2306 4987 2250
% 
\special{pn 8}%
\special{pa 5128 2306}%
\special{pa 4988 2250}%
\special{fp}%
% STR 2 0 3 0
% 3 5159 559 5159 629 5 0
% $n$
\put(51.5900,-6.2900){\makebox(0,0){$n$}}%
% STR 2 0 3 0
% 3 5159 1679 5159 1749 5 0
% $n+1$
\put(51.5900,-17.4900){\makebox(0,0){$n+1$}}%
% VECTOR 2 2 3 0
% 2 3479 1889 3479 2169
% 
\special{pn 8}%
\special{pa 3480 1890}%
\special{pa 3480 2170}%
\special{dt 0.045}%
\special{sh 1}%
\special{pa 3480 2170}%
\special{pa 3500 2102}%
\special{pa 3480 2116}%
\special{pa 3460 2102}%
\special{pa 3480 2170}%
\special{fp}%
% VECTOR 2 2 3 0
% 4 3899 1889 3899 2169 4739 1889 4739 2169
% 
\special{pn 8}%
\special{pa 3900 1890}%
\special{pa 3900 2170}%
\special{dt 0.045}%
\special{sh 1}%
\special{pa 3900 2170}%
\special{pa 3920 2102}%
\special{pa 3900 2116}%
\special{pa 3880 2102}%
\special{pa 3900 2170}%
\special{fp}%
\special{pa 4740 1890}%
\special{pa 4740 2170}%
\special{dt 0.045}%
\special{sh 1}%
\special{pa 4740 2170}%
\special{pa 4760 2102}%
\special{pa 4740 2116}%
\special{pa 4720 2102}%
\special{pa 4740 2170}%
\special{fp}%
% VECTOR 2 2 3 0
% 2 5159 1889 5159 2169
% 
\special{pn 8}%
\special{pa 5160 1890}%
\special{pa 5160 2170}%
\special{dt 0.045}%
\special{sh 1}%
\special{pa 5160 2170}%
\special{pa 5180 2102}%
\special{pa 5160 2116}%
\special{pa 5140 2102}%
\special{pa 5160 2170}%
\special{fp}%
% STR 2 0 3 0
% 3 2405 1100 2405 1200 5 0
% {\normalsize $\Fg$ (of type $A_{2n}$)}
\put(24.0500,-12.0000){\makebox(0,0){{\normalsize $\Fg$ (of type $A_{2n}$)}}}%
% STR 2 0 3 0
% 3 2405 2210 2405 2310 5 0
% {\normalsize $\Fg^{\omega}$ (of type $B_{n}$)}
\put(24.0500,-23.1000){\makebox(0,0){{\normalsize $\Fg^{\omega}$ (of type $B_{n}$)}}}%
\end{picture}%
}
\end{center}

\vsp\vsp

Let $\ha{A}=(\ha{a}_{ij})_{i,j \in \ha{I}}$ denote 
the Cartan matrix of $\Fg^{\omega}$, with 
index set $\ha{I}:=\bigl\{1,\,2,\,\dots,\,n\bigr\}$.
Let $\ha{W}=\langle \ha{s}_{i} \mid i \in \ha{I} \rangle$ be 
the Weyl group of $\Fg^{\omega}$, where 
$\ha{s}_{i}$, $i \in \ha{I}$, are the simple reflections, and 
let $\ha{e},\,\ha{w}_{0} \in \ha{W}$ denote the unit element and 
the longest element of $\ha{W}$, respectively. 
Set 
\begin{equation*}
\ha{\Gamma}:=\bigl\{\ha{w} \cdot \ha{\Lambda}_{i} \mid 
\ha{w} \in \ha{W},\, i \in \ha{I} \bigr\}, 
\end{equation*}
where $\ha{\Lambda}_{i} \in (\Fh^{\omega})^{\ast}$, 
$i \in \ha{I}$, are the fundamental weights for $\Fg^{\omega}$ 
given by: $\ha{\Lambda}_{i}=a_{i} \Lambda_{i}|_{\Fh^{\omega}}$ 
for $i \in \ha{I}$, with 
%
%%%%%%%%%%%%%
%%% eq:ai %%%
%%%%%%%%%%%%%
%
\begin{equation} \label{eq:ai}
a_{i}:=
 \begin{cases}
 \dfrac{1}{2} & 
 \text{if $\ell=2n$, $n \in \BZ_{\ge 2}$, and $i=n$}, \\[5mm]
 1 & \text{otherwise}.
 \end{cases}
\end{equation}
We define $\ha{\ggms}$ (resp., $\ha{\edge}$) for $\Fg^{\omega}$ 
in the same manner as we defined $\ggms$ (resp., $\edge$) for $\Fg$, 
and denote by $\ha{\pwp}$ the set of pseudo-Weyl polytopes 
$\ha{P}(\ha{\mu}_{\bullet})
 \subset \Fh^{\omega} \cap \Fh_{\BR}$
with GGMS datum $\ha{\mu}_{\bullet}=
(\ha{\mu}_{\ha{w}})_{\ha{w} \in \ha{W}} \in \ha{\ggms}$. 
Also, we define a bijection 
$\ha{D}:\ha{\ggms} \rightarrow \ha{\edge}$ 
as in \S\ref{subsec:MV}; if $\ha{D}(\ha{\mu})=
\ha{M}_{\bullet}=
(\ha{M}_{\ha{\gamma}})_{\ha{\gamma} \in \ha{\Gamma}} 
\in \ha{\edge}$, then 
%
%%%%%%%%%%%%%%%
%%% eq:hapo %%%
%%%%%%%%%%%%%%%
%
\begin{align}
\ha{P}(\ha{\mu}_{\bullet}) & = 
\bigl\{ 
 h \in \Fh^{\omega} \cap \Fh_{\BR} \mid 
 h \ge_{\ha{w}} \ha{\mu}_{\ha{w}} \ 
 \text{for all $\ha{w} \in \ha{W}$}
\bigr\} \nonumber \\
& = 
\bigl\{ 
 h \in \Fh^{\omega} \cap \Fh_{\BR} \mid 
 \pair{\ha{\gamma}}{h} \ge \ha{M}_{\ha{\gamma}} \ 
 \text{for all $\ha{\gamma} \in \ha{\Gamma}$}
\bigr\}, \label{eq:hapo}
\end{align}
where the partial ordering 
$\ge_{\ha{w}}$ on $\Fh^{\omega} \cap \Fh_{\BR}$ 
for each $\ha{w} \in \ha{W}$ is defined by: 
$h \ge_{\ha{w}} h'$ if $\ha{w}^{-1} \cdot h - \ha{w}^{-1} \cdot h' 
\in \sum_{j \in \ha{I}} \BR_{\ge 0} h_{j}^{\omega}$. 
Now, let $\ha{\edge}_{\MV}$ denote 
the subset of $\ha{\edge}$ consisting of all elements 
(called BZ data for $\Fg^{\omega}$) 
$\ha{M}_{\bullet}=
(\ha{M}_{\ha{\gamma}})_{\ha{\gamma} \in \ha{\Gamma}} 
\in \ha{\edge}$, with 
$\ha{M}_{\ha{\gamma}} \in \BZ$ for $\ha{\gamma} \in \ha{\Gamma}$, 
which satisfy the tropical Pl\"ucker relation at $(\ha{w},\,i,\,j)$ 
for each $\ha{w} \in \ha{W}$ and $i,\,j \in \ha{I}$ such that 
$\ha{w}\ha{s}_{i} > \ha{w}$, $\ha{w}\ha{s}_{j} > \ha{w}$, and 
$i \ne j$, where $>$ denotes the (weak) Bruhat ordering on $\ha{W}$. 
Also, let $\ha{\ggms}_{\MV}$ denote the subset of $\ha{\ggms}$ 
corresponding to $\ha{\edge}_{\MV}$ under the bijection 
$\ha{D}:\ha{\ggms} \rightarrow \ha{\edge}$. Set 
\begin{equation*}
\ha{\mv}:=
 \bigl\{\ha{P}(\ha{M}_{\bullet}) \mid 
 \ha{M}_{\bullet} \in \ha{\edge}_{\MV} \bigr\}=
 \bigl\{\ha{P}(\ha{\mu}_{\bullet}) \mid 
 \ha{\mu}_{\bullet} \in \ha{\ggms}_{\MV} \bigr\},
\end{equation*}
and call an element of $\ha{\mv}$ 
an MV polytope for $\Fg^{\omega}$. 
We endow $\ha{\mv}$ with a crystal structure 
in the same manner as we did for $\mv$, 
so that we have an isomorphism of crystals 
$\ha{\Psi}:\ha{\mv} \stackrel{\sim}{\rightarrow} \ha{\CB}(\infty)$, 
where $\ha{\CB}(\infty)$ denotes the crystal basis for 
the negative part $U_{q}^{-}((\Fg^{\omega})^{\vee})$ of 
the quantized universal enveloping algebra 
$U_{q}((\Fg^{\omega})^{\vee})$ associated to 
the (Langlands) dual Lie algebra 
$(\Fg^{\omega})^{\vee}$ of $\Fg^{\omega}$.
For each $j \in \ha{I}$, 
we denote by $\ha{f}_{j}$ 
(resp., $\ha{e}_{j}$) 
the lowering (resp., raising) 
Kashiwara operator on the crystal $\ha{\mv}$. 
Let $\ha{u}_{\infty} \in \ha{\CB}(\infty)$ denote 
the element of $\ha{\CB}(\infty)$ corresponding to 
the identity element $1 \in U_{q}^{-}((\Fg^{\omega})^{\vee})$, and  
$\ha{P}^{0} \in \ha{\mv}$ the MV polytope 
which is sent to $\ha{u}_{\infty}$ under the isomorphism 
$\ha{\Psi}:\ha{\mv} \stackrel{\sim}{\rightarrow} \ha{\CB}(\infty)$
(see Remark~\ref{rem:binf}). 

It is well-known (for a proof, see, e.g., \cite[Corollary~3.4]{FRS}) 
that there exists a group isomorphism 
$\Theta:\ha{W} \stackrel{\sim}{\rightarrow} W^{\omega}$ such that 
$\Theta(\ha{s}_{i})=s_{i}^{\omega}$ for all $i \in \ha{I}$, 
where
%
%%%%%%%%%%%%%%
%%% eq:sjo %%%
%%%%%%%%%%%%%%
%
\begin{equation} \label{eq:sjo}
s_{i}^{\omega}:=
\begin{cases}
s_{i}s_{\omega(i)}=s_{\omega(i)}s_{i} & 
 \text{if $1 \le i \le n-1$}, \\[1.5mm]
s_{n} &
 \text{if $\ell=2n-1$, $n \in \BZ_{\ge 2}$, and $i=n$}, \\[1.5mm]
s_{n}s_{\omega(n)}s_{n}=
s_{\omega(n)}s_{n}s_{\omega(n)} & 
 \text{if $\ell=2n$, $n \in \BZ_{\ge 2}$, and $i=n$}.
\end{cases}
\end{equation}
%
%%%%%%%%%%%%%%%%%
%%% rem:theta %%%
%%%%%%%%%%%%%%%%%
%
\begin{rem} \label{rem:theta}
(1) Recall that $\Fh^{\omega}$ is stable under 
the action of $W^{\omega}$ (see Remark~\ref{rem:omega}\,(3)), and 
that $\Fh^{\omega}$ is the Cartan subalgebra of $\Fg^{\omega}$.
It is easy to check that 
%
%%%%%%%%%%%%%%%%
%%% eq:theta %%%
%%%%%%%%%%%%%%%%
%
\begin{equation} \label{eq:theta}
\Theta(\ha{w}) \cdot h = \ha{w} \cdot h \quad 
\text{for all $\ha{w} \in \ha{W}$ and $h \in \Fh^{\omega}$}.
\end{equation}

\noindent
(2) It follows from \eqref{eq:theta} that 
for $h,\,h' \in \Fh^{\omega}$ and $\ha{w} \in \ha{W}$, 
%
%%%%%%%%%%%%%%%%
%%% eq:order %%%
%%%%%%%%%%%%%%%%
%
\begin{equation} \label{eq:order}
\text{$h \ge_{\ha{w}} h'$ \ 
if and only if \ 
$h \ge_{\Theta(\ha{w})} h'$.}
\end{equation}

\noindent
(3) Let $\ha{w} \in \ha{W}$, and set 
$w:=\Theta(\ha{w}) \in W^{\omega}$. 
We deduce from \cite[Lemma~3.2.1]{NS1} that 
for each $j \in \ha{I}$, 
\begin{equation*}
\ha{s}_{j}\ha{w} < \ha{w}
\quad \Longleftrightarrow \quad
s_{j}^{\omega}w < w \nonumber
\quad \Longleftrightarrow \quad 
\text{$s_{j}w < w$ and $s_{\omega(j)}w < w$}.
\end{equation*}
\end{rem}

%==============================%
%     START SUBSECTION 0204    %
%==============================%
%
\subsection{Action of the diagram automorphism $\omega$ on $\mv$.}
\label{subsec:damv}

We keep the notation and assumptions of \S\ref{subsec:da}. 
For an element $\mu_{\bullet}=(\mu_{w})_{w \in W} \in \ggms$, 
we define $\omega(\mu_{\bullet})$ to be a collection 
$(\mu_{w}')_{w \in W}$ of elements in $\Fh_{\BR}$ given by: 
$\mu_{w}'=\omega(\mu_{\omega(w)})$ for $w \in W$. 
Then, using Remark~\ref{rem:omega}\,(1) and (5), 
we can easily check that 
$\omega(\mu_{\bullet}) \in \ggms$ 
for all $\mu_{\bullet} \in \ggms$. 
%
%%%%%%%%%%%%%%%
%%% rem:daM %%%
%%%%%%%%%%%%%%%
%
\begin{rem} \label{rem:daM} 
Let $\mu_{\bullet}=(\mu_{w})_{w \in W} \in \ggms$. 
Set $(M_{\gamma})_{\gamma \in \Gamma}:=D(\mu_{\bullet}) \in \edge$ and 
$(M_{\gamma}')_{\gamma \in \Gamma}:=D(\omega(\mu_{\bullet})) \in \edge$. 
Then we have $M_{\gamma}'=M_{\omega(\gamma)}$ 
for all $\gamma \in \Gamma$. Indeed, 
using Remark~\ref{rem:omega}\,(4), we have 
\begin{align*}
M_{w \cdot \Lambda_{i}}' & = 
 \pair{w \cdot \Lambda_{i}}{\mu_{w}'} = 
 \pair{w \cdot \Lambda_{i}}{\omega^{-1}(\mu_{\omega(w)})} = 
 \pair{\omega(w \cdot \Lambda_{i})}{\mu_{\omega(w)}} \\
& =
 \pair{\omega(w) \cdot \Lambda_{\omega(i)}}{\mu_{\omega(w)}} = 
 M_{\omega(w) \cdot \Lambda_{\omega(i)}} = 
 M_{\omega(w \cdot \Lambda_{i})}
\end{align*}
for each $w \in W$ and $i \in I$. 
\end{rem}

Now we set 
\begin{equation*}
\ggms^{\omega} 
 :=\bigl\{ \mu_{\bullet} \in \ggms \mid 
   \omega(\mu_{\bullet})=\mu_{\bullet}\bigr\}
\quad \text{and} \quad 
\edge^{\omega}:=D(\ggms^{\omega}) \subset \edge.
\end{equation*}
The next lemma follows immediately from 
the definition of the action of $\omega$ on $\ggms$ and 
Remark~\ref{rem:daM}. 
%
%%%%%%%%%%%%%%%%%%
%%% lem:bzmvda %%%
%%%%%%%%%%%%%%%%%%
%
\begin{lem} \label{lem:bzmvda}
{\rm (1)} Let $\mu_{\bullet}=(\mu_{w})_{w \in W} \in \ggms$. 
Then, $\mu_{\bullet} \in \ggms^{\omega}$ if and only if 
$\omega(\mu_{w})=\mu_{\omega(w)}$ for all $w \in W$. 
In particular, if $\mu_{\bullet}=(\mu_{w})_{w \in W} \in \ggms^{\omega}$, 
then $\mu_{w} \in \Fh^{\omega}$ for all $w \in W^{\omega}$. 

\noindent 
{\rm (2)} 
Let $M_{\bullet}=(M_{\gamma})_{\gamma \in \Gamma} \in \edge$. 
Then, $M_{\bullet} \in \edge^{\omega}$ if and only if 
$M_{\omega(\gamma)}=M_{\gamma}$ for all $\gamma \in \Gamma$. 
\end{lem}

Let $P=P(\mu_{\bullet})$ be a pseudo-Weyl polytope 
with GGMS datum $\mu_{\bullet}=(\mu_{w})_{w \in W}$. 
Then it follows from \eqref{eq:po02} and 
Remark~\ref{rem:omega}\,(5) that 
the image $\omega(P)=\bigl\{\omega(h) \mid 
h \in P\bigr\}$ of $P$ (as a set) under $\omega \in \GL(\Fh)$ 
is identical to the pseudo-Weyl polytope 
$P(\omega(\mu_{\bullet})) \in \pwp$. For this reason, 
we define an action of $\omega$ on the set 
$\pwp=\bigl\{P(\mu_{\bullet}) \mid \mu_{\bullet} \in \ggms \bigr\}$ 
of pseudo-Weyl polytopes by: 
$\omega(P(\mu_{\bullet}))=P(\omega(\mu_{\bullet}))$
for $\mu_{\bullet} \in \ggms$. Since 
$\omega(P(\mu_{\bullet}))=P(\omega(\mu_{\bullet}))$ 
for $\mu_{\bullet} \in \ggms$, it follows that 
$\omega(P(\mu_{\bullet}))=P(\mu_{\bullet})$ 
if and only if $\omega(\mu_{\bullet})=\mu_{\bullet}$. 
Therefore, we have 
\begin{equation*}
\pwp^{\omega} :=
 \bigl\{P \in \pwp \mid \omega(P)=P\bigr\} = 
 \bigl\{ 
  P(\mu_{\bullet}) \mid 
  \mu_{\bullet} \in \ggms^{\omega}
 \bigr\}
  = 
 \bigl\{ 
  P(M_{\bullet}) \mid 
  M_{\bullet} \in \edge^{\omega}
 \bigr\}.
\end{equation*}

Using Remark~\ref{rem:daM}, along with 
Remark~\ref{rem:omega}\,(1),(4), 
we can check that the subset $\ggms_{\MV}$ of 
$\ggms$ is stable under the action of $\omega$ on $\ggms$, 
which implies that the set $\mv \subset \pwp$ 
of MV polytopes for $\Fg$ is stable under the action of 
$\omega$ on $\pwp$. We set 
\begin{equation*}
\ggms_{\MV}^{\omega}:=
\ggms_{\MV} \cap \ggms^{\omega}
\qquad \text{and} \qquad 
\edge_{\MV}^{\omega}:=
\edge_{\MV} \cap \edge^{\omega}=
D(\ggms_{\MV}^{\omega}), 
\end{equation*}
\begin{equation*}
\mv^{\omega} := \mv \cap \pwp^{\omega}=
 \bigl\{ 
  P(\mu_{\bullet}) \mid 
  \mu_{\bullet} \in \ggms_{\MV}^{\omega}
 \bigr\}
  = 
 \bigl\{ 
  P(M_{\bullet}) \mid 
  M_{\bullet} \in \edge_{\MV}^{\omega}
 \bigr\}.
\end{equation*}

%==============================%
%     START SUBSECTION 0205    %
%==============================%
%
\subsection{MV polytopes for $\Fg$ fixed by $\omega$ and 
MV polytopes for $\Fg^{\omega}$.}
\label{subsec:phi}

Recall that $\Fg$ is of type $A_{\ell}$, $\ell \ge 3$. 
Namely, $\Fg$ is either of type $A_{\ell}$ with $\ell=2n-1$, 
$n \in \BZ_{\ge 2}$, or of type $A_{\ell}$ with $\ell=2n$, 
$n \in \BZ_{\ge 2}$. If $\ell=2n-1$, $n \in \BZ_{\ge 2}$ 
(resp., $\ell=2n$, $n \in \BZ_{\ge 2}$), then $\Fg^{\omega}$ 
is of type $C_{n}$ (resp., of type $B_{n}$). 

For $\mu_{\bullet}=(\mu_{w})_{w \in W} \in \ggms^{\omega}$, 
we define $\Phi(\mu_{\bullet})$ to be a collection 
$(\ha{\mu}_{\ha{w}})_{\ha{w} \in \ha{W}}$ of elements 
in $\Fh^{\omega} \cap \Fh_{\BR}$ given by: 
$\ha{\mu}_{\ha{w}}=\mu_{\Theta(\ha{w})}$ 
for $\ha{w} \in \ha{W}$.
Using Remark~\ref{rem:theta}\,(2), along with 
Lemma~\ref{lem:bzmvda}\,(1) and the fact that 
$\Theta(\ha{w}_{0})=w_{0}$, we obtain the following lemma.
%
%%%%%%%%%%%%%%%%%%
%%% lem:phi-mu %%%
%%%%%%%%%%%%%%%%%%
%
\begin{lem} \label{lem:phi-mu}
We have $\Phi(\mu_{\bullet}) \in \ha{\ggms}$ 
for all $\mu_{\bullet} \in \ggms^{\omega}$. 
\end{lem}
%
%%%%%%%%%%%%%%%%%
%%% rem:phi-M %%%
%%%%%%%%%%%%%%%%%
%
\begin{rem} \label{rem:phi-M}
Let $\mu_{\bullet}=(\mu_{w})_{w \in W} \in \ggms^{\omega}$, 
and set $(\ha{\mu}_{\ha{w}})_{\ha{w} \in \ha{W}}:=
\Phi(\mu_{\bullet}) \in \ha{\ggms}$. Also, we set 
$(M_{\gamma})_{\gamma \in \Gamma}:=D(\mu_{\bullet}) \in \edge^{\omega}$ and 
$(\ha{M}_{\ha{\gamma}})_{\ha{\gamma} \in \ha{\Gamma}}:=
 \ha{D}(\Phi(\mu_{\bullet})) \in \ha{\edge}$. 
Then, for each $\ha{w} \in \ha{W}$ and $i \in \ha{I}$, 
we have $\ha{M}_{\ha{w} \cdot \ha{\Lambda}_{i}}=
 a_{i} M_{\Theta(w) \cdot \Lambda_{i}}$, where 
$a_{i}$ is as defined in \eqref{eq:ai}. 
Indeed, we have
\begin{align*}
\ha{M}_{\ha{w} \cdot \ha{\Lambda}_{i}}
  & = \pair{\ha{w} \cdot \ha{\Lambda}_{i}}{\ha{\mu}_{\ha{w}}}
    = \pair{\ha{\Lambda}_{i}}{\ha{w}^{-1} \cdot \ha{\mu}_{\ha{w}}}
    = \pair{\ha{\Lambda}_{i}}{\ha{w}^{-1} \cdot \mu_{\Theta(\ha{w})}} \\
  & = \pair{\ha{\Lambda}_{i}}{\Theta(\ha{w}^{-1}) \cdot \mu_{\Theta(\ha{w})}}
    \quad \text{by \eqref{eq:theta}}.
\end{align*}
Therefore, noting that 
$\ha{\Lambda}_{i}=a_{i} \Lambda_{i}|_{\Fh^{\omega}}$, 
we obtain 
\begin{align*}
\ha{M}_{\ha{w} \cdot \ha{\Lambda}_{i}}
 & = \pair{\ha{\Lambda}_{i}}{\Theta(\ha{w}^{-1}) \cdot \mu_{\Theta(\ha{w})}}
   = a_{i}\pair{\Lambda_{i}}{\Theta(\ha{w}^{-1}) \cdot \mu_{\Theta(\ha{w})}} \\
 & = a_{i}\pair{\Theta(\ha{w}) \cdot \Lambda_{i}}{ \mu_{\Theta(\ha{w})}} 
   = a_{i} M_{\Theta(\ha{w}) \cdot \Lambda_{i}}.
\end{align*}
\end{rem}

By Lemma~\ref{lem:phi-mu}, 
we can define a map (also denoted by) 
$\Phi:\pwp^{\omega} \rightarrow \ha{\pwp}$ by: 
$\Phi(P(\mu_{\bullet}))=\ha{P}(\Phi(\mu_{\bullet}))$ 
for $\mu_{\bullet} \in \ggms^{\omega}$. 
If $\mu_{\bullet}=(\mu_{w})_{w \in W} \in \ggms^{\omega}$ and 
$\Phi(\mu_{\bullet})=
 (\ha{\mu}_{\ha{w}})_{\ha{w} \in \ha{W}} \in \ha{\ggms}$, 
then it follows from \eqref{eq:order} that 
%
%%%%%%%%%%%%%%%%%
%%% eq:phi-po %%%
%%%%%%%%%%%%%%%%%
%
\begin{align}
\Phi(P(\mu_{\bullet}))=\ha{P}(\Phi(\mu_{\bullet}))
& = 
  \bigl\{ 
  h \in \Fh^{\omega} \mid h \ge_{\ha{w}} \ha{\mu}_{\ha{w}} \ 
  \text{for all $\ha{w} \in \ha{W}$}
  \bigr\} \nonumber \\
& = 
  \bigl\{ 
  h \in \Fh^{\omega} \mid 
  h \ge_{\Theta(\ha{w})} \mu_{\Theta(\ha{w})} \ 
  \text{for all $\ha{w} \in \ha{W}$}
  \bigr\} \nonumber \\
& = 
  \bigl\{ 
  h \in \Fh^{\omega} \mid h \ge_{w} \mu_{w} \ 
  \text{for all $w \in W^{\omega}$}
  \bigr\}. \label{eq:phi-po}
\end{align}

\begin{rem}
Let $\mu_{\bullet}=(\mu_{w})_{w \in W} \in \ggms^{\omega}$ and 
$\Phi(\mu_{\bullet})=(\ha{\mu}_{\ha{w}})_{\ha{w} \in \ha{W}} 
\in \ha{\ggms}$. Then we see from \eqref{eq:phi-po} that 
$P(\mu_{\bullet}) \cap \Fh^{\omega} \subset \Phi(P(\mu_{\bullet}))$. 
Also, since $\Phi(P(\mu_{\bullet}))=\ha{P}(\Phi(\mu_{\bullet}))$ is 
the convex hull in $\Fh^{\omega} \cap \Fh_{\BR}$ of 
the collection $\Phi(\mu_{\bullet})=(\ha{\mu}_{\ha{w}})_{\ha{w} \in \ha{W}}$ 
(see Remark~\ref{rem:vertex}) and 
$\ha{\mu}_{\ha{w}}=\mu_{\Theta(\ha{w})} \in 
P(\mu_{\bullet}) \cap \Fh^{\omega}$ 
for all $\ha{w} \in \ha{W}$, it follows that 
$\Phi(P(\mu_{\bullet}))=\ha{P}(\Phi(\mu_{\bullet})) 
\subset P(\mu_{\bullet}) \cap \Fh^{\omega}$. Therefore, 
we conclude that $\Phi(P(\mu_{\bullet}))=
P(\mu_{\bullet}) \cap \Fh^{\omega}$. 
In addition, if $\mu_{w} \in \Fh^{\omega}$ for some $w \in W$, 
then $\mu_{w}$ is a vertex of the convex polytope 
$P(\mu_{\bullet}) \cap \Fh^{\omega}=\Phi(P(\mu_{\bullet}))$, 
so that $\mu_{w}=\ha{\mu}_{\ha{w}}=\mu_{\Theta(\ha{w})}$ 
for some $\ha{w} \in \ha{W}$. 
\end{rem}
%
%%%%%%%%%%%%%%%%%%%
%%% prop:phi-bz %%%
%%%%%%%%%%%%%%%%%%%
%
\begin{prop} \label{prop:phi-bz}
We have $\Phi(\mu_{\bullet}) \in \ha{\ggms}_{\MV}$ 
for all $\mu_{\bullet} \in \ggms_{\MV}^{\omega}$. 
\end{prop}

The proof of this proposition will be given in 
\S\ref{subsec:phi-bz}. It follows from this proposition that 
$\Phi(\mv^{\omega}) \subset \ha{\mv}$. 
Hence the restriction of the map 
$\Phi:\pwp^{\omega} \rightarrow \ha{\pwp}$ to 
$\mv^{\omega}$ gives rise to a map 
$\Phi:\mv^{\omega} \rightarrow \ha{\mv}$.

Now we define operators $f_{j}^{\omega}$, $j \in \ha{I}$, 
on $\mv$ by:
%
%%%%%%%%%%%%%
%%% eq:ok %%%
%%%%%%%%%%%%%
%
\begin{equation} \label{eq:ok}
f_{j}^{\omega}=
\begin{cases}
f_{j}f_{\omega(j)} & 
 \text{if $1 \le j \le n-1$}, \\[3mm]
f_{n} &
 \text{if $\ell=2n-1$, $n \in \BZ_{\ge 2}$, and $j=n$}, \\[3mm]
f_{n}f_{\omega(n)}^{2}f_{n} & 
 \text{if $\ell=2n$, $n \in \BZ_{\ge 2}$, and $j=n$}.
\end{cases}
\end{equation}
%
%%%%%%%%%%%%%%
%%% rem:ok %%%
%%%%%%%%%%%%%%
%
\begin{rem} \label{rem:ok}
Since $\mv$ is isomorphic to $\CB(\infty)$ as a crystal for $\Fg^{\vee}$, 
we deduce from \cite[Proposition~7.4.1]{Kas} that 
$f_{j}f_{\omega(j)}=f_{\omega(j)}f_{j}$ if $1 \le j \le n-1$, 
and that 
$f_{n}f_{\omega(n)}^{2}f_{n}=f_{\omega(n)}f_{n}^{2}f_{\omega(n)}$ 
if $\ell=2n$, $n \in \BZ_{\ge 2}$. 
\end{rem}
%
%%%%%%%%%%%%%%%%%%%%
%%% thm:fixed-mv %%%
%%%%%%%%%%%%%%%%%%%%
%
\begin{thm} \label{thm:fixed-mv}
{\rm (1)} 
The subset $\mv^{\omega}$ of $\mv$ is stable under the operators 
$f_{j}^{\omega}$ for all $j \in \ha{I}$. 

\noindent {\rm (2)}
Each element $P \in \mv^{\omega}$ is of the form 
$P=f_{j_{1}}^{\omega}f_{j_{2}}^{\omega} \cdots f_{j_{k}}^{\omega}P^{0}$ 
for some  $j_{1},\,j_{2},\,\dots,\,j_{k} \in \ha{I}$. 

\noindent {\rm (3)}
The map 
$\Phi:\mv^{\omega} \rightarrow \ha{\mv}$ is a unique 
bijection such that 
$\Phi(P^{0})=\ha{P}^{0}$, and such that 
$\Phi \circ f_{j}^{\omega}= \ha{f}_{j} \circ \Phi$ 
for all $j \in \ha{I}$. 
\end{thm}

The proof of this theorem will be given in \S\ref{subsec:fixed-mv}. 

\begin{rem}
The existence of a bijection $\mv^{\omega} \rightarrow \ha{\mv}$ 
satisfying the conditions of part~(3) of 
Theorem~\ref{thm:fixed-mv}, follows immediately from 
\cite[Theorem~3.4.1]{NS2} (see also \cite[Theorem~14.4.9]{L} 
for the case in which $\Fg^{\vee}$ is not of type $A_{2n}$); 
note that the orbit Lie algebra associated to $\Fg^{\vee}$ is 
precisely the dual Lie algebra $(\Fg^{\omega})^{\vee}$ 
of $\Fg^{\omega}$. However, for our purpose, 
we need a more explicit description of the bijection 
in terms of polytopes, such as the one given in this subsection.
\end{rem}

%==============================%
%     START SUBSECTION 0206    %
%==============================%
%
\subsection{Proof of Proposition~\ref{prop:phi-bz}.}
\label{subsec:phi-bz}

This subsection is devoted to the proof of 
Proposition~\ref{prop:phi-bz}. 
We keep the notation and assumptions of \S\ref{subsec:phi}. 
We know from Lemma~\ref{lem:phi-mu} that 
if $\mu_{\bullet}=(\mu_{w})_{w \in W} \in \ggms^{\omega}$, 
then $\ha{\mu}_{\bullet}:=
\Phi(\mu_{\bullet})=(\ha{\mu}_{\ha{w}})_{\ha{w} \in \ha{W}}$ 
is an element of $\ha{\ggms}$.
In this subsection, by setting $\ha{M}_{\bullet}:=
\ha{D}(\ha{\mu}_{\bullet})=
(\ha{M}_{\ha{\gamma}})_{\ha{\gamma} \in \ha{\Gamma}}  \in \ha{\edge}$, we first prove 
that $\ha{M}_{\ha{\gamma}} \in \BZ$ 
for all $\ha{\gamma} \in \ha{\Gamma}$, and then prove that 
$\ha{M}_{\bullet}$ satisfies the tropical Pl\"ucker relations. 

We begin with the following simple lemma. 
%
%%%%%%%%%%%%%%%
%%% lem:psi %%%
%%%%%%%%%%%%%%%
%
\begin{lem} \label{lem:psi}
Let $P=P(\mu_{\bullet}) \in \mv$ be an MV polytope 
with GGMS datum $\mu_{\bullet}=(\mu_{w})_{w \in W} \in \ggms_{\MV}$, and 
$\bi=(i_{1},\,i_{2},\,\dots,\,i_{m})$ a reduced word 
for $w_{0} \in W$. Then we have $\psi_{\omega(\bi)}(\omega(P))=\psi_{\bi}(P)$, 
where $\omega(\bi):=(\omega(i_{1}),\,\omega(i_{2}),\,\dots,\,\omega(i_{m}))$ 
is also a reduced word for $w_{0} \in W$. 
\end{lem}

\begin{proof}
If we write $\psi_{\bi}(P) \in \BZ_{\ge 0}^{m}$ as 
$\psi_{\bi}(P)=(L_{1},\,L_{2},\,\dots,\,L_{m}) \in \BZ_{\ge 0}^{m}$, 
then by the definition, we have $\mu_{w_{k}^{\bi}}-\mu_{w_{k-1}^{\bi}}=
L_{k} w_{k-1}^{\bi} \cdot h_{i_{k}}$, with 
$w_{k}^{\bi}=s_{i_{1}}s_{i_{2}} \cdots s_{i_{k}}$, for $1 \le k \le m$. 
Similarly, if we write $\psi_{\omega(\bi)}(\omega(P)) \in \BZ_{\ge 0}^{m}$
as $\psi_{\omega(\bi)}(\omega(P))=
(L_{1}',\,L_{2}',\,\dots,\,L_{m}') \in \BZ_{\ge 0}^{m}$, 
and denote $\omega(\mu_{\bullet}) \in \ggms_{\MV}$ by 
$\mu_{\bullet}'=(\mu_{w}')_{w \in W}$, then we have
$\mu_{w_{k}^{\omega(\bi)}}'-
 \mu_{w_{k-1}^{\omega(\bi)}}'=
L_{k}' w_{k-1}^{\omega(\bi)} \cdot h_{\omega(i_{k})}$, 
with $w_{k}^{\omega(\bi)}=s_{\omega(i_{1})}s_{\omega(i_{2})} \cdots 
s_{\omega(i_{k})}$, for $1 \le k \le m$. Because 
$\mu_{w_{k}^{\omega(\bi)}}'= 
 \omega(\mu_{\omega(w_{k}^{\omega(\bi)})})=
 \omega(\mu_{w_{k}^{\bi}})$ for $1 \le k \le m$ 
by the definition of $\omega(\mu_{\bullet})$, we have
\begin{align*}
L_{k}' w_{k-1}^{\omega(\bi)} \cdot h_{\omega(i_{k})} 
& = 
\mu_{w_{k}^{\omega(\bi)}}'-
\mu_{w_{k-1}^{\omega(\bi)}}' 
=
\omega(\mu_{w_{k}^{\bi}})-
\omega(\mu_{w_{k-1}^{\bi}}) \\
 & = 
\omega(\mu_{w_{k}^{\bi}}- \mu_{w_{k-1}^{\bi}})=
\omega(L_{k}w_{k-1}^{\bi} \cdot h_{i_{k}}) \\
 & =
L_{k} \omega(w_{k-1}^{\bi}) \cdot \omega(h_{i_{k}})= 
L_{k} w_{k-1}^{\omega(\bi)} \cdot h_{\omega(i_{k})}, 
\end{align*}
from which it follows that $L_{k}=L_{k}'$ for all $1 \le k \le m$. 
This proves the lemma.
\end{proof}

Let $P=P(\mu_{\bullet}) \in \mv$ be an MV polytope with 
GGMS datum $\mu_{\bullet}=(\mu_{w})_{w \in W} \in \ggms_{\MV}$. 
Then, by Lemma~\ref{lem:psi}, we have 
$\psi_{\omega(\bi)}(\omega(P))=\psi_{\bi}(P)$ 
for a reduced word $\bi$ for $w_{0} \in W$. 
Since $\psi_{\omega(\bi)}:\mv \rightarrow \BZ_{\ge 0}^{m}$ is 
a bijection, it follows that
%
%%%%%%%%%%%%%%
%%% eq:psi %%%
%%%%%%%%%%%%%%
%
\begin{equation} \label{eq:psi}
\omega(P)=P
\quad \Longleftrightarrow \quad
\psi_{\omega(\bi)}(P)=\psi_{\bi}(P).
\end{equation}

Now we recall from \cite[Lemma~3.2.1]{NS1} that if 
$\ha{w}_{0}=\ha{s}_{j_{1}}\ha{s}_{j_{2}} \cdots \ha{s}_{j_{\ha{m}}}$, 
$j_{1},\,j_{2},\,\dots,\,j_{\ha{m}} \in \ha{I}$, is a reduced decomposition 
of the longest element $\ha{w}_{0}$ of $\ha{W}$, then 
$w_{0}=s_{j_{1}}^{\omega}s_{j_{2}}^{\omega} \cdots s_{j_{\ha{m}}}^{\omega}$ 
is a reduced decomposition of the longest element $w_{0}$ of $W$, where 
$s_{j}^{\omega}$, $j \in \ha{I}$, are as defined in \eqref{eq:sjo}. 
Using this fact, to each reduced word 
$\bj=(j_{1},\,j_{2},\,\dots,\,j_{\ha{m}})$ 
for $\ha{w}_{0} \in \ha{W}$, we associate a reduced word 
$\bi=(i_{1},\,i_{2},\,\dots,\,i_{m})$ for $w_{0} \in W$ 
as follows. For each $1 \le k \le \ha{m}$, we define 
elements $i_{l}^{(k)} \in I$, $1 \le l \le N_{k}$, 
where $N_{k}=\ell(s_{j_{k}}^{\omega})$, by: 
\begin{equation*}
\begin{cases}
\text{$i_{1}^{(k)}=j_{k}$,
$i_{2}^{(k)}=\omega(j_{k})$, 
with $N_{k}=2$, if $1 \le j_{k} \le n-1$}, \\[3mm]
\text{$i_{1}^{(k)}=j_{k}$,
with $N_{k}=1$, if $\ell=2n-1$, 
$n \in \BZ_{\ge 2}$, and $j_{k}=n$}, \\[3mm]
\text{$i_{1}^{(k)}=j_{k}$, $i_{2}^{(k)}=\omega(j_{k})$, $i_{3}^{(k)}=j_{k}$, 
with $N_{k}=3$, if $\ell=2n$, $n \in \BZ_{\ge 2}$, and $j_{k}=n$}. 
\end{cases}
\end{equation*}
Then we set
\begin{align*}
\bi 
 & = (i_{1},\,i_{2},\,\dots,\,i_{m}) \\
 & := \bigl(
      i_{1}^{(1)},\,\dots,\,i_{N_{1}}^{(1)},\ 
      i_{1}^{(2)},\,\dots,\,i_{N_{2}}^{(2)},\ \dots,\ 
      i_{1}^{(\ha{m})},\,\dots,\,i_{N_{\ha{m}}}^{(\ha{m})}
     \bigr) \in \BZ_{\ge 0}^{m},
\end{align*}
and call it the canonical reduced word for $w_{0} \in W$ 
associated to $\bj$. Recall that 
$\omega(\bi)=(\omega(i_{1}),\,\omega(i_{2}),\,\dots,\,\omega(i_{m}))$ 
is also a reduced word for $w_{0} \in W$. 
%
%%%%%%%%%%%%%%
%%% prop:L %%%
%%%%%%%%%%%%%%
%
\begin{prop} \label{prop:L}
Let $P=P(\mu_{\bullet}) \in \mv^{\omega}$ be an MV polytope 
with GGMS datum $\mu_{\bullet}=(\mu_{w})_{w \in W} \in \ggms_{\MV}$. 
Let $\bj=(j_{1},\,j_{2},\,\dots,\,j_{\ha{m}})$ be a reduced word 
for $\ha{w}_{0} \in \ha{W}$, and 
$\bi=(i_{1},\,i_{2},\,\dots,\,i_{m})$ 
the associated canonical reduced word for $w_{0} \in W$. 
If we write $\psi_{\bi}(P) \in \BZ_{\ge 0}^{m}$ as 
\begin{align*}
\psi_{\bi}(P)
 & = (L_{1},\,L_{2},\,\dots,\,L_{m}) \\
 & = \bigl(
      L_{1}^{(1)},\,\dots,\,L_{N_{1}}^{(1)},\ 
      L_{1}^{(2)},\,\dots,\,L_{N_{2}}^{(2)},\ \dots,\ 
      L_{1}^{(\ha{m})},\,\dots,\,L_{N_{\ha{m}}}^{(\ha{m})}
     \bigr) \in \BZ_{\ge 0}^{m},
\end{align*}
then we have $L_{1}^{(k)}=\cdots=L_{N_{k}}^{(k)}$ 
for all $1 \le k \le \ha{m}$.
\end{prop}

\begin{proof}
We prove the equalities 
$L_{1}^{(k)}=\cdots=L_{N_{k}}^{(k)}$ in the case that
$\ell=2n$, $n \in \BZ_{\ge 2}$, and $j_{k}=n$ (hence $N_{k}=3$); 
the proofs for the other cases are similar (or, even simpler). 
For simplicity of notation, we further assume that 
$k=1$ and $n=2$; we have 
\begin{equation*}
\bj=(2,\,1,\,2,\,1)
\quad \text{and} \quad
\bi=(2,\,3,\,2,\,1,\,4,\,2,\,3,\,2,\,1,\,4),
\end{equation*}
with $\ha{m}=4$ and $m=10$, and 
$(L_{1},\,L_{2},\,L_{3})=(L_{1}^{(1)},\,L_{2}^{(1)},\,L_{3}^{(1)})$. 
If we take a reduced word 
\begin{equation*}
\bi'=(3,\,2,\,3,\,1,\,4,\,2,\,3,\,2,\,1,\,4)
\end{equation*}
for $w_{0} \in W$, then 
the bijection $\psi_{\bi'} \circ \psi_{\bi}^{-1} : 
\BZ_{\ge 0}^{10} \rightarrow \BZ_{\ge 0}^{10}$ is identical 
to the transition map $R^{\bi'}_{\bi} : 
\BZ_{\ge 0}^{10} \rightarrow \BZ_{\ge 0}^{10}$. 
Therefore, by setting 
\begin{equation*}
\psi_{\bi'}(P)=(L_{1}',\,L_{2}',\,\dots,\,L_{10}')
\in \BZ_{\ge 0}^{10},
\end{equation*}
we obtain from Remark~\ref{rem:trans} the following relations 
(note that $a_{23}=a_{32}=-1$ in our case): 
%
%%%%%%%%%%%%%%
%%% eq:L01 %%%
%%%%%%%%%%%%%%
%
\begin{equation} \label{eq:L01}
\begin{array}{l}
L_{1}'=L_{2}+L_{3}-p, \quad
L_{2}'=p, \\[3mm]
L_{3}'=L_{1}+L_{2}-p, \quad 
\text{where $p=\min (L_{1},\,L_{3})$}, \\[3mm]
\text{$L_{k}'=L_{k}$ for $4 \le k \le 10$}.
\end{array}
\end{equation}
Also, since $\omega(P)=P$ and 
$\omega(\bi)=(3,\,2,\,3,\,4,\,1,\,3,\,2,\,3,\,4,\,1)$, by setting 
\begin{equation*}
\psi_{\omega(\bi)}(P)=(L_{1}'',\,L_{2}'',\,\dots,\,L_{10}'') 
\in \BZ_{\ge 0}^{10},
\end{equation*}
we obtain from \eqref{eq:psi} the relation
$L_{k}''=L_{k}$ for all $1 \le k \le 10$. Since $L_{1}''=L_{1}'$, 
$L_{2}''=L_{2}'$, $L_{3}''=L_{3}'$ by the definitions 
(see \eqref{eq:length}), we have
%
%%%%%%%%%%%%%%
%%% eq:L02 %%%
%%%%%%%%%%%%%%
%
\begin{equation} \label{eq:L02}
L_{1}=L_{1}', \quad 
L_{2}=L_{2}', \quad 
L_{3}=L_{3}'. 
\end{equation}
By combining \eqref{eq:L01} and \eqref{eq:L02}, we get 
\begin{equation*}
\begin{array}{l}
L_{1}=L_{2}+L_{3}-p, \quad 
L_{2}=p, \\[3mm]
L_{3}=L_{1}+L_{2}-p, \quad 
\text{where $p=\min (L_{1},\,L_{3})$}.
\end{array}
\end{equation*}
Hence we deduce that $L_{1}=L_{3}$, and then that 
$L_{2}=\min(L_{1},\,L_{3})=L_{1}$. 
This proves the proposition. 
\end{proof}

The argument in the proof of Proposition~\ref{prop:L} 
also shows the following proposition.
%
%%%%%%%%%%%%%%%
%%% prop:L2 %%%
%%%%%%%%%%%%%%%
%
\begin{prop} \label{prop:L2}
Let $P=P(\mu_{\bullet}) \in \mv$ be an MV polytope 
with GGMS datum $\mu_{\bullet}=(\mu_{w})_{w \in W} \in \ggms_{\MV}$. 
Let $\bj=(j_{1},\,j_{2},\,\dots,\,j_{\ha{m}})$ be a reduced word 
for $\ha{w}_{0} \in \ha{W}$, and $\bi=(i_{1},\,i_{2},\,\dots,\,i_{m})$ 
the associated canonical reduced word for $w_{0} \in W$. 
We write $\psi_{\bi}(P) \in \BZ_{\ge 0}^{m}$ and 
$\psi_{\omega(\bi)}(P) \in \BZ_{\ge 0}^{m}$ as 
\begin{align*}
\psi_{\bi}(P)
 & = (L_{1},\,L_{2},\,\dots,\,L_{m}) \\
 & = \bigl(
      L_{1}^{(1)},\,\dots,\,L_{N_{1}}^{(1)},\ 
      L_{1}^{(2)},\,\dots,\,L_{N_{2}}^{(2)},\ \dots,\ 
      L_{1}^{(\ha{m})},\,\dots,\,L_{N_{\ha{m}}}^{(\ha{m})}
     \bigr) \in \BZ_{\ge 0}^{m}, \\
\psi_{\omega(\bi)}(P)
 & = (L_{1}',\,L_{2}',\,\dots,\,L_{m}') \\
 & = \bigl(
      L_{1}^{\prime (1)},\,\dots,\,L_{N_{1}}^{\prime (1)},\ 
      L_{1}^{\prime (2)},\,\dots,\,L_{N_{2}}^{\prime (2)},\ \dots,\ 
      L_{1}^{\prime (\ha{m})},\,\dots,\,L_{N_{\ha{m}}}^{\prime (\ha{m})}
     \bigr) \in \BZ_{\ge 0}^{m}. 
\end{align*}
If $L_{1}^{(k)}=L_{2}^{(k)}= \cdots =L_{N_{k}}^{(k)}$ 
for all $1 \le k \le \ha{m}$, then 
we have $L_{k}=L_{k}'$ for all $1 \le k \le m$. 
\end{prop}

%%%%%%%%%%%%%
%%% cor:L %%%
%%%%%%%%%%%%%
%
\begin{cor} \label{cor:L}
Keep the notation and assumptions of Proposition~\ref{prop:L}. 
Let $\ha{P}:=\Phi(P(\mu_{\bullet}))=
P(\Phi(\mu_{\bullet})) \in \ha{\pwp}$ be a pseudo-Weyl polytope with 
GGMS datum $\ha{\mu}=\Phi(\mu_{\bullet})=
(\ha{\mu}_{\ha{w}})_{\ha{w} \in \ha{W}} \in \ha{\ggms}$. 
Then, for a reduced word $\bj=(j_{1},\,j_{2},\,\dots,\,j_{\ha{m}})$ 
for $\ha{w}_{0} \in \ha{W}$, we have 
$\ha{\mu}_{\ha{w}_{k}^{\bj}}-\ha{\mu}_{\ha{w}_{k-1}^{\bj}}=
L_{l}^{(k)}\ha{w}_{k-1}^{\bj} \cdot h_{j_{k}}^{\omega}$
for every $1 \le l \le N_{k}$, $1 \le k \le \ha{m}$, 
with $\ha{w}_{k}^{\bj}:=
\ha{s}_{j_{1}}\ha{s}_{j_{2}} \cdots \ha{s}_{j_{k}}$, 
$1 \le k \le \ha{m}$.
\end{cor}

\begin{proof}
Again, we assume that $\ell=2n$, $n \in \BZ_{\ge 2}$, 
$j_{k}=n$, and further that $k=1$ and $n=2$. 
By the definition of $\ha{\mu}_{\bullet}=\Phi(\mu_{\bullet})$, 
we have $\ha{\mu}_{\ha{w}^{\bj}_{1}}-\ha{\mu}_{\ha{e}} = 
\mu_{\Theta(\ha{w}^{\bj}_{1})}-\mu_{\Theta(\ha{e})}$, 
where $\Theta(\ha{w}^{\bj}_{1}) = 
\Theta(\ha{s}_{j_{1}}) = 
s_{j_{1}}^{\omega}=s_{2}s_{3}s_{2}$
in our case. Also, recall that 
$(L_{1},\,L_{2},\,L_{3})=(L_{1}^{(1)},\,L_{2}^{(1)},\,L_{3}^{(1)})$ 
are determined via the length formula: 
\begin{equation*}
\begin{cases}
\mu_{s_{2}}-\mu_{e}=L_{1} h_{2}, \\[1.5mm]
\mu_{s_{2}s_{3}}-\mu_{s_{2}}=L_{2} s_{2} \cdot h_{3}, \\[1.5mm]
\mu_{s_{2}s_{3}s_{2}}-\mu_{s_{2}s_{3}}=L_{3} s_{2}s_{3} \cdot h_{2}. 
\end{cases}
\end{equation*}
Therefore, we have
\begin{align*}
\ha{\mu}_{\ha{w}^{\bj}_{1}}-\ha{\mu}_{\ha{e}} 
& = \mu_{s_{2}s_{3}s_{2}}-\mu_{e} \\
& = (\mu_{s_{2}s_{3}s_{2}}-\mu_{s_{2}s_{3}})+
    (\mu_{s_{2}s_{3}}-\mu_{s_{2}})+
    (\mu_{s_{2}}-\mu_{e}) \\
& = L_{3} s_{2}s_{3} \cdot h_{2} + 
    L_{2} s_{2} \cdot h_{3} + 
    L_{1} h_{2} \\
& = L_{3} h_{3} + 
    L_{2} (h_{2}+h_{3}) + 
    L_{1} h_{2} \\
& = (L_{1}+L_{2})h_{2}+(L_{2}+L_{3})h_{3}.
\end{align*}
Here we know from Proposition~\ref{prop:L} that
$L_{1}=L_{2}=L_{3}$. Hence we conclude that 
\begin{equation*}
\ha{\mu}_{\ha{w}^{\bj}_{1}}-\ha{\mu}_{\ha{e}} 
  = 2L_{1}(h_{2}+h_{3}) \\
  = L_{1}h_{2}^{\omega} 
  = L_{2}h_{2}^{\omega} 
  = L_{3}h_{2}^{\omega}. 
\end{equation*}
This proves the corollary. 
\end{proof}

Let $\mu_{\bullet}=(\mu_{w})_{w \in W} \in \ggms^{\omega}$, 
and set $\ha{\mu}_{\bullet}=\Phi(\mu_{\bullet})=
(\ha{\mu}_{\ha{w}})_{\ha{w} \in \ha{W}}$. Since 
$\ha{\mu}_{\ha{w}_{0}}=\mu_{w_{0}}=0$, we can show that 
$\ha{\mu}_{\ha{w}} \in \sum_{j \in \ha{I}} \BZ h_{j}^{\omega}$ 
for all $\ha{w} \in \ha{W}$ by repeated use of Corollary~\ref{cor:L};
take a reduced word $\bj=(j_{1},\,j_{2},\,\dots,\,j_{\ha{m}})$ for 
$\ha{w}_{0} \in \ha{W}$ such that $\ha{w}_{k}^{\bj}=\ha{w}$ 
for some $0 \le k \le \ha{m}$. 
Hence, for $\ha{M}_{\bullet}=\ha{D}(\ha{\mu}_{\bullet})=
(\ha{M}_{\ha{\gamma}})_{\ha{\gamma} \in \ha{\Gamma}}$, 
we have $\ha{M}_{\ha{\gamma}} \in \BZ$ 
for all $\ha{\gamma} \in \ha{\Gamma}$. 

Thus, it remains to prove that 
$\ha{M}_{\bullet}=\ha{D}(\ha{\mu}_{\bullet})=
(\ha{M}_{\ha{\gamma}})_{\ha{\gamma} \in \ha{\Gamma}}$ 
satisfies the tropical Pl\"ucker relations. 
We prove the tropical Pl\"ucker relation at 
$(\ha{w},\,n-1,\,n)$ for $\ha{w} \in \ha{W}$ in the case 
$\ell=2n$, $n \in \BZ_{\ge 2}$; note that 
$\ha{a}_{n-1,n}=-1$ and $\ha{a}_{n,n-1}=-2$. 
Since the proofs of the other tropical Pl\"ucker relations 
are similar (or, even simpler), we leave them to the reader. 
For simplicity of notation, we further assume that 
$\ha{w}=\ha{e}$ and $n=2$. Namely, we will prove 
%
%%%%%%%%%%%%%%%%
%%% eq:tp2-3 %%%
%%%%%%%%%%%%%%%%
%
\begin{equation} \label{eq:tp2-3}
\ha{M}_{\ha{s}_{2} \cdot \ha{\Lambda}_{2}}+ 
\ha{M}_{\ha{s}_{1}\ha{s}_{2} \cdot \ha{\Lambda}_{2}}+
\ha{M}_{\ha{s}_{1} \cdot \ha{\Lambda}_{1}} = 
\min \left(
 \begin{array}{l}
 2\ha{M}_{\ha{s}_{1}\ha{s}_{2} \cdot \ha{\Lambda}_{2}}+
 \ha{M}_{\ha{\Lambda}_{1}}, \\[3mm]
 2\ha{M}_{\ha{\Lambda}_{2}}+
 \ha{M}_{\ha{s}_{1}\ha{s}_{2}\ha{s}_{1} \cdot \ha{\Lambda}_{1}}, \\[3mm]
 \ha{M}_{\ha{\Lambda}_{2}} + 
 \ha{M}_{\ha{s}_{2}\ha{s}_{1}\ha{s}_{2} \cdot \ha{\Lambda}_{2}}+ 
 \ha{M}_{\ha{s}_{1} \cdot \ha{\Lambda}_{1}}
 \end{array}
 \right),
\end{equation}
%
%%%%%%%%%%%%%%%%
%%% eq:tp2-4 %%%
%%%%%%%%%%%%%%%%
%
\begin{equation} \label{eq:tp2-4}
\ha{M}_{\ha{s}_{2}\ha{s}_{1} \cdot \ha{\Lambda}_{1}}+ 
2\ha{M}_{\ha{s}_{1}\ha{s}_{2} \cdot \ha{\Lambda}_{2}}+
\ha{M}_{\ha{s}_{1} \cdot \ha{\Lambda}_{1}} = 
\min \left(
 \begin{array}{l}
 2\ha{M}_{\ha{\Lambda}_{2}}+
 2\ha{M}_{\ha{s}_{1}\ha{s}_{2}\ha{s}_{1} \cdot \ha{\Lambda}_{1}}, \\[3mm]
 2\ha{M}_{\ha{s}_{2}\ha{s}_{1}\ha{s}_{2} \cdot \ha{\Lambda}_{2}}+
 2\ha{M}_{\ha{s}_{1} \cdot \ha{\Lambda}_{1}}, \\[3mm]
 \ha{M}_{\ha{s}_{1}\ha{s}_{2}\ha{s}_{1} \cdot \ha{\Lambda}_{1}} + 
 2\ha{M}_{\ha{s}_{1}\ha{s}_{2} \cdot \ha{\Lambda}_{2}}+ 
 \ha{M}_{\ha{\Lambda}_{1}}
 \end{array}
 \right)
\end{equation}
(see \eqref{eq:tp2-1} and \eqref{eq:tp2-2}). 

We consider reduced words $\bj$ and $\bj'$
for $\ha{w}_{0} \in \ha{W}$ of the form
\begin{equation*}
\bj=(j_{1},\,j_{2},\,j_{3},\,j_{4}):=(1,\,2,\,1,\,2), \qquad 
\bj'=(j_{1}',\,j_{2}',\,j_{3}',\,j_{4}'):=(2,\,1,\,2,\,1).
\end{equation*}
As in the proof of \cite[Proposition~5.4]{Kam1}, 
we see by use of the length formula \eqref{eq:length} that 
$\ha{M}_{\bullet}=
 \ha{D}(\ha{\mu}_{\bullet})=
 (\ha{M}_{\ha{\gamma}})_{\ha{\gamma} \in \ha{\Gamma}}$ 
satisfies \eqref{eq:tp2-3} and \eqref{eq:tp2-4} if 
the following relations hold 
(cf. Remark~\ref{rem:trans}): 
%
%%%%%%%%%%%%%
%%% eq:mk %%%
%%%%%%%%%%%%%
%
\begin{equation} \label{eq:mk}
\begin{array}{l}
\ha{L}_{1}'
  =\ha{L}_{2}+\ha{L}_{3}+\ha{L}_{4}-\ha{p}_{1}, \quad 
\ha{L}_{2}'
  =2\ha{p}_{1}-\ha{p}_{2}, \\[3mm]
\ha{L}_{3}'
  =\ha{p}_{2}-\ha{p}_{1}, \quad 
\ha{L}_{4}'
  =\ha{L}_{1}+2\ha{L}_{2}+\ha{L}_{3}-\ha{p}_{2}, \quad \text{where} \\[3mm]
\ha{p}_{1}=\min 
  \bigl(
    \ha{L}_{1}+\ha{L}_{2}, \ 
    \ha{L}_{1}+\ha{L}_{4}, \ 
    \ha{L}_{3}+\ha{L}_{4}
  \bigr), \\[3mm]
\ha{p}_{2}=\min 
  \bigl(
    \ha{L}_{1}+2\ha{L}_{2}, \ 
    \ha{L}_{1}+2\ha{L}_{4}, \ 
    \ha{L}_{3}+2\ha{L}_{4}
  \bigr);
\end{array}
\end{equation}
here the $\ha{L}_{k}$, $1 \le k \le 4$, and the 
$\ha{L}_{k}'$, $1 \le k \le 4$, are determined via 
the length formula; 
\begin{equation*}
\begin{cases}
\ha{\mu}_{\ha{w}_{k}^{\bj}}-\ha{\mu}_{\ha{w}_{k-1}^{\bj}}=
\ha{L}_{k} \ha{w}_{k-1}^{\bj} \cdot h_{j_{k}}^{\omega}, 
& \text{with $\ha{w}_{k}^{\bj}=
  \ha{s}_{j_{1}}\ha{s}_{j_{2}} \cdots \ha{s}_{j_{k}}$}, \\[3mm]
\ha{\mu}_{\ha{w}_{k}^{\bj'}}-\ha{\mu}_{\ha{w}_{k-1}^{\bj'}}=
\ha{L}_{k}' \ha{w}_{k-1}^{\bj'} \cdot h_{j_{k}'}^{\omega}, 
& \text{with $\ha{w}_{k}^{\bj'}=
  \ha{s}_{j_{1}'}\ha{s}_{j_{2}'} \cdots \ha{s}_{j_{k}'}$}. 
\end{cases}
\end{equation*}
(Equation \eqref{eq:tp2-3} follows from 
the first one: $\ha{L}_{1}'=\ha{L}_{2}+\ha{L}_{3}+\ha{L}_{4}-\ha{p}_{1}$, 
and equation \eqref{eq:tp2-4} follows from 
the last one: $\ha{L}_{4}'=\ha{L}_{1}+2\ha{L}_{2}+\ha{L}_{3}-\ha{p}_{2}$.)
In other words, 
$\ha{M}_{\bullet}=
 \ha{D}(\ha{\mu}_{\bullet})=
 (\ha{M}_{\ha{\gamma}})_{\ha{\gamma} \in \ha{\Gamma}}$ 
satisfies \eqref{eq:tp2-3} and \eqref{eq:tp2-4} if 
\begin{equation*}
\ha{R}_{\bj}^{\bj'}
(\ha{L}_{1},\,\ha{L}_{2},\,\ha{L}_{3},\,\ha{L}_{4})=
(\ha{L}_{1}',\,\ha{L}_{2}',\,\ha{L}_{3}',\,\ha{L}_{4}'), 
\end{equation*}
where $\ha{R}_{\bj}^{\bj'}:\BZ_{\ge 0}^{4} \rightarrow \BZ_{\ge 0}^{4}$ 
is the transition map between Lusztig parametrizations of the canonical 
basis of $U_{q}^{-}((\Fg^{\omega})^{\vee})$. 
We know from Theorem~\ref{thm:trans} and 
\cite[Proposition~7.4]{BZ1} 
that the transition map $\ha{R}_{\bj}^{\bj'}:
\BZ_{\ge 0}^{4} \rightarrow \BZ_{\ge 0}^{4}$ is 
the tropicalization of the transition map 
$\ha{\CR}_{\bj}^{\bj'}:
\BR_{> 0}^{4} \rightarrow \BR_{> 0}^{4}$ defined by: 
$\ha{\CR}_{\bj}^{\bj'}=
(x_{\bj'}^{\omega})^{-1} \circ x_{\bj}^{\omega}$. 
Here we should remark that the bijections 
$x_{\bj}^{\omega},\,x_{\bj'}^{\omega}:\BR_{> 0}^{4} 
 \rightarrow N_{> 0} \cap G^{\omega}$ are given by: 
\begin{align*}
& x_{\bj}^{\omega}(u_{1},\,u_{2},\,u_{3},\,u_{4})=
  x_{j_{1}}^{\omega}(u_{1})
  x_{j_{2}}^{\omega}(u_{2})
  x_{j_{3}}^{\omega}(u_{3})
  x_{j_{4}}^{\omega}(u_{4}) \\
& x_{\bj'}^{\omega}(u_{1},\,u_{2},\,u_{3},\,u_{4})=
  x_{j_{1}'}^{\omega}(u_{1})
  x_{j_{2}'}^{\omega}(u_{2})
  x_{j_{3}'}^{\omega}(u_{3})
  x_{j_{4}'}^{\omega}(u_{4})
\end{align*}
for $(u_{1},\,u_{2},\,u_{3},\,u_{4}) \in \BR_{> 0}^{4}$, 
where $N_{> 0} \cap G^{\omega}$ is the set of fixed points of $N_{> 0}$ 
under the action of (the lifting to $G$ of) $\omega$, and 
$x_{j}^{\omega}(u):=\exp (u x_{j}^{\omega})$ 
for $u \in \BC$, $j \in \ha{I}$. 

Now, to the reduced words 
$\bj=(1,\,2,\,1,\,2)$ and $\bj'=(2,\,1,\,2,\,1)$ 
for $\ha{w}_{0} \in \ha{W}$, we associate canonical 
reduced words 
\begin{equation*}
\bi=(1,\,4,\,2,\,3,\,2,\,1,\,4,\,2,\,3,\,2)
\quad \text{and} \quad
\bi'=(2,\,3,\,2,\,1,\,4,\,2,\,3,\,2,\,1,\,4)
\end{equation*}
for $w_{0} \in W$, and set 
\begin{equation*}
\psi_{\bi}(P(\mu_{\bullet}))=
(L_{1},\,L_{2},\,\dots,\,L_{10}) \in \BZ_{\ge 0}^{10}
\quad \text{and} \quad
\psi_{\bi'}(P(\mu_{\bullet}))=
(L_{1}',\,L_{2}',\,\dots,\,L_{10}') \in \BZ_{\ge 0}^{10}.
\end{equation*}
Then we have 
\begin{equation*}
R_{\bi}^{\bi'}(L_{1},\,L_{2},\,\dots,\,L_{10})=
(L_{1}',\,L_{2}',\,\dots,\,L_{10}'). 
\end{equation*}
Also, since $\omega(P(\mu_{\bullet}))=P(\mu_{\bullet})$, 
it follows from Proposition~\ref{prop:L} and 
Corollary~\ref{cor:L} that 
%
%%%%%%%%%%%%%%
%%% eq:L03 %%%
%%%%%%%%%%%%%%
%
\begin{equation} \label{eq:L03}
\begin{cases}
L_{1}=L_{2}=\ha{L}_{1}, &  \\[1.5mm]
L_{3}=L_{4}=L_{5}=\ha{L}_{2}, & \\[1.5mm]
L_{6}=L_{7}=\ha{L}_{3}, & \\[1.5mm]
L_{8}=L_{9}=L_{10}=\ha{L}_{4}, & 
\end{cases}
\qquad \text{and} \qquad
\begin{cases}
L_{1}'=L_{2}'=L_{3}'=\ha{L}_{1}', & \\[1.5mm]
L_{4}'=L_{5}'=\ha{L}_{2}', & \\[1.5mm]
L_{6}'=L_{7}'=L_{8}'=\ha{L}_{3}', & \\[1.5mm]
L_{9}'=L_{10}'=\ha{L}_{4}'.
\end{cases}
\end{equation}
Thus, summarizing the above, 
what we must show is that 
if $(L_{1},\,L_{2},\,\dots,\,L_{10}) \in \BZ_{\ge 10}$ and 
$(L_{1}',\,L_{2}',\,\dots,\,L_{10}') \in \BZ_{\ge 10}$ 
are related by the transition map as
\begin{equation*}
R_{\bi}^{\bi'}(L_{1},\,L_{2},\,\dots,\,L_{10})=
 (L_{1}',\,L_{2}',\,\dots,\,L_{10}'), 
\end{equation*}
and if the relations \eqref{eq:L03} hold, then 
the relation 
\begin{equation*}
\ha{R}_{\bj}^{\bj'}
 (\ha{L}_{1},\,\ha{L}_{2},\,\ha{L}_{3},\,\ha{L}_{4})=
 (\ha{L}_{1}',\,\ha{L}_{2}',\,\ha{L}_{3}',\,\ha{L}_{4}')
\end{equation*}
holds. To show this statement, 
in view of the functoriality of the tropicalization 
(see \cite[Corollary~2.10]{BK}, and also 
\cite[Propositions~1.9 and 1.10]{NY}), 
it suffices to show the following lemma. 
%
%%%%%%%%%%%%%
%%% lem:R %%%
%%%%%%%%%%%%%
%
\begin{lem} \label{lem:R}
Define a map $H_{1}:\BR_{> 0}^{4} \rightarrow \BR_{> 0}^{10}$
by\,{\rm:}
\begin{equation*}
H_{1}(u_{1},\,u_{2},\,u_{3},\,u_{4})=
\left(
 u_{1},\,u_{1},\ 
 \frac{u_{2}}{\sqrt{2}},\,\sqrt{2}u_{2},\,\frac{u_{2}}{\sqrt{2}}, \ 
 u_{3},\,u_{3}, \ 
 \frac{u_{4}}{\sqrt{2}},\,\sqrt{2}u_{4},\,\frac{u_{4}}{\sqrt{2}} 
\right), 
\end{equation*}
and a map $H_{2}:\BR_{> 0}^{10} \rightarrow \BR_{> 0}^{4}$ 
by\,{\rm:}
\begin{equation*}
H_{1}(t_{1},\,t_{2},\,\dots,\,t_{10})=
(\sqrt{2}t_{1},\,t_{4},\,\sqrt{2}t_{6},\,t_{9}).
\end{equation*}
Then, the composition 
$H_{2} \circ \CR_{\bi}^{\bi'} \circ H_{1}:\BR_{> 0}^{4} 
\rightarrow \BR_{> 0}^{4}$ is identical to the transition map 
$\ha{\CR}_{\bj}^{\bj'}: \BR_{> 0}^{4} \rightarrow \BR_{> 0}^{4}$.
\end{lem}

\begin{proof}
Let $(u_{1},\,u_{2},\,u_{3},\,u_{4}) \in \BR_{> 0}^{4}$, 
and set $(u_{1}',\,u_{2}',\,u_{3}',\,u_{4}'):=
\ha{\CR}_{\bj}^{\bj'}(u_{1},\,u_{2},\,u_{3},\,u_{4}) \in \BR_{> 0}^{4}$. 
Then it follows from the definition of $\ha{\CR}_{\bj}^{\bj'}$ that 
%
%%%%%%%%%%%%%%%%
%%% eq:xjo01 %%%
%%%%%%%%%%%%%%%%
%
\begin{equation} \label{eq:xjo01}
x_{\bj}^{\omega}(u_{1},\,u_{2},\,u_{3},\,u_{4}) = 
x_{\bj'}^{\omega}(u_{1}',\,u_{2}',\,u_{3}',\,u_{4}').  
\end{equation}
Here we know from \cite[Proposition~7.4(b)]{BZ1} that 
%
%%%%%%%%%%%%%%%%
%%% eq:xjo02 %%%
%%%%%%%%%%%%%%%%
%
\begin{equation} \label{eq:xjo02}
\begin{array}{l}
x_{\bj}^{\omega}(u_{1},\,u_{2},\,u_{3},\,u_{4}) = 
  x_{\bi}
   \left(
    u_{1},\,u_{1},\ 
    \dfrac{u_{2}}{\sqrt{2}},\,\sqrt{2}u_{2},\,\dfrac{u_{2}}{\sqrt{2}}, \ 
    u_{3},\,u_{3}, \ 
    \dfrac{u_{4}}{\sqrt{2}},\,\sqrt{2}u_{4},\,\dfrac{u_{4}}{\sqrt{2}} 
   \right), \\[7.5mm]
x_{\bj'}^{\omega}(u_{1}',\,u_{2}',\,u_{3}',\,u_{4}') = 
  x_{\bi'}\left(
    \dfrac{u_{1}'}{\sqrt{2}},\,\sqrt{2}u_{1}',\,\dfrac{u_{1}'}{\sqrt{2}}, \ 
    u_{2}',\,u_{2}', \ 
    \dfrac{u_{3}'}{\sqrt{2}},\,\sqrt{2}u_{3}',\,\dfrac{u_{3}'}{\sqrt{2}}, \ 
    u_{4}',\,u_{4}'
  \right).
\end{array}
\end{equation}
By combining \eqref{eq:xjo01} and \eqref{eq:xjo02}, 
we conclude that 
\begin{equation*}
(H_{2} \circ \CR_{\bi}^{\bi'} \circ H_{1})
 (u_{1},\,u_{2},\,u_{3},\,u_{4})=
 (u_{1}',\,u_{2}',\,u_{3}',\,u_{4}'),
\end{equation*}
which is the desired equality. This proves the lemma. 
\end{proof}

Thus, we have completed the proof of 
the tropical Pl\"ucker relations for 
$\ha{M}_{\bullet}=
 \ha{D}(\ha{\mu}_{\bullet})=
 (\ha{M}_{\ha{\gamma}})_{\ha{\gamma} \in \ha{\Gamma}}$, 
and hence the proof of Proposition~\ref{prop:phi-bz}.

%==============================%
%     START SUBSECTION 0207    %
%==============================%
%
\subsection{Proof of Theorem~\ref{thm:fixed-mv}.}
\label{subsec:fixed-mv}

This subsection is devoted to the proof of Theorem~\ref{thm:fixed-mv}. 
We keep the notation and assumptions of \S\ref{subsec:phi}.
We first prove the following proposition, which is a part of 
Theorem~\ref{thm:fixed-mv}\,(3). 
%
%%%%%%%%%%%%%%%%
%%% prop:bij %%%
%%%%%%%%%%%%%%%%
%
\begin{prop} \label{prop:bij}
The map $\Phi:\mv^{\omega} \rightarrow \ha{\mv}$ is a bijection.
\end{prop}

\begin{proof}
First we show the injectivity of 
$\Phi:\mv^{\omega} \rightarrow \ha{\mv}$. 
Let $P \in \mv^{\omega}$ and $P' \in \mv^{\omega}$ be 
such that $\Phi(P)=\Phi(P')$. 
Take a reduced word 
$\bj=(j_{1},\,j_{2},\,\dots,\,j_{\ha{m}})$ for 
$\ha{w}_{0} \in \ha{W}$, and the associated 
canonical reduced word $\bi=(i_{1},\,i_{2},\,\dots,\,i_{m})$ 
for $w_{0} \in W$. Then we see from Proposition~\ref{prop:L} 
and Corollary~\ref{cor:L} that $\psi_{\bi}(P)=\psi_{\bi}(P')$, 
which implies that $P=P'$ since $\psi_{\bi}:\mv \rightarrow 
\BZ_{\ge 0}^{m}$ is a bijection. Thus, the injectivity of 
$\Phi:\mv^{\omega} \rightarrow \ha{\mv}$ follows.

Next we show the surjectivity of 
$\Phi:\mv^{\omega} \rightarrow \ha{\mv}$. 
Let $\ha{P}=\ha{P}(\ha{\mu}_{\bullet})$ be an 
element of $\ha{\mv}$ with GGMS datum 
$\ha{\mu}_{\bullet}=
 (\ha{\mu}_{\ha{w}})_{\ha{w} \in \ha{W}}
 \in \ha{\ggms}_{\MV}$. 
Take a reduced word 
$\bj=(j_{1},\,j_{2},\,\dots,\,j_{\ha{m}})$ for 
$\ha{w}_{0} \in \ha{W}$, and the associated 
canonical reduced word $\bi=(i_{1},\,i_{2},\,\dots,\,i_{m})$ 
for $w_{0} \in W$. By the length formula, we have 
$\ha{\mu}_{\ha{w}_{k}^{\bj}}-\ha{\mu}_{\ha{w}_{k-1}^{\bj}}=
\ha{L}_{k} \ha{w}_{k-1}^{\bj} \cdot h_{j_{k}}^{\omega}$,
with $\ha{w}_{k}^{\bj}=
\ha{s}_{j_{1}}\ha{s}_{j_{2}} \cdots \ha{s}_{j_{k}}$, 
for $1 \le k \le \ha{m}$. 
Now we define an element $P=P(\mu_{\bullet})$ of $\mv$ to 
be a unique preimage under 
$\psi_{\bi}:\mv \rightarrow \BZ_{\ge 0}^{m}$ of the element
\begin{equation*}
(L_{1},\,L_{2},\,\dots,\,L_{m}):= 
\bigl(
  \underbrace{\ha{L}_{1},\,\dots,\,\ha{L}_{1}}_{\text{$N_{1}$ times}}, \ 
  \underbrace{\ha{L}_{2},\,\dots,\,\ha{L}_{2}}_{\text{$N_{2}$ times}}, \ 
  \dots,\ 
  \underbrace{\ha{L}_{\ha{m}},\,\dots,\,\ha{L}_{\ha{m}}}_{\text{$N_{\ha{m}}$ times}}
\bigr)
\end{equation*}
of $\BZ_{\ge 0}^{m}$. Then, from Proposition~\ref{prop:L2} 
together with \eqref{eq:psi}, we deduce that $\omega(P)=P$. 
Moreover, it follows from Corollary~\ref{cor:L} that 
$\Phi(P)=\ha{P}$. Thus, the surjectivity of 
$\Phi:\mv^{\omega} \rightarrow \ha{\mv}$ follows. 
This proves the proposition.
\end{proof}

To prove part (1) of Theorem~\ref{thm:fixed-mv}, 
we need the following lemma. 
%
%%%%%%%%%%%%%%
%%% lem:fo %%%
%%%%%%%%%%%%%%
%
\begin{lem} \label{lem:fo}
Let $P=P(\mu_{\bullet}) \in \mv$ be an MV polytope with 
GGMS datum $\mu_{\bullet}=(\mu_{w})_{w \in W} \in \ggms_{\MV}$. 
Then, we have $\omega(f_{j}P)=f_{\omega(j)}\,\omega(P)$ 
for all $j \in I$.
\end{lem}

\begin{proof}
Fix $j \in I$, and take a reduced word 
$\bi=(i_{1},\,i_{2},\,\dots,\,i_{m})$ for 
$w_{0} \in W$ such that $i_{1}=j$. 
If we write $\psi_{\bi}(P) \in \BZ_{\ge 0}^{m}$ as 
$\psi_{\bi}(P)=(L_{1},\,L_{2},\,\dots,\,L_{m}) \in \BZ_{\ge 0}^{m}$, 
then we know from \cite[Proposition~3.4]{Kam2} that 
$\psi_{\bi}(f_{j}P)=(L_{1}+1,\,L_{2},\,\dots,\,L_{m})$. 
Also, we know from Lemma~\ref{lem:psi} that 
$\psi_{\omega(\bi)}(\omega(P))=
 \psi_{\bi}(P)=(L_{1},\,L_{2},\,\dots,\,L_{m})$.
Because $\omega(\bi)=
(\omega(i_{1}),\,\omega(i_{2}),\,\dots,\,\omega(i_{m}))$ is 
a reduced word for $w_{0} \in W$ such that 
$\omega(i_{1})=\omega(j)$, it follows again from 
\cite[Proposition~3.4]{Kam2} that 
$\psi_{\omega(\bi)}(f_{\omega(j)}\,\omega(P))=
(L_{1}+1,\,L_{2},\,\dots,\,L_{m})$. Therefore, 
we obtain $\psi_{\omega(\bi)}(f_{\omega(j)}\,\omega(P))=
\psi_{\bi}(f_{j}(P))$, which is equal to 
$\psi_{\omega(\bi)}(\omega(f_{j}P))$ again by Lemma~\ref{lem:psi}. 
From this fact, we conclude that $f_{\omega(j)}(\omega(P))=\omega(f_{j}P)$
since $\psi_{\omega(\bi)}:\mv \rightarrow \BZ_{\ge 0}^{m}$ is a bijection. 
This proves the lemma. 
\end{proof}

The following proposition is precisely part (1) of 
Theorem~\ref{thm:fixed-mv}. 
%
%%%%%%%%%%%%%%%%%%%
%%% prop:stable %%%
%%%%%%%%%%%%%%%%%%%
%
\begin{prop} \label{prop:stable}
The subset $\mv^{\omega}$ of $\mv$ is stable under the operators 
$f_{j}^{\omega}$ for all $j \in \ha{I}$. 
\end{prop}

\begin{proof}
Let $P \in \mv^{\omega}$ be an MV polytope such that $\omega(P)=P$. 
We prove that $f_{j}^{\omega}P \in \mv^{\omega}$, i.e., 
$\omega(f_{j}^{\omega}P)=f_{j}^{\omega}P$ 
in the case that $\ell=2n$, $n \in \BZ_{\ge 2}$, and $j=n$; 
the proofs for the other cases are simpler. 
Repeated application of Lemma~\ref{lem:fo} shows that 
\begin{align*}
\omega(f_{j}^{\omega}P) 
 & = \omega(f_{n}^{\omega}P) 
   = \omega(f_{n}f_{\omega(n)}^{2}f_{n}P) \\
 & = f_{\omega(n)}f_{n}^{2}f_{\omega(n)}\,\omega(P)
   = f_{\omega(n)}f_{n}^{2}f_{\omega(n)} P, 
\end{align*}
since $\omega(P)=P$ by assumption.
Here we recall from Remark~\ref{rem:ok} that 
as operators on $\mv$, 
$f_{n}f_{\omega(n)}^{2}f_{n}=
 f_{\omega(n)}f_{n}^{2}f_{\omega(n)}$. 
Therefore, we obtain 
\begin{equation*}
\omega(f_{n}^{\omega}P) 
 = f_{\omega(n)}f_{n}^{2}f_{\omega(n)} P 
 = f_{n}f_{\omega(n)}^{2}f_{n} P =f_{n}^{\omega}P.
\end{equation*}
This proves the proposition. 
\end{proof}

The following proposition is a part of 
Theorem~\ref{thm:fixed-mv}\,(3). 
%
%%%%%%%%%%%%%%%%
%%% prop:com %%%
%%%%%%%%%%%%%%%%
%
\begin{prop} \label{prop:com}
We have $\Phi \circ f_{j}^{\omega}= \ha{f}_{j} \circ \Phi$ 
for all $j \in \ha{I}$. 
\end{prop}

\begin{proof}
Let $P \in \mv^{\omega}$, and let 
$\ha{\mu}_{\bullet}=
 (\ha{\mu}_{\ha{w}})_{\ha{w} \in \ha{W}} \in \ha{\ggms}_{\MV}$ and 
$\ha{\mu}_{\bullet}''=
 (\ha{\mu}_{\ha{w}}'')_{\ha{w} \in \ha{W}} \in \ha{\ggms}_{\MV}$ 
be the GGMS data of $\Phi(P) \in \ha{\mv}$ and 
$\Phi(f_{j}^{\omega}P) \in \ha{\mv}$, respectively. 
Recall that $\ha{f}_{j}\Phi(P) \in \ha{\mv}$ is defined to be 
the unique MV polytope in $\ha{\mv}$ with GGMS datum 
$\ha{\mu}_{\bullet}'=(\ha{\mu}_{\ha{w}}')_{\ha{w} \in \ha{W}} \in \ha{\ggms}_{\MV}$ 
such that $\ha{\mu}_{\ha{e}}'=\ha{\mu}_{\ha{e}}-h_{j}^{\omega}$, and 
$\ha{\mu}_{\ha{w}}'=\ha{\mu}_{\ha{w}}$ for all 
$\ha{w} \in \ha{W}$ with $\ha{s}_{j}\ha{w} < \ha{w}$. 
Hence, in order to prove that 
$\ha{f}_{j}\Phi(P)=\Phi(f_{j}^{\omega}P)$, 
it suffices to show that 
%
%%%%%%%%%%%%%%%%
%%% eq:con01 %%%
%%%%%%%%%%%%%%%%
%
\begin{equation} \label{eq:con01}
\ha{\mu}_{\ha{e}}''=\ha{\mu}_{\ha{e}}-h_{j}^{\omega},
\quad \text{and} \quad
\ha{\mu}_{\ha{w}}''=\ha{\mu}_{\ha{w}}
\quad \text{for all 
$\ha{w} \in \ha{W}$ with $\ha{s}_{j}\ha{w} < \ha{w}$}.
\end{equation}

Let $\mu_{\bullet}=(\mu_{w})_{w \in W} \in \ggms_{\MV}^{\omega}$ and 
$\mu_{\bullet}''=(\mu_{w}'')_{w \in W} \in \ggms_{\MV}^{\omega}$ 
be the GGMS data of $P \in \mv^{\omega}$ and 
$f_{j}^{\omega}P \in \mv^{\omega}$, respectively. 
Note that $\ha{\mu}_{\ha{w}}=\mu_{\Theta(\ha{w})}$ and 
$\ha{\mu}_{\ha{w}}''=\mu_{\Theta(\ha{w})}''$ for 
$\ha{w} \in \ha{W}$ by the definition of the map $\Phi$.
Also, we deduce from the definitions of 
the lowering Kashiwara operators $f_{j}$, $j \in I$, 
and the operators $f_{j}^{\omega}$, $j \in \ha{I}$, that 
$\mu_{e}''=\mu_{e}-h_{j}^{\omega}$. 
Hence we obtain 
$\ha{\mu}_{\ha{e}}''=\mu_{e}''=
\mu_{e}-h^{\omega}_{j} = 
\ha{\mu}_{\ha{e}}-h^{\omega}_{j}$.
Next, let $\ha{w} \in \ha{W}$ be such that $\ha{s}_{j}\ha{w} < \ha{w}$, 
and set $w:=\Theta(\ha{w})$. Then it follows from 
Remark~\ref{rem:theta} that $s_{j}w < w$ and $s_{\omega(j)}w < w$.
Therefore, we deduce again from the definitions of 
the lowering Kashiwara operators $f_{j}$, $j \in I$, 
and the operators $f_{j}^{\omega}$, $j \in \ha{I}$, that 
$\mu_{w}''=\mu_{w}$, and hence 
$\ha{\mu}_{\ha{w}}''=\mu_{w}''=\mu_{w}=\ha{\mu}_{\ha{w}}$.
This proves \eqref{eq:con01}, 
which completes the proof of the proposition.
\end{proof}

Now we prove the remaining parts of 
Theorem~\ref{thm:fixed-mv}. Recall from Remark~\ref{rem:binf} 
the MV polytope $P^{0}=P(\mu_{\bullet}^{0}) \in \mv$ 
corresponding to $u_{\infty} \in \CB(\infty)$ under the 
isomorphism $\Psi:\mv \rightarrow \CB(\infty)$ of 
crystals for $\Fg^{\vee}$. It is obvious that 
$\omega(P^{0})=P^{0}$, i.e., $P^{0} \in \mv^{\omega}$. 
Also, it follows from the definition of the map 
$\Phi:\mv^{\omega} \rightarrow \ha{\mv}$ that 
$\Phi(P^{0})=\ha{P}^{0}$, where $\ha{P}^{0} \in \ha{\mv}$ is 
the MV polytope corresponding to 
$\ha{u}_{\infty} \in \ha{\CB}(\infty)$ under the 
isomorphism $\ha{\Psi}:\ha{\mv} \rightarrow 
\ha{\CB}(\infty)$ of crystals for $(\Fg^{\omega})^{\vee}$. 
Moreover, because $\ha{\mv}$ is isomorphic to $\ha{\CB}(\infty)$, 
each element $\ha{P} \in \ha{\mv}$ is of the form 
$\ha{P}=\ha{f}_{j_{1}}\ha{f}_{j_{2}} \cdots \ha{f}_{j_{k}} \ha{P}^{0}$
for some $j_{1},\,j_{2},\,\dots,\,j_{k} \in \ha{I}$. 

Let $P \in \mv^{\omega}$, and set $\ha{P}:=\Phi(P) \in \ha{\mv}$. 
Then, as above, 
there exist $j_{1},\,j_{2},\,\dots,\,j_{k} \in \ha{I}$
such that $\ha{P}=\ha{f}_{j_{1}}\ha{f}_{j_{2}} \cdots 
\ha{f}_{j_{k}} \ha{P}^{0}$. Here we note that 
$f_{j_{1}}^{\omega} f_{j_{2}}^{\omega} \cdots 
 f_{j_{k}}^{\omega} P^{0} \in \mv^{\omega}$ 
by Proposition~\ref{prop:stable}.
Therefore, by using Proposition~\ref{prop:com}, 
we obtain 
\begin{equation*}
\Phi(f_{j_{1}}^{\omega} f_{j_{2}}^{\omega} \cdots 
f_{j_{k}}^{\omega} P^{0}) = 
\ha{f}_{j_{1}}\ha{f}_{j_{2}} \cdots 
\ha{f}_{j_{k}} \Phi(P^{0})= 
\ha{f}_{j_{1}}\ha{f}_{j_{2}} \cdots 
\ha{f}_{j_{k}} \ha{P}^{0} = 
\ha{P}=\Phi(P), 
\end{equation*}
which implies that 
$P=f_{j_{1}}^{\omega} f_{j_{2}}^{\omega} \cdots 
f_{j_{k}}^{\omega} P^{0}$ by Proposition~\ref{prop:bij}. 
This proves part (2) of Theorem~\ref{thm:fixed-mv}, 
and hence the uniqueness assertion in 
Theorem~\ref{thm:fixed-mv}\,(3). 
Thus, we have completed the proof of 
Theorem~\ref{thm:fixed-mv}. 

%=========================%
%     START SECTION 03    %
%=========================%
%
\section{Descriptions of the lowering Kashiwara operators $\ha{f}_{j}$.}
\label{sec:hk}

We maintain the assumption that 
$\Fg$ is either of type $A_{\ell}$ with $\ell=2n-1$, 
$n \in \BZ_{\ge 2}$, or of type $A_{\ell}$ with $\ell=2n$, 
$n \in \BZ_{\ge 2}$. 

%==============================%
%     START SUBSECTION 0301    %
%==============================%
%
\subsection{Description of $\ha{f}_{j}$ in terms of BZ data.}
\label{subsec:comb}
First, let us recall from \cite[\S5.2]{Kam2} 
the description of the lowering Kashiwara operators 
$f_{j}$, $j \in I$, on $\mv$ in terms of BZ data. 
Note that
%
%%%%%%%%%%%%%%%
%%% eq:minu %%%
%%%%%%%%%%%%%%%
%
\begin{equation} \label{eq:minu}
\pair{\gamma}{h_{j}} \in \bigl\{-1,\,0,\,1\bigr\} 
\qquad
\text{for all $\gamma \in \Gamma$ and $j \in I$},
\end{equation}
since every fundamental weight for $\Fg$ (of type $A_{\ell}$) 
is minuscule. The next theorem follows 
immediately from \cite[\S5.1]{Kam2} and \eqref{eq:minu}. 
%
%%%%%%%%%%%%%%
%%% thm:am %%%
%%%%%%%%%%%%%%
%
\begin{thm} \label{thm:am}
Let $P \in \mv$, and $j \in I$. 
Let 
$M_{\bullet}=(M_{\gamma})_{\gamma \in \Gamma} \in \edge_{\MV}$ and 
$M_{\bullet}'=(M_{\gamma}')_{\gamma \in \Gamma} \in \edge_{\MV}$ 
be the BZ data of $P$ and $f_{j}P$, respectively.
Then, for each $\gamma \in \Gamma$, 
\begin{equation*}
M_{\gamma}'=
\begin{cases}
\min \bigl( M_{\gamma}, M_{s_{j} \cdot \gamma}+c_{j}^{M_{\bullet}} \bigr)
   & \text{\rm if $\pair{\gamma}{h_{j}}=1$}, \\[3mm]
M_{\gamma} 
   & \text{\rm otherwise},
\end{cases}
\end{equation*}
where $c_{j}^{M_{\bullet}}:=
M_{\Lambda_{j}}-M_{s_{j}\Lambda_{j}}-1$. 
\end{thm}
The following corollary will be needed 
in the next subsection. 
%
%%%%%%%%%%%%%%
%%% cor:am %%%
%%%%%%%%%%%%%%
%
\begin{cor} \label{cor:am}
Keep the notation of Theorem~\ref{thm:am}. 
Then we have $f_{j}P \supset P$. 
\end{cor}

\begin{proof}
It follows from \eqref{eq:po01} that 
\begin{align*}
P & =
 \bigl\{ 
 h \in \Fh_{\BR} \mid 
 \gamma(h) \ge M_{\gamma} \ 
 \text{for all $\gamma \in \Gamma$}
 \bigr\}, \\
f_{j}P & =
 \bigl\{ 
 h \in \Fh_{\BR} \mid 
 \gamma(h) \ge M_{\gamma}' \ 
 \text{for all $\gamma \in \Gamma$}
 \bigr\}.
\end{align*}
Also, we see from Theorem~\ref{thm:am} that 
$M_{\gamma} \ge M_{\gamma}'$ for all $\gamma \in \Gamma$. 
Hence we obtain $f_{j}P \supset P$, as desired. 
\end{proof}

By Theorem~\ref{thm:fixed-mv}, 
giving a description of $\ha{f}_{j}$, $j \in \ha{I}$, on $\ha{\mv}$ 
is equivalent to giving a description of 
$f_{j}^{\omega}$, $j \in \ha{I}$, on $\mv^{\omega}$.
Applying Theorem~\ref{thm:am} successively, we can prove 
the following proposition. 
%
%%%%%%%%%%%%%%%
%%% prop:ok %%%
%%%%%%%%%%%%%%%
%
\begin{prop} \label{prop:ok}
Let $P \in \mv^{\omega}$ and $j \in \ha{I}$.
Let $M_{\bullet}=
     (M_{\gamma})_{\gamma \in \Gamma} \in \edge_{\MV}^{\omega}$ 
and $M_{\bullet}'=
     (M_{\gamma}')_{\gamma \in \Gamma} \in \edge_{\MV}^{\omega}$
be the BZ data of $P$ and $f_{j}^{\omega}P$, 
respectively. 

%%%%%
\svsp
%%%%%

\noindent
{\rm (1)} If $1 \le j \le n-1$, then for $\gamma \in \Gamma$, 
we have
%
%%%%%%%%%%%%%%
%%% eq:ok1 %%%
%%%%%%%%%%%%%%
%
\begin{equation} \label{eq:ok1}
M_{\gamma}' = 
\begin{cases}
 M_{\gamma} & 
 \text{\rm if $\pair{\gamma}{h_{j}} \le 0$ and 
       $\pair{\gamma}{h_{\omega(j)}} \le 0$}, \\[3mm]
 \min \bigl(
   M_{\gamma},\ M_{s_{\omega(j)} \cdot \gamma}+c_{j}^{M_{\bullet}}
   \bigr)
   & \text{\rm if $\pair{\gamma}{h_{j}} \le 0$ and 
           $\pair{\gamma}{h_{\omega(j)}}=1$}, \\[3mm]
 \min \bigl(
   M_{\gamma},\ M_{s_{j} \cdot \gamma}+c_{j}^{M_{\bullet}}
   \bigr)
   & \text{\rm if $\pair{\gamma}{h_{j}}=1$ and 
           $\pair{\gamma}{h_{\omega(j)}} \le 0$}, \\[5mm]
 \min \left(
   \begin{array}{l}
   M_{\gamma},\ 
   M_{s_{j} \cdot \gamma}+c_{j}^{M_{\bullet}}, \\[1.5mm]
   M_{s_{\omega(j)} \cdot \gamma}+c_{j}^{M_{\bullet}}, \\[1.5mm]
   M_{s_{j}s_{\omega(j)} \cdot \gamma}+2c_{j}^{M_{\bullet}}
   \end{array}
   \right)
   & \text{\rm if 
     $\pair{\gamma}{h_{j}}=\pair{\gamma}{h_{\omega(j)}}=1$}. 
\end{cases}
\end{equation}

%%%%%
\svsp
%%%%%

\noindent
{\rm (2)} If $\ell=2n-1$, $n \in \BZ_{\ge 2}$, and 
$j=n$, then for $\gamma \in \Gamma$, 
we have
%
%%%%%%%%%%%%%%
%%% eq:ok2 %%%
%%%%%%%%%%%%%%
%
\begin{equation} \label{eq:ok2}
M_{\gamma}'=
\begin{cases}
 \min \bigl(
   M_{\gamma},\ 
   M_{s_{n} \cdot\gamma}+c_{n}^{M_{\bullet}}
   \bigr)
   & \text{\rm if $\pair{\gamma}{h_{n}}=1$}, \\[3mm]
 M_{\gamma} & \text{\rm otherwise}.
\end{cases}
\end{equation}

%%%%%
\svsp
%%%%%

\noindent
{\rm (3)} If $\ell=2n$, $n \in \BZ_{\ge 2}$, and 
$j=n$, then for $\gamma \in \Gamma$, we have
%
%%%%%%%%%%%%%%
%%% eq:ok3 %%%
%%%%%%%%%%%%%%
%
\begin{equation} \label{eq:ok3}
M_{\gamma}'=
\begin{cases}
 \min \left( \begin{array}{l}
      M_{\gamma},\, 
      M_{s_{\omega(n)} \cdot \gamma}+c_{n}^{M_{\bullet}}, \\[1.5mm]
      M_{s_{n}s_{\omega(n)} \cdot \gamma}+2c_{n}^{M_{\bullet}}
      \end{array} \right)
   & \text{\rm if $\pair{\gamma}{h_{n}}=0$ 
      and $\pair{\gamma}{h_{\omega(n)}}=1$}, \\[7mm]
 \min \bigl(M_{\gamma}, \, 
            M_{s_{n} \cdot \gamma}+c_{n}^{M_{\bullet}} \bigr) 
   & \text{\rm if $\pair{\gamma}{h_{n}}=1$ 
      and $\pair{\gamma}{h_{\omega(n)}}=-1$}, \\[3mm]
 \min \left( \begin{array}{l}
            M_{\gamma}, \, 
            M_{s_{n} \cdot \gamma}+c_{n}^{M_{\bullet}}, \\[1.5mm]
            M_{s_{\omega(n)}s_{n} \cdot \gamma}+2c_{n}^{M_{\bullet}}
       \end{array} \right)
   & \text{\rm if $\pair{\gamma}{h_{n}}=1$ 
      and $\pair{\gamma}{h_{\omega(n)}}=0$}, \\[7mm]
 \min \bigl(M_{\gamma},\, 
            M_{s_{\omega(n)} \cdot \gamma}+c_{n}^{M_{\bullet}}\bigr)
   & \text{\rm if $\pair{\gamma}{h_{n}}=-1$ 
      and $\pair{\gamma}{h_{\omega(n)}}=1$}, \\[3mm]
 M_{\gamma}
   & \text{\rm otherwise}.
\end{cases}
\end{equation}
\end{prop}

\begin{proof}
Since the proofs of these formulas are rather straightforward, 
we only sketch them, leaving the details to the reader. 
In the proof of part~(1), we need the equations 
$c_{\omega(j)}^{M_{\bullet}}=c_{j}^{M_{\bullet}}$ and 
$c_{j}^{M_{\bullet}''}=c_{j}^{M_{\bullet}}$, 
where $M_{\bullet}'' \in \edge_{\MV}$ is the BZ datum 
of $f_{\omega(j)}P$; 
recall that $f_{j}^{\omega}=f_{j}f_{\omega(j)}$. 
The first one 
$c_{\omega(j)}^{M_{\bullet}}=c_{j}^{M_{\bullet}}$ 
follows immediately from Lemma~\ref{lem:bzmvda}\,(2), 
and the second one 
$c_{j}^{M_{\bullet}''}=c_{j}^{M_{\bullet}}$ 
is easily shown by using 
Theorem~\ref{thm:am}, along with 
Remark~\ref{rem:omega}\,(4) and 
Lemma~\ref{lem:bzmvda}\,(2). 
Also, in the proof of part~(3), 
we need the following equations: 
$c_{\omega(n)}^{M_{\bullet}^{(1)}}=c_{n}^{M_{\bullet}}+1$, \, 
$c_{\omega(n)}^{M_{\bullet}^{(2)}}=c_{n}^{M_{\bullet}}$, and 
$c_{n}^{M_{\bullet}^{(3)}}=c_{n}^{M_{\bullet}}$,
where $M_{\bullet}^{(1)} \in \edge_{\MV}$ 
(resp., $M_{\bullet}^{(2)},\,M_{\bullet}^{(3)} \in \edge_{\MV}$) 
is the BZ datum of $f_{n}P$ 
(resp., $f_{\omega(n)}f_{n}P$, $f_{\omega(n)}^{2}f_{n}P$); 
recall that $f_{n}^{\omega}=f_{n}f_{\omega(n)}^{2}f_{n}$. 
These equations are easily shown by using 
Theorem~\ref{thm:am}, along with Remark~\ref{rem:omega}\,(4), 
Lemma~\ref{lem:bzmvda}\,(2), and Lemma~\ref{lem:tp} below.
\end{proof}
%
%%%%%%%%%%%%%%
%%% lem:tp %%%
%%%%%%%%%%%%%%
%
\begin{lem} \label{lem:tp}
Assume that $\ell=2n$, $n \in \BZ_{\ge 2}$. 
Let $M_{\bullet}=(M_{\gamma})_{\gamma \in \Gamma} \in \edge_{\MV}^{\omega}$, 
and let $w \in W^{\omega}$ be such that $ws_{n} > w$ and 
$ws_{n+1} > w$. Then, we have 
%
%%%%%%%%%%%%%
%%% eq:tp %%%
%%%%%%%%%%%%%
%
\begin{equation} \label{eq:tp}
2M_{ws_{n} \cdot \Lambda_{n}} = 
M_{w \cdot \Lambda_{n}} + M_{ws_{n+1}s_{n} \cdot \Lambda_{n}}.
\end{equation}
\end{lem}

\begin{proof}
By the tropical Pl\"ucker relation at $(w,\,n,\,n+1)$ 
(see \eqref{eq:tp1-1}), we have 
%
%%%%%%%%%%%%%%%
%%% eq:tp01 %%%
%%%%%%%%%%%%%%%
%
\begin{equation} \label{eq:tp01}
M_{ws_{n} \cdot \Lambda_{n}}+
M_{ws_{n+1} \cdot \Lambda_{n+1}} = 
 \min \bigl( 
  M_{w \cdot \Lambda_{n}} + 
  M_{ws_{n}s_{n+1} \cdot \Lambda_{n+1}}, \ 
  M_{ws_{n+1}s_{n} \cdot \Lambda_{n}} +
  M_{w \cdot \Lambda_{n+1}} 
 \bigr).
\end{equation}
Since $M_{\bullet}=(M_{\gamma})_{\gamma \in \Gamma} 
\in \edge_{\MV}^{\omega}$ and $w \in W^{\omega}$ by assumption, 
it follows immediately from Lemma~\ref{lem:bzmvda}\,(2) 
along with Remark~\ref{rem:omega}\,(4) that 
%
%%%%%%%%%%%%%%%
%%% eq:tp02 %%%
%%%%%%%%%%%%%%%
%
\begin{equation} \label{eq:tp02}
\begin{cases}
M_{ws_{n+1} \cdot \Lambda_{n+1}}
 =M_{ws_{n} \cdot \Lambda_{n}}, & \\[1.5mm]
M_{w \cdot \Lambda_{n}}=M_{w \cdot \Lambda_{n+1}}, & \\[1.5mm]
M_{ws_{n}s_{n+1} \cdot \Lambda_{n+1}}= 
  M_{ws_{n+1}s_{n} \cdot \Lambda_{n}}. &
\end{cases}
\end{equation}
Combining \eqref{eq:tp01} and \eqref{eq:tp02}, 
we obtain \eqref{eq:tp}, as desired. 
\end{proof}

%==============================%
%     START SUBSECTION 0302    %
%==============================%
%
\subsection{Description of $\ha{f}_{j}$ in terms of GGMS data.}
\label{subsec:poly}

First, let us recall from \cite[\S5.1]{Kam2}
the description of the lowering Kashiwara operators 
$f_{j}$, $j \in I$, on $\mv$ in terms of GGMS data. 
Fix $j \in I$ and $P \in \mv$. 
Let $\mu_{\bullet}=
 (\mu_{w})_{w \in W} \in \ggms_{\MV}$ be 
the GGMS datum of $P$, and set 
$M_{\bullet}=(M_{\gamma})_{\gamma \in \Gamma}:=
 D(\mu_{\bullet}) \in \edge_{\MV}$. 
Define a reflection $\sigma:\Fh \rightarrow \Fh$ by: 
$\sigma(h)=s_{j} \cdot h+c h_{j}$ for $h \in \Fh$, 
where $c:=c^{M_{\bullet}}_{j}=
M_{\Lambda_{j}}-M_{s_{j} \cdot \Lambda_{j}}-1$. 
Also, we set 
\begin{equation*}
W_{+}:=
  \bigl\{w \in W \mid 
   s_{j} w > w \bigr\}, \qquad
W_{-}:=
  \bigl\{w \in W \mid 
   s_{j} w < w \bigr\}; 
\end{equation*}
note that $W=W_{+} \cup W_{-}$. 
The following was conjectured by Anderson-Mirkovi\'c, and 
proved by Kamnitzer. 
%
%%%%%%%%%%%%%%%
%%% thm:pam %%%
%%%%%%%%%%%%%%%
%
\begin{thm}[{\cite[Theorem~5.5]{Kam2}}] \label{thm:pam}
Keep the notation above. 
Then, $f_{j}P \in \mv$ is 
the smallest pseudo-Weyl polytope $P' \in \pwp$ 
with GGMS datum $\mu_{\bullet}'=(\mu_{w}')_{w \in W} \in \ggms$ 
such that 

\noindent
{\rm (i)} $\mu_{w}'=\mu_{w}$ for all $w \in W_{-}$, 

\noindent
{\rm (ii)} $\mu_{e}' = \mu_{e}-h_{j}$, 

\noindent
{\rm (iii)} $P'$ contains $\mu_{w}$ for all $w \in W_{+}$, and 

\noindent
{\rm (iv)} if $w \in W_{-}$ is such that 
$\pair{\alpha_{j}}{\mu_{w}} \ge c$, 
then $P'$ contains $\sigma(\mu_{w})$.
\end{thm}

The aim of this subsection is 
to give a description of 
the lowering Kashiwara operators 
$\ha{f}_{j}$, $j \in \ha{I}$, on $\ha{\mv}$ 
in terms of GGMS data. For this aim, 
we introduce some additional notation.
For each $M_{\bullet}=
(M_{\gamma})_{\gamma \in \Gamma} \in \edge^{\omega}$, 
we define a convex polytope 
$\ti{P}(M_{\bullet})$ in $\Fh_{\BR}$ by:
%
%%%%%%%%%%%%%%%%%
%%% eq:tipo02 %%%
%%%%%%%%%%%%%%%%%
%
\begin{equation} \label{eq:tipo02}
\ti{P}(M_{\bullet})=
 \bigl\{ 
 h \in \Fh_{\BR} \mid 
 \pair{\gamma}{h} \ge M_{\gamma} \ 
 \text{for all $\gamma \in \ti{\Gamma}$}
 \bigr\},
\end{equation}
where $\ti{\Gamma}:=
\bigl\{w\Lambda_{i} \mid w \in W^{\omega}, \, i \in I\bigr\}$.
Also, for each $\mu_{\bullet}=
(\mu_{w})_{w \in W} \in \ggms^{\omega}$, we define 
a convex polytope $\ti{P}(\mu_{\bullet})$ in $\Fh_{\BR}$ by:
%
%%%%%%%%%%%%%%%%%
%%% eq:tipo01 %%%
%%%%%%%%%%%%%%%%%
%
\begin{equation} \label{eq:tipo01}
\ti{P}(\mu_{\bullet})=
 \bigl\{ 
 h \in \Fh_{\BR} \mid h \ge_{w} \mu_{w} \ 
 \text{for all $w \in W^{\omega}$}
 \bigr\}.
\end{equation}
Then, it is obvious that 
$\ti{P}(\mu_{\bullet})=\ti{P}(D(\mu_{\bullet}))$ 
for all $\mu_{\bullet} \in \ggms^{\omega}$. 
Moreover, it follows from this equality that the set 
$\ti{P}(\mu_{\bullet})=\ti{P}(D(\mu_{\bullet}))$ is indeed 
a convex polytope (but, not necessarily a pseudo-Weyl polytope) 
in $\Fh_{\BR}$, since it is clearly a bounded polyhedral set 
(see \cite[Chapters~I and I\!I]{E}). We set 
$\ti{\pwp}:=
 \bigl\{
  \ti{P}(M_{\bullet}) \mid M_{\bullet} \in \edge^{\omega}
 \bigr\}=
 \bigl\{
  \ti{P}(\mu_{\bullet}) \mid \nu_{\bullet} \in \ggms^{\omega}
 \bigr\}$.

%%%%%%%%%%%%%%%%
%%% rem:tipo %%%
%%%%%%%%%%%%%%%%
%
\begin{rem} \label{rem:tipo}
(1) For each $\mu_{\bullet} \in \ggms^{\omega}$, we have 
$\ti{P}(\mu_{\bullet}) \supset P(\mu_{\bullet})$. 

\noindent
(2) We see from Remark~\ref{rem:omega}\,(5) and 
Lemma~\ref{lem:bzmvda}\,(1) that the set 
$\omega(\ti{P}(\mu_{\bullet}))=
 \bigl\{\omega(h) \mid h \in \ti{P}(\mu_{\bullet})\bigr\}$ 
is identical to $\ti{P}(\mu_{\bullet})$ for all 
$\mu_{\bullet} \in \ggms^{\omega}$. 

\noindent
(3) It follows from \eqref{eq:phi-po} and 
\eqref{eq:tipo01} that 
%
%%%%%%%%%%%%%%%%%
%%% eq:tipo03 %%%
%%%%%%%%%%%%%%%%%
%
\begin{equation} \label{eq:tipo03}
\Phi(P(\mu_{\bullet}))=
\ha{P}(\Phi(\mu_{\bullet}))=
\ti{P}(\mu_{\bullet}) \cap \Fh^{\omega}
\quad
\text{for all $\mu_{\bullet} \in \ggms^{\omega}$}. 
\end{equation}
\end{rem}

From \eqref{eq:tipo03}, we deduce that 
if $\ti{P}(\mu_{\bullet}) = \ti{P}(\mu_{\bullet}')$ 
for $\mu_{\bullet},\,\mu_{\bullet}' \in \ggms^{\omega}$, 
then $\mu_{\bullet} = \mu_{\bullet}'$ since 
$\Phi:\mv^{\omega} \rightarrow \ha{\mv}$ is a bijection. 
Equivalently, 
if $\ti{P}(M_{\bullet}) = \ti{P}(M_{\bullet}')$ for 
$M_{\bullet},\,M_{\bullet}' \in \edge^{\omega}$, 
then $M_{\bullet} = M_{\bullet}'$. 
Thus, by abuse of terminology, we say that 
$\mu_{\bullet} \in \ggms^{\omega}$ 
(resp., $M_{\bullet} \in \edge^{\omega}$) 
is the GGMS (resp., BZ) datum of the convex polytope 
$\ti{P}(\mu_{\bullet})$ (resp., $\ti{P}(M_{\bullet})$). 

Fix $j \in \ha{I}$ and $\ha{P} \in \ha{\mv}$. 
Set $P:=\Phi^{-1}(\ha{P}) \in \mv^{\omega}$. 
Let 
$\mu_{\bullet}=(\mu_{w})_{w \in W} \in \ggms_{\MV}^{\omega}$ and 
$\mu_{\bullet}'=(\mu_{w}')_{w \in W} \in \ggms_{\MV}^{\omega}$ 
be 
the GGMS data of $P$ and $f_{j}^{\omega}P$, respectively, and set 
$M_{\bullet}=
 (M_{\gamma})_{\gamma \in \Gamma}:=D(\mu_{\bullet}) 
 \in \edge_{\MV}^{\omega}$ and 
$M_{\bullet}'=
 (M_{\gamma}')_{\gamma \in \Gamma}:=D(\mu_{\bullet}') 
 \in \edge_{\MV}^{\omega}$.
We define reflections $\sigma:\Fh \rightarrow \Fh$ and 
$\tau:\Fh \rightarrow \Fh$ by: 
\begin{equation*}
\sigma(h)=s_{j} \cdot h+c h_{j}
\qquad \text{and} \qquad
\tau(h)=s_{\omega(j)} \cdot h+c h_{\omega(j)},
\end{equation*}
for $h \in \Fh$, where $c:=c^{M_{\bullet}}_{j}=
M_{\Lambda_{j}}-M_{s_{j} \cdot \Lambda_{j}}-1$; 
note that $c^{M_{\bullet}}_{j}=c^{M_{\bullet}}_{\omega(j)}$, 
and that 
\begin{equation*}
\begin{cases}
\sigma\tau=\tau\sigma & 
 \text{if $1 \le j \le n-1$}, \\[1.5mm]
\sigma=\tau &
 \text{if $\ell=2n-1$, $n \in \BZ_{\ge 2}$, and $j=n$}, \\[1.5mm]
\sigma\tau\sigma=\tau\sigma\tau & 
 \text{if $\ell=2n$, $n \in \BZ_{\ge 2}$, and $j=n$}.
\end{cases}
\end{equation*}
Also, we set 
\begin{equation*}
W^{\omega}_{+}:=
  \bigl\{w \in W^{\omega} \mid 
   s_{j}^{\omega}w > w \bigr\}, \qquad
W^{\omega}_{-}:=
  \bigl\{w \in W^{\omega} \mid 
   s_{j}^{\omega}w < w \bigr\}; 
\end{equation*}
note that $W^{\omega}=W_{+}^{\omega} \cup W_{-}^{\omega}$ 
by Remark~\ref{rem:theta}\,(3). We deduce that 
\begin{align*}
\ha{f}_{j}\ha{P} & 
   =\ha{f}_{j}\Phi(P)
   =\Phi(f_{j}^{\omega}P) 
    \qquad \text{by Theorem~\ref{thm:fixed-mv}} \\
 & =\Phi(P(\mu_{\bullet}'))=\ha{P}(\Phi(\mu_{\bullet}'))
   =\ti{P}(\mu_{\bullet}') \cap \Fh^{\omega}
    \qquad \text{by \eqref{eq:tipo03}}. 
\end{align*}
Thus, it suffices to give a description of 
the convex polytope 
$\ti{P}(\mu_{\bullet}') \subset \Fh_{\BR}$. 

%%%%%%%%%%%%%%%%
%%% thm:pd01 %%%
%%%%%%%%%%%%%%%%
%
\begin{thm} \label{thm:pd01}
Keep the notation above. Assume that 
$\ell=2n-1$, $n \in \BZ_{\ge 2}$, or 
$\ell=2n$, $n \in \BZ_{\ge 2}$, and 
$1 \le j \le n-1$. Then, $\ti{P}(\mu_{\bullet}')$ is 
the smallest convex polytope $\ti{P}$ in $\ti{\pwp}$ 
with GGMS datum $\mu_{\bullet}''=(\mu_{w}'')_{w \in W} \in \ggms^{\omega}$
satisfying the following conditions {\rm (i)-(v)}{\rm:}

\noindent
{\rm (i)} If $w \in W^{\omega}_{-}$, then 
$\mu_{w}''=\mu_{w}$.

\noindent
{\rm (ii)} $\mu_{e}'' = \mu_{e}-h_{j}^{\omega}$.

\noindent
{\rm (iii)} If $w \in W^{\omega}_{+}$, 
then $\mu_{w} \in \ti{P}$.

\noindent
{\rm (iv)} If $w \in W$ is such that 
$s_{j}w < w$ and $\pair{\alpha_{j}}{\mu_{w}} \ge c$, 
then $\sigma(\mu_{w}) \in \ti{P}$. 
Also, if $w \in W$ is such that $s_{\omega(j)}w < w$ and 
$\pair{\alpha_{\omega(j)}}{\mu_{w}} \ge c$, then 
$\tau(\mu_{w}) \in \ti{P}$.

\noindent
{\rm (v)} If $w \in W^{\omega}_{-}$ is such that
$\pair{\alpha_{j}}{\mu_{w}} \ge c$, then $\sigma\tau(\mu_{w}) \in \ti{P}$.
\end{thm}

\begin{proof}
First we prove that the convex polytope 
$\ti{P}(\mu_{\bullet}')$ satisfies 
conditions (i)-(v). 
We see from the proof of Proposition~\ref{prop:com} along with
Remark~\ref{rem:theta}\,(3) that $\mu_{e}'=\mu_{e}-h_{j}^{\omega}$, and 
$\mu_{w}'=\mu_{w}$ for all $w \in W^{\omega}_{-}$, i.e., that 
$\ti{P}(\mu_{\bullet}')$ satisfies conditions (i) and (ii). 
Furthermore, by Remark~\ref{rem:tipo}\,(1) and 
Corollary~\ref{cor:am}, we have 
%
%%%%%%%%%%%%%%%%%
%%% eq:larger %%%
%%%%%%%%%%%%%%%%%
%
\begin{equation} \label{eq:larger}
\ti{P}(\mu_{\bullet}') \supset P(\mu_{\bullet}')=f_{j}^{\omega}P=
f_{j}f_{\omega(j)}P \supset f_{\omega(j)}P. 
\end{equation}
Also, we see from Remark~\ref{rem:theta}\,(3) that 
if $w \in W_{+}^{\omega}$, then $s_{\omega(j)}w > w$. 
Hence, by Theorem~\ref{thm:pam}, 
$f_{\omega(j)}P$ contains $\mu_{w}$ 
for all $w \in W^{\omega}_{+}$. 
Therefore, it follows from \eqref{eq:larger} that 
$\ti{P}(\mu_{\bullet}')$ contains $\mu_{w}$ for all 
$w \in W^{\omega}_{+}$, i.e., that 
$\ti{P}(\mu_{\bullet}')$ satisfies condition (iii). 
If $w \in W$ is such that $s_{\omega(j)}w < w$ and 
$\pair{\alpha_{\omega(j)}}{\mu_{w}} \ge c$, then 
by Theorem~\ref{thm:pam}, 
$\tau(\mu_{w}) \in f_{\omega(j)}P$. 
Therefore, $\tau(\mu_{w}) \in \ti{P}(\mu_{\bullet}')$ 
again by \eqref{eq:larger}. 
Similarly, using the equation 
$f_{j}^{\omega}=f_{j}f_{\omega(j)}=f_{\omega(j)}f_{j}$ 
(see Remark~\ref{rem:ok}), we can show that 
if $w \in W$ is such that $s_{j}w < w$ and 
$\pair{\alpha_{j}}{\mu_{w}} \ge c$, then 
$\sigma(\mu_{w}) \in \ti{P}(\mu_{\bullet}')$. 
Namely, we have shown that 
$\ti{P}(\mu_{\bullet}')$ satisfies condition (iv). 
It remains to show that $\ti{P}(\mu_{\bullet}')$ satisfies 
condition (v). Let $w \in W^{\omega}_{-}$ be such that
$\pair{\alpha_{j}}{\mu_{w}} \ge c$. 
By \eqref{eq:tipo01}, 
it suffices to show that 
%
%%%%%%%%%%%%%
%%% eq:s0 %%%
%%%%%%%%%%%%%
%
\begin{equation} \label{eq:s0}
\sigma\tau(\mu_{w}) \ge_{v} \mu_{v}'
\qquad \text{for all $v \in W^{\omega}$}.
\end{equation}
%
%%%%%%%%%%%
%%% c01 %%%
%%%%%%%%%%%
%
\begin{claim} \label{c01}
If $v \in W^{\omega}_{-}$, then 
$\sigma\tau(\mu_{v}) 
 \ge_{s_{j}^{\omega}v} 
 \mu_{s_{j}^{\omega}v}'$. 
\end{claim}

We set $\gamma_{i}:=s_{j}^{\omega}v \cdot \Lambda_{i}$ for $i \in I$. 
Since $v \in W^{\omega}_{-}$, and hence 
$s_{j}^{\omega}v \in W^{\omega}_{+}$, 
it follows from Remark~\ref{rem:theta}\,(3) and 
\cite[Proposition~4\,(i) in \S5.2]{MP} that 
$(s_{j}^{\omega}v)^{-1} \cdot h_{j}$ and 
$(s_{j}^{\omega}v)^{-1} \cdot h_{\omega(j)}$ are 
positive coroots of $\Fg$. 
Therefore, we have 
%
%%%%%%%%%%%%%%
%%% eq:s01 %%%
%%%%%%%%%%%%%%
%
\begin{equation} \label{eq:s01}
\begin{cases}
\pair{\gamma_{i}}{h_{j}}=
\pair{\Lambda_{i}}{(s_{j}^{\omega}v)^{-1} \cdot h_{j}} \ge 0, & \\[3mm]
\pair{\gamma_{i}}{h_{\omega(j)}}=
\pair{\Lambda_{i}}{(s_{j}^{\omega}v)^{-1} \cdot h_{\omega(j)}} \ge 0. &
\end{cases}
\end{equation}
Hence, by \eqref{eq:minu}, we deduce that
%
%%%%%%%%%%%%%%
%%% eq:s02 %%%
%%%%%%%%%%%%%%
%
\begin{equation} \label{eq:s02}
\bigl(
 \pair{\gamma_{i}}{h_{j}},\,
 \pair{\gamma_{i}}{h_{\omega(j)}}
\bigr)=
(0,0),\,(0,1),\,(1,0), \text{ or } (1,1).
\end{equation}
Now, recall that 
$\sigma\tau(\mu_{v}) \ge_{s_{j}^{\omega}v} 
 \mu_{s_{j}^{\omega}v}'$ if and only if 
\begin{equation*}
\pair{\gamma_{i}}{\sigma\tau(\mu_{v})} \ge 
\pair{\gamma_{i}}{\mu_{s_{j}^{\omega}v}'}=
M_{\gamma_{i}}' \quad \text{for all $i \in I$.}
\end{equation*}
Also, by direct calculation, we obtain 
\begin{align*}
\pair{\gamma_{i}}{\sigma\tau(\mu_{v})}
 & = \pair{\gamma_{i}}{s_{j}^{\omega} \cdot \mu_{v}+ch_{j}^{\omega}}
   = \pair{\gamma_{i}}{s_{j}^{\omega} \cdot \mu_{v}}+
     c \pair{\gamma_{i}}{h_{j}^{\omega}} \\
 & = \pair{v\Lambda_{i}}{\mu_{v}}+c \pair{\gamma_{i}}{h_{j}^{\omega}}
   = M_{v \cdot \Lambda_{i}}+c \pair{\gamma_{i}}{h_{j}^{\omega}} \\
 & = M_{s_{j}^{\omega} \cdot \gamma_{i}}+c \pair{\gamma_{i}}{h_{j}^{\omega}}. 
\end{align*}
If $\pair{\gamma_{i}}{h_{j}} = \pair{\gamma_{i}}{h_{\omega(j)}} = 1$, 
then $M_{s_{j}^{\omega} \cdot \gamma_{i}}+c \pair{\gamma_{i}}{h_{j}^{\omega}} = 
M_{s_{j}^{\omega} \cdot \gamma_{i}}+2c$. 
Therefore, in this case, we deduce from 
Proposition~\ref{prop:ok}\,(1) that 
$\pair{\gamma_{i}}{\sigma\tau(\mu_{v})}= 
 M_{s_{j}^{\omega} \cdot \gamma_{i}}+2c \ge M_{\gamma_{i}}'$. 
Similarly, we can show that 
$\pair{\gamma_{i}}{\sigma\tau(\mu_{v})} \ge M_{\gamma_{i}}'$ 
in all other cases of \eqref{eq:s02}. This proves Claim~\ref{c01}. 
%
%%%%%%%%%%%
%%% c02 %%%
%%%%%%%%%%%
%
\begin{claim} \label{c02} 
Inequality \eqref{eq:s0} holds 
for all $v \in W^{\omega}_{+}$. 
\end{claim}

Since $v \in W^{\omega}_{+}$, and hence 
$s_{j}^{\omega}v \in W^{\omega}_{-}$, 
it follows from Claim~\ref{c01} that
%
%%%%%%%%%%%%%%%
%%% eq:s2-1 %%%
%%%%%%%%%%%%%%%
%
\begin{equation} \label{eq:s2-1}
\sigma\tau(\mu_{s_{j}^{\omega}v}) \ge_{v} \mu_{v}'.
\end{equation}
Also, since $\mu_{\bullet}=(\mu_{w})_{w \in W} \in \ggms$, 
it follows that $\mu_{w} \ge_{s_{j}^{\omega}v} \mu_{s_{j}^{\omega}v}$, 
from which we deduce by direct calculation that 
$\tau(\mu_{w}) \ge_{s_{j}v} \tau(\mu_{s_{j}^{\omega}v})$, 
and then that 
$\sigma\tau(\mu_{w}) \ge_{v} \sigma\tau(\mu_{s_{j}^{\omega}v})$.
Combining the last inequality with \eqref{eq:s2-1}, 
we get $\sigma\tau(\mu_{w}) \ge_{v} \mu_{v}'$, as desired. 
This proves Claim~\ref{c02}. 
%
%%%%%%%%%%%
%%% c03 %%%
%%%%%%%%%%%
%
\begin{claim} \label{c03} 
Inequality \eqref{eq:s0} holds 
for all $v \in W^{\omega}_{-}$. 
\end{claim}

Since $w \in W^{\omega}_{-} \subset W^{\omega}$, 
it follows from Lemma~\ref{lem:bzmvda}\,(1) that 
$\omega(\mu_{w})=\mu_{w}$, and hence 
$\pair{\alpha_{j}}{\mu_{w}}=\pair{\alpha_{\omega(j)}}{\mu_{w}}$. 
Using this, we have
%
%%%%%%%%%%%%%%%
%%% eq:s3-0 %%%
%%%%%%%%%%%%%%%
%
\begin{align}
\sigma\tau(\mu_{w})-\mu_{w}
 & = (s_{j}^{\omega} \cdot \mu_{w}+c h_{j}^{\omega})-\mu_{w} \nonumber \\
 & = \mu_{w}-
     \pair{\alpha_{j}}{\mu_{w}}h_{j}-
     \pair{\alpha_{\omega(j)}}{\mu_{w}}h_{\omega(j)}+
     c h_{j}^{\omega}- \mu_{w} \nonumber \\
 & = \bigl(c-\pair{\alpha_{j}}{\mu_{w}}\bigr) h_{j}^{\omega}.  \label{eq:s3-0}
\end{align}
Since $v \in W^{\omega}_{-}$, 
it follows from Remark~\ref{rem:theta}\,(3) and 
\cite[Proposition~4\,(i) in \S5.2]{MP} that 
$v^{-1}(h_{j}^{\omega})$ is 
a negative coroot of $\Fg^{\omega}$, 
and hence $h_{j}^{\omega} \le_{v} 0$. 
But, since $c-\pair{\alpha_{j}}{\mu_{w}} \le 0$ by assumption, 
we see from \eqref{eq:s3-0} that 
%
%%%%%%%%%%%%%%%
%%% eq:s3-1 %%%
%%%%%%%%%%%%%%%
%
\begin{equation} \label{eq:s3-1}
\sigma\tau(\mu_{w}) \ge_{v} \mu_{w}.
\end{equation} 
Also, it follows that $\mu_{w} \ge_{v} \mu_{v}$ 
since $\mu_{\bullet}=(\mu_{w})_{w \in W} \in \ggms$, and 
that $\mu_{v}'=\mu_{v}$ since $v \in W^{\omega}_{-}$ and 
$\ti{P}(\mu_{\bullet}')$ satisfies condition (i) as shown above. 
Combining these facts with \eqref{eq:s3-1}, we deduce that 
$\sigma\tau(\mu_{w}) \ge_{v} \mu_{w} \ge_{v} 
\mu_{v}=\mu_{v}'$. This proves Claim~\ref{c03}.

\vsp

By Claims~\ref{c02} and \ref{c03}, 
inequality \eqref{eq:s0} holds 
for all $v \in W^{\omega}=W^{\omega}_{+} \cup W^{\omega}_{-}$, that is, 
$\ti{P}(\mu_{\bullet}')$ satisfies condition (v). 
Thus we have proved that the convex polytope $\ti{P}(\mu_{\bullet}')$ 
satisfies conditions (i)-(v). 

%%%%
\vsp
%%%%

Next, we prove that 
if $\ti{P}'' \in \ti{\pwp}$ satisfies 
conditions (i)-(v), then $\ti{P}''$ must contain 
$\ti{P}(\mu_{\bullet}')$. 
Let $\ti{P}''=\ti{P}(\mu_{\bullet}'') \in \ti{\pwp}$ 
be a convex polytope with GGMS datum $\mu_{\bullet}''=
(\mu_{w}'')_{w \in W} \in \ggms^{\omega}$ 
satisfying conditions (i)-(v), 
and set $(M_{\gamma}'')_{\gamma \in \Gamma}:=
D(\mu_{\bullet}'') \in \edge^{\omega}$.
In order to prove that $\ti{P}'' \supset \ti{P}(\mu_{\bullet}')$, 
it suffices to show that $M_{\gamma}' \ge M_{\gamma}''$
for all $\gamma \in \ti{\Gamma}$ (see \eqref{eq:tipo02}). 
%
%%%%%%%%%%%
%%% c04 %%%
%%%%%%%%%%%
%
\begin{claim} \label{c04} 
The inequality $M_{\gamma} \ge M_{\gamma}''$ holds 
for all $\gamma \in \ti{\Gamma}$.
\end{claim}

Since $\ti{P}''=\ti{P}(\mu_{\bullet}'')$ 
satisfies conditions (i) and (iii), 
it follows that $\mu_{w} \in \ti{P}''$ 
for all $w \in W^{\omega}=W^{\omega}_{+} \cup W^{\omega}_{-}$. 
Hence, by \eqref{eq:tipo01}, we have $\mu_{w} \ge_{w} \mu_{w}''$ 
for all $w \in W^{\omega}$, which implies that 
$M_{\gamma} \ge M_{\gamma}''$ for all $\gamma \in \ti{\Gamma}$, 
as desired. This proves Claim~\ref{c04}. 
%
%%%%%%%%%%%
%%% c05 %%%
%%%%%%%%%%%
%
\begin{claim} \label{c05} 
Let $\gamma \in \ti{\Gamma}$ be such that $\pair{\gamma}{h_{j}}=1$. 
Then, we have $M_{s_{j} \cdot \gamma}+c \ge M_{\gamma}''$.
\end{claim}

Write the $\gamma \in \ti{\Gamma}$ in the form 
$\gamma=s_{j}w \cdot \Lambda_{i}$, with $w \in W$ and $i \in I$. 
Since $\pair{\gamma}{h_{j}}=1 > 0$, 
it follows from \cite[Proposition~4\,(i) in \S5.2]{MP} that $s_{j}w < w$. 
Also, we have
%
%%%%%%%%%%%%%%%
%%% eq:s5-1 %%%
%%%%%%%%%%%%%%%
%
\begin{align}
\pair{\gamma}{\sigma(\mu_{w})}
& = \pair{\gamma}{s_{j} \cdot \mu_{w}+c h_{j}}
  = \pair{\gamma}{s_{j} \cdot \mu_{w}}+c \pair{\gamma}{h_{j}}
  = \pair{s_{j} \cdot \gamma}{\mu_{w}}+c \nonumber \\ 
& = \pair{w \cdot \Lambda_{i}}{\mu_{w}}+c 
  =M_{w \cdot \Lambda_{i}}+c
  =M_{s_{j} \cdot \gamma}+c. \label{eq:s5-1}
\end{align}
Assume first that $\pair{\alpha_{j}}{\mu_{w}} \ge c$. 
Since $s_{j}w < w$ as seen above, we have 
$\sigma(\mu_{w}) \in \ti{P}''$ by condition (iv).
Therefore, $\pair{\gamma}{\sigma(\mu_{w})} \ge M_{\gamma}''$
(see \eqref{eq:tipo02}), and hence 
$M_{s_{j} \cdot \gamma}+c \ge M_{\gamma}''$ by \eqref{eq:s5-1}.
Assume next that $\pair{\alpha_{j}}{\mu_{w}} < c$. 
Recall that $\mu_{w} \ge_{s_{j}w} \mu_{s_{j}w}$ since 
$\mu_{\bullet}=(\mu_{w})_{w \in W} \in \ggms$. 
Hence it follows that 
%
%%%%%%%%%%%%%%%
%%% eq:s5-2 %%%
%%%%%%%%%%%%%%%
%
\begin{equation} \label{eq:s5-2}
\pair{\gamma}{\mu_{w}} \ge 
\pair{\gamma}{\mu_{s_{j}w}}=M_{\gamma}.
\end{equation}
Because $\pair{\gamma}{h_{j}}=1$ and 
$c-\pair{\alpha_{j}}{\mu_{w}} > 0$ by 
assumption, we have
\begin{align*}
\pair{\gamma}{\sigma(\mu_{w})} 
 & = \pair{\gamma}{s_{j} \cdot \mu_{w}+ch_{j}} \\
 & = \pair{\gamma}{\mu_{w}}+
     \bigl(\underbrace{c-\pair{\alpha_{j}}{\mu_{w}}}_{> 0}\bigr)
     \underbrace{\pair{\gamma}{h_{j}}}_{=1} > 
     \pair{\gamma}{\mu_{w}} \\[1.5mm]
 & \ge M_{\gamma} \quad \text{by \eqref{eq:s5-2}} \\
 & \ge M_{\gamma}'' \quad \text{by Claim~\ref{c04}}.
\end{align*}
Combining this with \eqref{eq:s5-1}, 
we obtain $M_{s_{j}\gamma}+c \ge M_{\gamma}''$.
This proves Claim~\ref{c05}. 
%
%%%%%%%%%%%
%%% c06 %%%
%%%%%%%%%%%
%
\begin{claim} \label{c06}
Let $\gamma \in \ti{\Gamma}$ be such that 
$\pair{\gamma}{h_{\omega(j)}}=1$. Then, we have
$M_{s_{\omega(j)} \cdot \gamma}+c \ge M_{\gamma}''$.
\end{claim}

Note that $\omega(\gamma) \in \ti{\Gamma}$ 
by Remark~\ref{rem:omega}\,(4), and that 
$\pair{\omega(\gamma)}{h_{j}}=\pair{\gamma}{h_{\omega(j)}}=1$. 
Hence, by Claim~\ref{c05}, 
$M_{s_{j} \cdot \omega(\gamma)}+c \ge M_{\omega(\gamma)}''$.
Since $M_{\bullet} \in \edge_{\MV}^{\omega} \subset \edge^{\omega}$ 
and $M_{\bullet}'' \in \edge^{\omega}$, it follows from 
Lemma~\ref{lem:bzmvda}\,(2) that 
$M_{s_{j} \cdot \omega(\gamma)}=M_{s_{\omega(j)} \cdot \gamma}$ and 
$M_{\omega(\gamma)}''=M_{\gamma}''$. Therefore, we obtain 
$M_{s_{\omega(j)} \cdot \gamma}+c \ge M_{\gamma}''$, as desired. 
This proves Claim~\ref{c06}. 

%%%%%%%%%%%
%%% c07 %%%
%%%%%%%%%%%
%
\begin{claim} \label{c07}
Let $\gamma \in \ti{\Gamma}$ be such that 
$\pair{\gamma}{h_{j}}=\pair{\gamma}{h_{\omega(j)}}=1$. 
Then, $M_{s_{j}^{\omega} \cdot \gamma}+2c \ge M_{\gamma}''$.
\end{claim}

Write the $\gamma \in \ti{\Gamma}$ in the form 
$\gamma=s_{j}^{\omega}w \cdot \Lambda_{i}$, 
with $w \in W^{\omega}$ and $i \in I$. 
Since $\pair{\gamma}{h_{j}}=\pair{\gamma}{h_{\omega(j)}}=1$, 
it follows from \cite[Proposition~4\,(i) in \S5.2]{MP} that 
$s_{j}w < w$ and $s_{\omega(j)}w < w$. 
Hence, by Remark~\ref{rem:theta}\,(3), 
$w \in W^{\omega}_{-}$. Also, we have
%
%%%%%%%%%%%%%%%
%%% eq:s7-1 %%%
%%%%%%%%%%%%%%%
%
\begin{align}
\pair{\gamma}{\sigma\tau(\mu_{w})} 
 & =\pair{\gamma}{s_{j}^{\omega} \cdot \mu_{w}}+
    c \pair{\gamma}{h_{j}^{\omega}} 
   = \pair{s_{j}^{\omega} \cdot \gamma}{\mu_{w}}+2c \nonumber \\
 & = \pair{w \cdot \Lambda_{i}}{\mu_{w}}+2c=M_{w \cdot \Lambda_{i}}+2c
   = M_{s_{j}^{\omega} \cdot \gamma}+2c. \label{eq:s7-1}
\end{align}
Assume first that $\pair{\alpha_{j}}{\mu_{w}} \ge c$. 
Since $w \in W^{\omega}_{-}$ as shown above, 
we have $\sigma\tau(\mu_{w}) \in \ti{P}''$ 
by condition (v). Therefore, 
$\pair{\gamma}{\sigma\tau(\mu_{w})} \ge M_{\gamma}''$ 
(see \eqref{eq:tipo02}), and hence 
$M_{s_{j}^{\omega} \cdot \gamma}+2c \ge M_{\gamma}''$
by \eqref{eq:s7-1}.
Assume next that $\pair{\alpha_{j}}{\mu_{w}} < c$. 
Recall that $\mu_{w} \ge_{s_{j}^{\omega}w} \mu_{s_{j}^{\omega}w}$ 
since $\mu_{\bullet}=(\mu_{w})_{w \in W} \in \ggms$. 
Hence it follows that 
%
%%%%%%%%%%%%%%%
%%% eq:s7-2 %%%
%%%%%%%%%%%%%%%
%
\begin{equation} \label{eq:s7-2}
\pair{\gamma}{\mu_{w}} \ge 
\pair{\gamma}{\mu_{s_{j}^{\omega}w}}=M_{\gamma}.
\end{equation}
Note that since $w \in W^{\omega}_{-} \subset W^{\omega}$, 
we have $\omega(\mu_{w})=\mu_{w}$ by Lemma~\ref{lem:bzmvda}\,(1), 
and hence $\pair{\alpha_{j}}{\mu_{w}}=
 \pair{\alpha_{\omega(j)}}{\mu_{w}}$.
From the assumptions that 
$\pair{\gamma}{h_{j}}=\pair{\gamma}{h_{\omega(j)}}=1$ and 
$c-\pair{\alpha_{j}}{\mu_{w}} > 0$, using the equality 
$\pair{\alpha_{j}}{\mu_{w}}=
 \pair{\alpha_{\omega(j)}}{\mu_{w}}$, 
we have
\begin{align*}
\pair{\gamma}{\sigma\tau(\mu_{w})} 
 & = \pair{\gamma}{s_{j}^{\omega} \cdot \mu_{w}+ch_{j}^{\omega}} \\
 & = \pair{\gamma}{
        \mu_{w}-\pair{\alpha_{j}}{\mu_{w}}h_{j}-
        \pair{\alpha_{\omega(j)}}{\mu_{w}}h_{\omega(j)}+
        ch_{j}^{\omega}} \\
 & = \pair{\gamma}{
        \mu_{w}-\pair{\alpha_{j}}{\mu_{w}}h_{j}^{\omega}+
        ch_{j}^{\omega}} \\
 & = \pair{\gamma}{\mu_{w}}+
     \bigl(\underbrace{c- \pair{\alpha_{j}}{\mu_{w}}}_{> 0}\bigr)
     \underbrace{\pair{\gamma}{h_{j}^{\omega}}}_{=2} 
   > \pair{\gamma}{\mu_{w}} \\[1.5mm]
 & \ge M_{\gamma} \quad \text{by \eqref{eq:s7-2}} \\
 & \ge M_{\gamma}'' \quad \text{by Claim~\ref{c04}}.
\end{align*}
Combining this inequality with \eqref{eq:s7-1}, 
we obtain $M_{s_{j}^{\omega} \cdot \gamma}+2c \ge M_{\gamma}''$.
This proves Claim~\ref{c07}. 

\vsp

By using Claims~\ref{c04}-\ref{c07}, 
we can show that $M_{\gamma}' \ge M_{\gamma}''$ 
for all $\gamma \in \ti{\Gamma}$. 
As an example, let us consider the case in which 
$\gamma(h_{j})=1$ and $\gamma(h_{\omega(j)}) \le 0$. Then, 
\begin{align*}
M_{\gamma}' 
& = 
   \min \bigl(M_{\gamma},\ M_{s_{j}\gamma}+c \bigr)
  \quad \text{by Proposition~\ref{prop:ok}\,(1)} \\
& \ge M_{\gamma}'' 
  \quad \text{by Claims~\ref{c04} and \ref{c05}}.
\end{align*}
The proofs for the other cases are similar.
Thus we have proved that $\ti{P}'' \supset \ti{P}(\mu_{\bullet}')$. 
This completes the proof of Theorem~\ref{thm:pd01}. 
\end{proof}

The proof of the following theorem is similar to 
(and even simpler than) that of Theorem~\ref{thm:pd01}. 
%
%%%%%%%%%%%%%%%%
%%% thm:pd02 %%%
%%%%%%%%%%%%%%%%
%
\begin{thm} \label{thm:pd02}
Keep the notation above. Assume that 
$\ell=2n-1$, $n \in \BZ_{\ge 2}$, and $j=n$. 
Then, the convex polytope $\ti{P}(\mu_{\bullet}')$ is 
the smallest convex polytope $\ti{P}$ in $\ti{\pwp}$ 
with GGMS datum $\mu_{\bullet}''=(\mu_{w}'')_{w \in W} \in \ggms^{\omega}$ 
such that 

\noindent
{\rm (i)} $\mu_{w}''=\mu_{w}$ for all $w \in W^{\omega}_{-}$, 

\noindent
{\rm (ii)} $\mu_{e}'' = \mu_{e}-h_{n}^{\omega}$, 

\noindent
{\rm (iii)} $\ti{P}$ contains $\mu_{w}$ for all $w \in W^{\omega}_{+}$, and 

\noindent
{\rm (iv)} if $w \in W^{\omega}_{-}$ is such that
$\pair{\alpha_{j}}{\mu_{w}} \ge c$, then 
$\ti{P}$ contains $\sigma(\mu_{w})$. 
\end{thm}
The proof of the following theorem is similar to 
(but, a little more complicated than) 
that of Theorem~\ref{thm:pd01}; we leave it to the reader. 
%
%%%%%%%%%%%%%%%%
%%% thm:pd03 %%%
%%%%%%%%%%%%%%%%
%
\begin{thm} \label{thm:pd03}
Keep the notation above. Assume that 
$\ell=2n$, $n \in \BZ_{\ge 2}$, and $j=n$. 
Then, the convex polytope $\ti{P}(\mu_{\bullet}')$ is 
the smallest convex polytope $\ti{P}$ in $\ti{\pwp}$
with GGMS datum $\mu_{\bullet}''=(\mu_{w}'')_{w \in W} \in \ggms^{\omega}$ 
satisfying the following conditions {\rm (i)-(v)}{\rm:}

\noindent
{\rm (i)} If $w \in W^{\omega}_{-}$, then 
$\mu_{w}''=\mu_{w}$.

\noindent
{\rm (ii)} $\mu_{e}'' = \mu_{e} - h_{n}^{\omega}$.

\noindent
{\rm (iii)} If $w \in W^{\omega}_{+}$, 
then $\mu_{w} \in \ti{P}$.

\noindent
{\rm (iv)} If $w \in W$ is such that 
$s_{j}w < w$ and $\pair{\alpha_{j}}{\mu_{w}} \ge c$, 
then $\sigma(\mu_{w}) \in \ti{P}$. 
Also, if $w \in W$ is such that $s_{\omega(j)}w < w$ and 
$\pair{\alpha_{\omega(j)}}{\mu_{w}} \ge c$, then 
$\tau(\mu_{w}) \in \ti{P}$.

\noindent
{\rm (v)} If $w \in W^{\omega}_{-}$ is such that
$\pair{\alpha_{j}}{\mu_{w}} \ge c$, then 
$\sigma\tau\sigma(\mu_{w}) \in \ti{P}$. 
\end{thm}

%======================%
%     BIBLIOGRAPHY     %
%======================%

{\small
\setlength{\baselineskip}{13pt}
\renewcommand{\refname}{References}

}

\end{document}